\documentclass[a4paper]{scrartcl}
\usepackage{amsmath,amsfonts,amssymb,amsopn}
\usepackage{graphicx}
\usepackage{epstopdf}
\usepackage{tensor}
\usepackage{colonequals}
\usepackage{verbatim}
\usepackage{xargs}
\usepackage{mathtools}
\usepackage{pgfplots}
\pgfplotsset{compat=newest}
\usepackage{tikz}
\usepackage{subcaption}
\usepackage{bm}

\ifpdf
  \DeclareGraphicsExtensions{.eps,.pdf,.png,.jpg}
\else
  \DeclareGraphicsExtensions{.eps}
\fi

\newcommand{\logLogSlopeTriangleB}[5]
{

    \pgfplotsextra
    {
        \pgfkeysgetvalue{/pgfplots/xmin}{\xmin}
        \pgfkeysgetvalue{/pgfplots/xmax}{\xmax}
        \pgfkeysgetvalue{/pgfplots/ymin}{\ymin}
        \pgfkeysgetvalue{/pgfplots/ymax}{\ymax}

\pgfmathsetmacro{\xArel}{#1}
        \pgfmathsetmacro{\yArel}{#3}
        \pgfmathsetmacro{\xBrel}{#1-#2}
        \pgfmathsetmacro{\yBrel}{\yArel}
        \pgfmathsetmacro{\xCrel}{\xArel}

        \pgfmathsetmacro{\lnxB}{\xmin*(1-(#1-#2))+\xmax*(#1-#2)} \pgfmathsetmacro{\lnxA}{\xmin*(1-#1)+\xmax*#1} \pgfmathsetmacro{\lnyA}{\ymin*(1-#3)+\ymax*#3} \pgfmathsetmacro{\lnyC}{\lnyA+#4*(\lnxA-\lnxB)}
        \pgfmathsetmacro{\yCrel}{\lnyC-\ymin)/(\ymax-\ymin)}

\coordinate (A) at (rel axis cs:\xArel,\yArel);
        \coordinate (B) at (rel axis cs:\xBrel,\yBrel);
        \coordinate (C) at (rel axis cs:\xCrel,\yCrel);

\draw[#5]   (A)-- node[pos=0.5,anchor=north] {}
                    (B)--
                    (C)-- node[pos=0.5,anchor=west] {\footnotesize{#4}}
                    cycle;
    }
}

\newcommand{\logLogSlopeTriangleT}[5]
{

    \pgfplotsextra
    {
        \pgfkeysgetvalue{/pgfplots/xmin}{\xmin}
        \pgfkeysgetvalue{/pgfplots/xmax}{\xmax}
        \pgfkeysgetvalue{/pgfplots/ymin}{\ymin}
        \pgfkeysgetvalue{/pgfplots/ymax}{\ymax}

\pgfmathsetmacro{\xArel}{#1}
        \pgfmathsetmacro{\yArel}{#3}
        \pgfmathsetmacro{\xBrel}{#1-#2}
        \pgfmathsetmacro{\yBrel}{\yArel}
        \pgfmathsetmacro{\xCrel}{\xArel}

        \pgfmathsetmacro{\lnxB}{\xmin*(1-(#1-#2))+\xmax*(#1-#2)} \pgfmathsetmacro{\lnxA}{\xmin*(1-#1)+\xmax*#1} \pgfmathsetmacro{\lnyA}{\ymin*(1-#3)+\ymax*#3} \pgfmathsetmacro{\lnyC}{\lnyA+#4*(\lnxA-\lnxB)}
        \pgfmathsetmacro{\yCrel}{\lnyC-\ymin)/(\ymax-\ymin)}

        \pgfmathsetmacro{\xDrel}{\xBrel}
        \pgfmathsetmacro{\yDrel}{\yCrel}

\coordinate (A) at (rel axis cs:\xArel,\yArel);
        \coordinate (B) at (rel axis cs:\xBrel,\yBrel);
        \coordinate (C) at (rel axis cs:\xCrel,\yCrel);
        \coordinate (D) at (rel axis cs:\xDrel,\yDrel);

\draw[#5]   (C)-- node[pos=0.5,anchor=south] { }
                    (D)-- node[pos=0.5,anchor=east] {\footnotesize{#4}}
                    (B)--
                    cycle;
    }
}
 \newcommand{\myTitle}{Tangential Errors of Tensor Surface Finite Elements}

\newcommand{\myAbstract}{We discretize a tangential tensor field equation using a surface-finite element approach with a penalization term to ensure almost tangentiality. It is natural to measure the quality of such a discretization intrinsically, i.e., to examine the tangential convergence behavior in contrast to the normal behavior. We show optimal order convergence with respect to the tangential quantities in particular for an isogeometric penalization term that is based only on the geometric information of the discrete surface.}
\newcommand{\myKeywords}{tensor field approximation on surfaces; surface finite elements; a priori error estimates}

\newcommand{\EN}[2]{#2{}} \usepackage[numbers]{natbib}
\usepackage{amsthm}
\usepackage{hyperref}
\usepackage[capitalize,nameinlink]{cleveref}
\usepackage{xcolor}
\hypersetup{
    colorlinks,
    allcolors={green!50!black},
    urlcolor={red!50!black}
}

\theoremstyle{plain}
\newtheorem{theorem}{Theorem}[section]
\newtheorem{corollary}{Corollary}[section]
\newtheorem{lemma}{Lemma}[section]

\theoremstyle{plain}
\newtheorem{remark}{Remark}[section]
\newtheorem{definition}{Definition}[section]

\newcommand{\myBold}[1]{\bm{#1}}
\def\isimastyle{0}

\def\R{\mathbb{R}}

\DeclareMathOperator{\Id}{Id}

\newcommand{\ul}{\underline}

\newcommand{\G}{\Gamma}
\newcommand{\Gh}{\Gamma_h}

\newcommand{\K}{\mathcal{K}}
\newcommand{\Kh}{\K_h}

\newcommand{\leqC}{\lesssim}

\newcommandx{\TangentBundle}[2][1={}]{\mathrm{T}_{#1}{#2}}
\newcommandx{\DualTangentBundle}[2][1={}]{\mathrm{T}^\star_{#1}{#2}}

\newcommandx{\TensorBundle}[3][1={n},2={}]{\mathrm{T}^{#1}_{#2}{#3}}
\newcommandx{\TensorBundleTan}[2][1={n},2={}]{\TensorBundle[#1][#2]{\Gamma}}
\newcommandx{\TensorBundleAmb}[2][1={n},2={}]{\TensorBundle[#1][#2]{\R^{d+1}}}

\newcommandx{\TensorField}[3][1={n},2={}]{\mathcal{T}^{#1}_{#2}{#3}}
\newcommandx{\TensorFieldTan}[2][1={n},2={}]{\TensorField[#1][#2]{\Gamma}}
\newcommandx{\TensorFieldAmb}[2][1={n},2={}]{\TensorField[#1][#2]{\R^{d+1}}}

\newcommand{\VV}[1]{\uppercase{#1}}

\newcommand{\TT}[1]{\myBold{#1}}

\newcommand{\kg}{k_g}
\newcommand{\kp}{k_p}
\newcommand{\ku}{k_u}

\newcommand{\Weingarten}{\mathcal{W}}

\newcommandx{\CovariantDerivative}[1]{\myBold{\nabla}_{\!#1}}
\newcommand{\CDer}{\CovariantDerivative{\Gamma}}

\newcommand{\ChDer}{\CovariantDerivative{\Gamma_h}}
\newcommand{\DDer}{\mathcal{D}}
\newcommand{\DD}{\DDer}

\newcommandx{\Tensor}[3][2=r,3={}]{{#1}^{#2}_{#3}}
\newcommandx{\TensorVar}[3][2=r,3={}]{{#1}^{(#2)}}
\newcommandx{\TensorSet}[4][3=r,4={}]{\Tensor{#1}[#3][#4]({#2})}

\newcommand{\Inner}[3]{\left({#1} , {#2}\right)_{#3}}
\newcommand{\Abs}[1]{\left\lvert{#1}\right\rvert}
\newcommand{\Norm}[2]{\lVert{#1}\rVert_{#2}}

\newcommand{\Set}[1]{\big\{{#1}\big\}}

\newcommand{\EnergyNorm}[2]{{\lvert\kern-0.25ex\lvert\kern-0.25ex\lvert{#1}\rvert\kern-0.25ex\rvert\kern-0.25ex\rvert}_{#2}}

\newcommand{\GNorm}[1]{\EnergyNorm{#1}{\star}}
\newcommand{\FNorm}[1]{\lVert{#1}\rVert}
\newcommand{\FInner}[2]{\langle{#1} , {#2}\rangle}

\newcommandx{\SobolevSpace}[3][1=s,2=p]{\myBold{W}^{#1,#2}({#3})}
\newcommandx{\SobolevSpaceTan}[3][1=s,2=p]{\myBold{W}_\text{tan}^{#1,#2}({#3})}
\newcommandx{\SobolevSpaceAmb}[3][1=s,2=p]{\myBold{W}^{#1,#2}({#3})}
\newcommandx{\SobolevNormTan}[4][1=s,2=p]{\Norm{#3}{\SobolevSpaceTan[#1][#2]{#4}}}
\newcommandx{\SobolevNormAmb}[4][1=s,2=p]{\Norm{#3}{\SobolevSpaceAmb[#1][#2]{#4}}}
\newcommandx{\SobolevNorm}[4][1=s,2=p]{\Norm{#3}{\SobolevSpace[#1][#2]{#4}}}

\newcommandx{\HSpace}[2][1={1}]{\myBold{H}^{#1}({#2})}
\newcommandx{\HSpaceTan}[2][1={1}]{\myBold{H}_\text{tan}^{#1}({#2})}
\newcommandx{\HSpaceAmb}[2][1={1}]{\myBold{H}^{#1}({#2})}
\newcommandx{\HNormTan}[3][1={1}]{\Norm{#2}{\HSpaceTan[#1]{#3}}}
\newcommandx{\HNormAmb}[3][1={1}]{\Norm{#2}{\HSpaceAmb[#1]{#3}}}
\newcommandx{\HNorm}[3][1={1}]{\Norm{#2}{\HSpace[#1]{#3}}}

\newcommandx{\LSpace}[2][1={2}]{\myBold{L}^{#1}(#2)}
\newcommandx{\LSpaceTan}[2][1={2}]{\myBold{L}_\text{tan}^{#1}(#2)}
\newcommandx{\LSpaceAmb}[2][1={2}]{\myBold{L}^{#1}(#2)}
\newcommandx{\LNormTan}[3][1={2}]{\Norm{#2}{\LSpaceTan[#1]{#3}}}
\newcommandx{\LNormAmb}[3][1={2}]{\Norm{#2}{\LSpaceAmb[#1]{#3}}}
\newcommandx{\LNorm}[3][1={2}]{\Norm{#2}{\LSpace[#1]{#3}}}

\newcommand{\LInfSpace}[1]{\LSpace[\infty]{#1}}
\newcommand{\LInfNorm}[2]{\LNorm[\infty]{#1}{#2}}

\newcommandx{\Vh}[1][1=h]{\myBold{V}_{\!\!{#1}}}

\newcommand{\TensorProductSpace}[2]{\left[{#1}\right]^{#2}}

\newcommand{\nn}{\myBold{n}}
\newcommand{\nnh}{\myBold{n}_h}
\newcommand{\nnhTilde}{\nn_{h,\kp}}

\newcommand{\PP}{\myBold{P}}
\renewcommand{\P}{P}
\newcommand{\PPh}{\PP_{\!h}}
\newcommand{\Ph}{\P_h}
\newcommand{\QQ}{\myBold{Q}}
\newcommand{\Q}{Q}
\newcommand{\QQh}{\QQ_{\!h}}
\newcommand{\Qh}{\Q_h}

\newcommand{\QQhTilde}{\QQ_{h,\kp}}

\newcommandx{\PTensor}[2][1=k,2=l]{\TensorVar{\PP}[#1][#2]}
\newcommandx{\PhTensor}[2][1=k,2=l]{\TensorVar{\PP}[#1][#2]_h}
\newcommandx{\QTensor}[2][1=k,2=l]{\TensorVar{\QQ}[#1][#2]}
\newcommandx{\QhTensor}[2][1=k,2=l]{\TensorVar{\QQ}[#1][#2]_h}

\newcommand{\dPi}{B}

\newcommandx{\RTensor}[3][1={d+1},2={k},3={l}]{\big[\R^{#1}\big]^{{#2}+{#3}}}

\newcommand{\Ih}{\myBold{I}_h}
\newcommand{\I}{\myBold{I}}

\newcommandx{\PDQ}[4][1={\Weingarten},2={\PP},3={n}]{\sum_{r=1}^{k+l}{#1}\bigotimes_{r,l+1}{#2}\big[{#4}\underset{r}{\cdot}{#3}\big]}

\begin{document}

\title{\myTitle}

\author{Hanne Hardering\thanks{Technische Universit{\"a}t Dresden, Institut f{\"u}r Numerische Mathematik, D-01062 Dresden, Germany (\url{hanne.hardering@tu-dresden.de}).}
~~and~~
Simon Praetorius\thanks{Technische Universit{\"a}t Dresden, Institut f{\"u}r Wissenschaftliches Rechnen, D-01062 Dresden, Germany (\url{simon.praetorius@tu-dresden.de}).}}

\maketitle

\paragraph*{Abstract.}
\myAbstract


\section{Introduction}\label{sec:introduction}
We consider the following model problem: Let $\G$ be a closed $d$-dimensional surface embedded in $\R^{d+1}$. Given tensor-valued data $\TT{f}$, the goal is to find tensor fields $\TT{u}$ that solve
\begin{equation}
  -\bm{\Delta}_\G\TT{u} + \TT{u} = \TT{f}\quad\text{ on }\G\,.
\end{equation}
Here, $\bm{\Delta}_\G$ is the connection-Laplacian, the natural extension of the Laplace-Beltrami operator to tensor fields based on covariant derivatives.

We study this problem in a variational setting: Find $\TT{u}\in\HSpace[1]{\G, \TensorBundleTan}$ such that
\begin{equation}\label{eq:continuous_model_problem}
  \Inner{\CDer\TT{u}}{\CDer\TT{v}}{\G} + \Inner{\TT{u}}{\TT{v}}{\G} = \Inner{\TT{f}}{\TT{v}}{\G}\quad\forall \TT{v}\in\HSpace[1]{\G, \TensorBundleTan}\,,
\end{equation}
with $\TensorBundleTan$ the rank $n=k+l$ tensor bundle of $k$-covariant and $l$-contravariant tensors, $\CDer$ the covariant derivative\EN{}{,} and $\TT{f}\in\LSpace[2]{\G, \TensorBundleTan}$ the given data.

This problem is related to the \EN{minimisation}{minimization} of the full covariant $H^1$-norm of tensor fields and has applications in the area of nematic Liquid crystals described by surface Q-tensor fields, e.g., \cite{BDN2010Finite,NV2012Surface,NN2017Orientational,NN2019Finite}, orientation field computation on flexible manifolds and shells, e.g., \cite{NV2021Extrinsic,SSV2016Analysis,NR2019Hydrodynamic}, the time-\EN{discretisation}{discretization} of the vector and tensor-heat equation, e.g., \cite{CW2013Geodesics,GL2014Discrete,SS2019Vector}\EN{}{,} and surface Stokes and Navier-Stokes equations, e.g., \cite{RV2016Incompressible,RV2018Solving,Fr2018Higher,JO2018Incompressible,GA2018Hydrodynamic,JR2019Higher,JO2020Error,LL2020Divergence,BJPRV2021StokesComparison,TH2021Isogeometric}.

\subsection{Description of the numerical \EN{discretisation}{discretization}}\label{sec:description-of-the-numerical-discretization}
Let $\Gh$ be a piecewise polynomial approximation of the continuous surface $\G$, with local polynomial order $\kg$ and triangulation $\mathcal{K}_h$. We introduce the surface Lagrange finite element space $V_h\doteq V_h^{\ku}$ of local Lagrange elements of order $\ku$ on $\Gh$, see \eqref{eq:scalar-lagrange-space}. The product space $\Vh=\TensorProductSpace{V_h}{N}$ with $N=(d+1)^n$ is the base for the \EN{discretisation}{discretization} of tensor fields $\TT{u}_h:\Gh\to\TensorBundleAmb\cong\R^N$ in the ambient space.

These discrete tensor fields are not necessarily tangential to $\Gh$ nor to $\G$. In order to introduce a discrete formulation for \eqref{eq:continuous_model_problem}, we project the non-tangential tensors into the tangent space of the discrete surface. This tensor projection is denoted by $\PPh$ and is built from the discrete surface normal field $\nnh$.

Posing an equation with tangential operators only for tensors in the ambient space results in a non-trivial kernel of the bilinear form. Thus, we introduce a \EN{penalisation}{penalization} of normal contributions in the tensor fields. The normal projection is defined as the orthogonal projection to $\PPh$, i.e., $\QQh=\Id-\PPh$. We use the notation $\QQhTilde$ for a normal projection build from a normal vector $\nnhTilde$ that approximates the exact surface normal $\nn$ with order $\kp$ with typically $\kp\geq\kg$, see \eqref{eq:better-normal}.

Following these ideas, we formulate the discrete variational problem: Find $\TT{u}_h\in\Vh$ such that
\begin{multline}
	\label{eq:discrete_model_problem}
	A_h(\TT{u}_h,\,\TT{v}_h)\colonequals\Inner{\ChDer\PPh\TT{u}_h}{\ChDer\PPh\TT{v}_h}{\Kh} + \Inner{\PPh\TT{u}_h}{\PPh\TT{v}_h}{\Kh} \\
	 + \beta h^{-2\alpha}\Inner{\QQhTilde\TT{u}_h}{\QQhTilde\TT{v}_h}{\Kh} = \Inner{\TT{f}^e}{\PPh\TT{v}_h}{\Kh}\quad\forall\TT{v}_h\in\Vh\,,
\end{multline}
where the inner products and derivatives have to be understood elementwise. Here, $\alpha\in[0,1]$ and $\beta>0$ are parameters of the \EN{penalisation}{penalization} term. The parameters $\kg,\ku,\kp\geq 1$\EN{}{,} and $\alpha$ play a crucial role in the discussion of the errors of this \EN{discretisation}{discretization}. The right-hand side function $\TT{f}^e\in\LSpace[2]{\Gh, \TensorBundleAmb}$ denotes a proper extension of the continuous surface field $\TT{f}$.
The operator $\ChDer|_K=\CovariantDerivative{K}$ denotes the covariant tensor derivative on the smooth elements $K\in\mathcal{K}_h$ of the discrete surface $\Gh$. It can be represented by the projected Euclidean derivative $\DDer$ applied to the tensor fields and a normal correction term, i.e., $\ChDer\PPh\TT{u}_h=\PPh\DDer\TT{u}_h - \PPh\DDer\QQh\TT{u}_h$. This representation leads to an easy to implement numerical scheme.

An extrinsic \EN{discretisation}{discretization} similar to the one described here is introduced by \citet{NN2019Finite} for tensor fields on piecewise flat surfaces and is \EN{analysed}{analyzed} for the vector Laplacian without the zeroth order tangential term by \citet{HL2020Analysis} on \EN{parametrised}{parametrized} surfaces. Both approaches are based on surface finite elements for the \EN{discretisation}{discretization} of the components of these embedded tensors.

\emph{Alternative \EN{discretisation}{discretization} method.\;} The problem is also considered in the context of \EN{discretisation}{discretization} methods with a natural embedding structure, like Trace-FEM or Cut-FEM, for vector-valued PDEs, e.g., \cite{Le2016High,GJ2018Trace,BH2018Cut,JR2019Trace}. In contrast to methods with an extrinsic representations of the tensor fields, some groups consider intrinsic tensor and vector field \EN{parametrisations}{parametrizations} and \EN{discretisations}{discretizations}, e.g., a local Monge \EN{parametrisation}{parametrization}, see, \cite{TS2020Approximation}, a tensor hodge-decomposition in combination with discrete exterior calculus \EN{discretisations}{discretizations}, e.g., \cite{GL2014Discrete,Liu2015Discrete,Hirani2003Discrete}, a direct intrinsic \EN{parametrisation}{parametrization} of the surface, e.g., \cite{BaFaPu2021Intrinsic,BaPu2020Geometrically}, or an extension of spherical harmonics to tangential tensor fields on the sphere, e.g., \cite{Ja1976Tensor,Wi1982Tensor,FS2009Spherical}.
We focus here solely on the \EN{penalisation}{penalization} based surface finite element \EN{discretisation}{discretization}.

While the analysis of the vector field \EN{discretisations}{discretizations} has advanced far already, see, \cite{HL2020Analysis}, tensor field \EN{discretisations}{discretizations} are only experimentally considered, e.g., \cite{NN2019Finite,TS2020Approximation}.
Based on the work of \citet{HL2020Analysis}, we extend their findings towards higher rank tensor fields. Additionally, \citet{HL2020Analysis} observes in the conclusion that ``it is of course also of interest to investigate the convergence \EN{behaviour}{behavior} for the tangential part of the solution. [$\ldots$] Tangential convergence is arguably more natural to consider as we [$\ldots$] seek to approximate a tangential vector field.''. While numerical evidence that there is optimal order convergence for tangential $L^2$-errors is given, the observation is not proven. For non-tangential norms optimal order estimates require a higher order geometry approximation for the \EN{penalisation}{penalization} term. We show that estimates in the tangential $L^2$- and $H^1$-errors can be optimal even if improved geometry approximations are not available. In order to do so, we introduce a more flexible choice of \EN{penalisation}{penalization} term that also allows to \EN{analyse}{analyze} the experimental results of \citet{NN2017Orientational,NN2019Finite} on piecewise flat surfaces.

In summary,
\begin{enumerate}
	\item we are investigating an elliptic reference problem for tensor fields on surfaces and show that the analysis for vector fields can be \EN{generalised}{generalized} in a straightforward way to higher tensorial rank without loosing convergence orders;
	\item we investigate the tangential $L^2$- and $H^1$-norm errors of the discrete solution of the reference problem, in particular when higher order approximations of the surface are not readily available;
	\item we allow for flexibility in the \EN{penalisation}{penalization} term and show that the popular choice of scaling by $\sim h^{-2}$ is not optimal in multiple interesting cases;
	\item we show that even for the piecewise flat surface approximation a converging numerical scheme can be constructed.
\end{enumerate}

\subsection{Main results}\label{sec:main-results}
In this paper we \EN{analyse}{analyze} the error between the solution $\TT{u}$ of the continuous problem \eqref{eq:continuous_model_problem} and the solution $\TT{u}_h\in\Vh$ of the discrete problem \eqref{eq:discrete_model_problem} lifted to the continuous surface, $\TT{u}-\TT{u}^\ell_h$, in various norms.

The first estimate is in the energy norm $\EnergyNorm{\TT{e}_h}{A_h}^2=A_h(\TT{e}_h,\TT{e}_h)$. Assume that we have $\TT{u}\in\HSpaceTan[\ku+1]{\G}$ the solution of the continuous problem, then we get the estimate
\begin{align*}
	\EnergyNorm{\TT{u}^e-\TT{u}_h}{A_h} &\leqC  h^{m}\big(\Norm{\TT{f}}{\G} + \HNormTan[\ku+1]{\TT{u}}{\G}\big)\,,
\end{align*}
with $m=\min\{\ku,\kg+\alpha-1,\kp-\alpha\}$, shown in \cref{thm:energy-norm-estimate}. Note that the geometric part of the error, given by the term $\mathcal{O}(h^{\kg+\alpha-1})$, is of lower order than a corresponding error contribution for scalar surface finite elements, cf. \cite{De2009Higher}. This is due to the additional projections involved for vector and tensor fields in all the terms of the discrete problem \eqref{eq:discrete_model_problem}. For $\alpha\in\{0,1\}$ the last two terms coming from the penalty contribution result in an even lower order convergence, i.e., $\kg-1$ or $\kp-1$. In a former analysis by \citet{HL2020Analysis} of the vector Laplacian this could be compensated by setting $\alpha=1$ and increasing $\kp$. Allowing for a variable $\alpha$, we see that for the linear setting $\kp=\kg=1$ we still get convergence if we choose $1 > \alpha > 0$, with the optimum $\alpha=1/2$.

As one is often only interested in the tangential norms, we also consider the $H^1$-norm of the tangential part of the error. We observe an improvement from the energy norm estimate in many cases. In particular, we prove
\begin{align*}
	\HNormTan{\PP(\TT{u} - \TT{u}_h^{\ell})}{\G}
	&\leqC  h^{\hat m}\left(\Norm{\TT{f}}{\G} + \HNormTan[k_u+1]{\TT{u}}{\G}\right)\,
\end{align*}
with $\hat{m}=\min\{\ku,\kg,\kg+\kp-2\alpha, \kg-1+2\alpha\}$ in \cref{thm:better-h1-estimate}. We observe that this estimate is optimal for $\alpha=1/2$\EN{}{,} and the order will be $\min\{\ku,\kg\}$ in this case. This corresponds to the optimal order that we see in numerical experiments in \cref{sec:numerical-experiments}. Note that this is especially an improvement from the energy norm estimate for the typical choice $\alpha\in \{0,1\}$ in the isogeometric case $\kg=\kp$.

From the estimate in the energy norm and the estimate in the tangential $H^1$-norm we deduce further improvements in the corresponding $L^2$-norms. In particular we prove
\begin{align*}
	\Norm{\left(\Id-\PP\right)(\TT{u} - \TT{u}_h^{\ell})}{\G}
	&\leqC  h^{ m + \alpha}\left(\Norm{\TT{f}}{\G} + \HNormTan[k_u+1]{\TT{u}}{\G}\right)\,,\\
	\Norm{\PP(\TT{u} - \TT{u}_h^{\ell})}{\G}
	&\leqC  h^{\min\{\hat m +1,\tilde{m}\}}\left(\Norm{\TT{f}}{\G} + \HNormTan[k_u+1]{\TT{u}}{\G}\right)\,,
\end{align*}
with $\tilde{m}\colonequals \min\{\kg+1, 2m,m+3-\alpha\}$ in \cref{thm:l2-norm-estimate}. For $\kg\geq 2$, we again see that the optimal order of $\min\{\ku+1,\kg+1\}$ for the tangential part is obtained by choosing $\alpha=1/2$. Only for the case $\kg=1$ and $\kp\geq 2$, the choice $\alpha=1/2$ is sub-optimal as it only provides a convergence order of $1$ that can be improved to $2$ by choosing $\alpha=1$. If one is mainly interested in a better convergence of the normal part of the error and if it is possible to choose $\kp\geq\kg+1$, there will also be an improvement by choosing $\alpha$ greater than $1/2$. In this case $\alpha=1$ will yield the optimal order of convergence of $\min\{\ku+1,\kg+1\}$ for the full $L^2$-norm.

All results are backed up by the numerical experiments in \cref{sec:numerical-experiments}. We observe the claimed optimal order for $\alpha=1/2$. However, the optimal order can be obtained for more cases than in the analytical estimates, in particular for even choices of $\kg$. This is not covered by this paper.

\subsection{Outline of the paper}\label{sec:outline-of-the-paper}
We start this paper by introducing the notation and collecting fundamental results for discrete surfaces and tensor fields in the beginning of \cref{sec:tensor fields-on-surface}. This section is closed by introducing Sobolev spaces for tensor fields on continuous and discrete surfaces. The following \cref{sec:finite-element-discretization-of-tensor-equation} focusses on the variational problem \eqref{eq:continuous_model_problem} and its finite element \EN{discretisation}{discretization} \eqref{eq:discrete_model_problem}. The error analysis of the discrete problem is discussed in \cref{sec:error-analysis}. This is followed by numerical experiments in \cref{sec:numerical-experiments} confirming the found theoretical results.

\section{Tensor fields on surfaces}\label{sec:tensor fields-on-surface}
We study tensor-valued PDEs on curved surfaces. Therefore, we need to introduce notation for tensor fields and clarify the assumptions on the surface properties needed to make the assertion of geometrical and \EN{discretisation}{discretization} errors.
This section also covers some basic estimates for general tensor fields on the continuous and corresponding approximate discrete surface and the relation between these fields and surfaces.

\subsection{Surface and discrete surface approximation}\label{sec:surface-and-discrete-surface-approximation}
First, we introduce the notion of hypersurfaces and their approximation by piecewise polynomial discrete surfaces. This section covers definitions and general geometric estimates.

\subsubsection{The smooth surface}\label{sec:the-smooth-surface}
Let $\G$ be a compact, oriented, closed\EN{}{,} and $C^{\infty}$, d-dimensional hypersurface, isometrically embedded in $\R^{d+1}$.
We restrict the analysis to dimensions $1\leq d\leq 3$.
Let $\TangentBundle{\G}$ denote the tangential vector bundle of $\G$. We introduce for each point $x\in \G$ the ambient tangent space $\TangentBundle[x]{\R^{d+1}}\cong\R^{d+1}$ as the space of $\R^{d+1}$ vectors with origin in $x$. This space locally splits into an orthogonal direct sum $\R^{d+1}=\TangentBundle[x]{\G} \oplus N_{x}\G$, where $N_{x}\G$ is the normal space at $x$. We denote the normal bundle by $N\G$.
As we consider hypersurfaces only\EN{}{,} and $\G$ is oriented, we can pick out the unique outward pointing normal $\nn_\G(x)$ at each point $x\in\G$. This normal vector field induces projection operators $\P_\G(x) = \Id - \nn_\G(x)\otimes \nn_\G(x)$ into the tangent space $\TangentBundle[x]{\G}$ and $\Q_\G(x) = \nn_\G(x)\otimes \nn_\G(x)$ into the normal space $N_{x}\G$.

Let $U_{\delta}\subset \R^{d+1}$ be a $\delta$-\EN{neighbourhood}{neighborhood} of $\G$ with $\delta>0$ sufficiently small such that the closest-point projection $\pi:U_{\delta}\to\G$ from this \EN{neighbourhood}{neighborhood} to $\G$ is unique. The projection mapping is defined by
\begin{equation}\label{eq:closest-point-projection}
  \pi(x) \colonequals x - \rho(x)\nn(x)
\end{equation}
where $\rho$ is the signed distance function with $\rho(x)<0$ for $x$ in the interior of $\G$.
With this distance function the surface normal vector field is naturally extended to $U_{\delta}$ by $\nn\colonequals\DD\rho$ with $\DD$ the (componentwise) Euclidean derivative, see \cref{def:euclidean-derivative}, and $\nn_\G = \nn|_\G$. As $\P_\G$ is defined in terms of $\nn_\G$, we introduce the extended projections $\P$ and $\Q$ defined in terms of $\nn$.
The extended normal vector also allows to define the extended Weingarten map on the surface and in $U_\delta$, by $\Weingarten \colonequals -\DDer\nn$.

\subsubsection{Approximation of the surface}\label{sec:approximation-of-the-surface}
Let $\hat{\G}_{h}\subset U_{\delta}$ be a shape-regular, quasi-uniform triangulated surface with element diameter $h$ and vertices on $\G$. We denote by $\hat{\mathcal{K}}_{h}$ the set of triangular faces of $\hat{\G}_{h}$.
On this piecewise flat triangulated surface the usual $m$-th order Lagrange finite element spaces can be defined by
\begin{equation}\label{eq:scalar-flat-lagrange-space}
  \hat{V}_h^m \doteq \hat{V}_h^m(\hat{\G}_{h},\mathbb{R}) \coloneqq \Set{v\in C^0(\hat{\G}_{h}) \mid v|_{\hat{K}}\in\mathbb{P}^{\hat{K}}_m,\;\forall \hat{K}\in\hat{\mathcal{K}}_h}\,,
\end{equation}
where $\mathbb{P}^{\hat{K}}_m$ denotes the space of polynomials of degree at most $m$ on $\hat{K}$. Note that $\mathbb{P}^{\hat{K}}_m$ is defined as the mapping of polynomials $\mathbb{P}_m$ from a flat $d$-dimensional reference element $K_{\text{ref}}\subset\R^d$ to the flat surface element $\hat{K}\subset\R^{d+1}$. For vector-valued maps we use the notation
\begin{equation}\label{eq:vector-flat-lagrange-space}
  \hat{V}_h^m(\hat{\G}_{h},\mathbb{R}^N) \coloneqq \left[\hat{V}_h^m\right]^N\,,
\end{equation}
that has to be understood as a componentwise definition.

We consider the restriction of the closest-point projection $\pi$ to $\hat{\G}_{h}$ as a mapping $\pi:\hat{\G}_{h}\to\mathbb{R}^{d+1}$. Its componentwise Lagrange interpolation $\hat{I}_h^k\pi$ into $\hat{V}_h^k(\hat{\G}_{h},\mathbb{R}^{d+1})$ is denoted by $\pi_{h,k}$. Mapping the piecewise flat surface results in a higher-order approximation $\G_{h,k} \colonequals \pi_{h,k}(\hat{\G}_{h})$ of the surface $\G$ that is contained in $U_\delta$ if $h$ is small enough. In the following we assume that $h$ is chosen correspondingly. Note that $\G_{h,1} = \hat{\G}_{h}$.
Associated to the discrete surface $\G_{h,k}$ is a set of surface elements
\[
	\K_{h,k}\colonequals \Set{\pi_{h,k}(\hat{K}) \mid \hat{K}\in\hat{\K}_{h}}\,,
\]
such that $\G_{h,k}=\bigcup_{K\in\K_{h,k}}\bar{K}$.

In the lines of the construction of the \EN{parametrised}{parametrized} surface $\G_{h,k}$ by mapping the reference surface $\hat{\G}_h$ by $\pi_{h,k}$, we can also obtain the continuous surface $\G$ by lifting the discrete surface as $\G=\pi(\G_{h,k})$. This induces a partitioning of $\G$ in surface elements
\[
	\K \colonequals \Set{\pi(K) \mid K\in\K_{h,k}}\,,
\]
such that $\G=\bigcup_{K\in\K}\bar{K}$.

For details about the construction and its properties, see, e.g., \cite{De2009Higher,DE2013Finite}.

\subsubsection{Properties of the discrete surface}\label{sec:properties-of-the-discrete-surface}

For the \EN{parametrised}{parametrized} surface $\G_{h,\kg}$ with polynomial order $\kg \geq 1$ and discrete surface normal vector $\nn_{\Gh}$ it is shown  that the following geometric estimates hold \cite[Proposition 2.3]{De2009Higher}:
\begin{align}
  \LInfNorm{\pi - \pi_{h,\kg}}{\G_{h,\kg}}         &\leqC h^{\kg+1}, \\
  \LInfNorm{\nn - \nn_{\Gh}}{\G_{h,\kg}}           &\leqC h^{\kg}, 		\label{eq:n-nh} \\
  \LInfNorm{\DDer\nn - \DDer\nn^e_{\Gh}}{K}				 &\leqC h^{\kg-1},	\label{eq:w-wh}
\end{align}
with $a\leqC b$ means $a\leq C\; b$ with $C$ a constant independent of $h$. $K\in\K_{h,\kg}$ denotes an element of the triangulation of $\G_{h,\kg}$ and $\nn_{\Gh}^e = \nn_{\Gh}\circ(\pi\vert_{\G_{h,\kg}})^{-1}\circ\pi$ an extension of $\nn_{\Gh}$ into the \EN{neighbourhood}{neighborhood} of $K$.

The estimate \eqref{eq:w-wh} can even be \EN{generalised}{generalized} to $\LInfNorm{\DDer^l\nn - \DDer^l\nn^e_{\Gh}}{K}	\leqC h^{\kg-l}$ for all $0\leq l\leq \kg$ and can be proven analogously to the case $l=1$ \cite[Proposition 2.3]{De2009Higher} using interpolation error estimates for $\pi_{h,\kg}$. As we need only the estimate for $l=0,1$, we refrain from stating this more explicitly although \EN{generalisations}{generalizations} of several lemmata in \cref{sec:interpolation-errors} are possible. In the  following we denote by $\Gh\doteq\G_{h,\kg}$ the \EN{parametrised}{parametrized} surface approximation of order $\kg$ and denote all other quantities on this discrete surface with the subscript $h$ dropping the explicit notation of $\kg$.

Due to the construction of $\K$, each element $K_\G\in\K$ has a one-to-one association to a discrete element $K\in\Kh$ in the triangulation of $\Gh$. The corresponding coordinate association is defined in terms of the closest-point projection $\pi$. As in the definition of the continuous normal field, we lift $\nn_{\Gh}$ from $\Gh$ to $\G$ and extend to the \EN{neighbourhood}{neighborhood} $U_\delta$ by $\nnh\colonequals \nn_{\Gh}^e$. Note that $\nnh\vert_{\Gh}=\nn_{\Gh}$. This extension also defines discrete projections $\Ph=\Id-\nnh\otimes\nnh$ and $\Qh=\Id-\Ph$.

\subsubsection{Derivative of the closest-point projection}\label{sec:derivative-of-the-closest-point-projection}

At every $\tilde{x}\in\Gh$ inside an element, the differential of the mapping $\pi:\Gh\to\G$ is a bijective function, $d\pi_{\tilde{x}}:T_{\tilde{x}}\Gh\to T_{\pi(\tilde{x})}\G$. We denote by $x=\pi(\tilde{x})$ the associated point on $\G$. In extrinsic coordinates, $d\pi_{\tilde{x}}$ is given by
\begin{align*}
	\dPi(\tilde{x})\colonequals\DDer_{\Gh}\pi(\tilde{x}) = \P(x)(\Id +\rho(\tilde{x})\Weingarten(x))\Ph(\tilde{x}),
\end{align*}
and the inverse of $\dPi$ is given explicitly by
\begin{align*}
	\dPi^{-1}(x)=\Ph(\tilde{x})\left(\Id - \frac{\nn(x)\otimes \nnh(\tilde{x})}{\langle \nn(x), \nnh(\tilde{x})\rangle} \right)(\Id +\rho(\tilde{x})\Weingarten(x))^{-1}\P(x).
\end{align*}
It is easy to check that $\dPi(\tilde{x})\dPi^{-1}(x)=\P(x)$ and $\dPi^{-1}(x)\dPi(\tilde{x})=\Ph(\tilde{x})$ for $\tilde{x}\in\Gh$ and $x=\pi(\tilde{x})$.

The following Lemma \EN{summarises}{summarizes} geometric estimates for the mapping $\dPi$ and its inverse in relation to the projection operators $\P$ and $\Ph$. These estimates allow the comparison of covariant derivatives on the different surfaces.

\begin{lemma}[Estimates involving $B$, $B^{-1}$ and $\det{B}$]\label{lem:B}
	Let $\G$ and $\Gh$ be as above. Then there exist constants $C$ independent of the grid-size $h$ such that we have for $B$ and $B^{-1}$ the estimates
\begin{align*}
		\LInfNorm{\dPi}{\Gh}          &\leq C\,, &
		\LInfNorm{\dPi^{-1}}{\G}         &\leq C\,, \\
		\LInfNorm{\P\Ph - \dPi}{\Gh}  &\leqC h^{\kg+1}\,, &
		\LInfNorm{\Ph\P - \dPi^{-1}}{\G} &\leqC h^{\kg+1}\,,
\intertext{and for the surface measures $d\G=\Abs{\det(B)}d\Gh$ the estimates}
\LInfNorm{1-\Abs{\det(\dPi)}}{\Gh}    &\leqC h^{\kg+1}\,, \\
		\LInfNorm{\Abs{\det(\dPi)}}{\Gh}      &\leq C\,, &
		\LInfNorm{\Abs{\det(\dPi)}^{-1}}{\Gh} &\leq C\,.
\intertext{Furthermore, we have}
\LInfNorm{\dPi\dPi^{T} - \P}{\G}            &\leqC h^{\kg+1}\,, &
		\LInfNorm{\dPi^{-T}\dPi^{-1} - \Ph}{\Gh} &\leqC h^{\kg+1}\,.
	\end{align*}
\end{lemma}
\begin{proof}
	See \cite[Lemma 4.1]{HL2020Analysis} and \cite[Lemma 4.4]{DE2013Finite}.
\end{proof}

\begin{remark}
	For $\kg=1$ the discrete surface $\Gh$ coincides with the reference surface $\hat{\G}_h$. We denote by $\hat{\dPi}$ the corresponding derivative of $\pi$. The closest-point projection $\pi$ on $\hat{\G}_h$ is bounded, i.e.,
\begin{equation}\label{eq:boundedness-B-hat}
		\Norm{\pi}{C^\infty(\hat{K})} \leq C,\quad \Norm{\pi\vert_{\hat{\G}_h}^{-1}}{C^\infty(K)} \leq C
	\end{equation}
	for $\hat{K}\in\hat{\K}_h$ and $K\in\K$. The first boundedness follows from the smoothness of $\pi$ and the second estimate by boundedness of the first derivative $\hat{\dPi}^{-1}$ of $\pi$ and chain rule.
\end{remark}

\begin{lemma}[Boundedness of the discrete projection derivatives]\label{lem:pi_h-pi_hinv}
	Let the derivative of the discrete mapping $\pi_h$ on $\hat{\G}_h$ be denoted by
	\[
		\hat{B}_h\colonequals \DDer_{\hat{\G}_h} \pi_h = \DDer_{\hat{\G}_h} I_h^{\kg} \pi\vert_{\hat{\G}_h}\,.
	\]
	Then, boundedness of $\hat{B}_h$ and its inverse $\hat{B}^{-1}_h$ w.r.t. the corresponding $C^0$ norm holds, i.e., $\Norm{\hat{B}^{-1}_h}{C^0(K_h)} \leq C$,
	for $K_h\in\Kh$. As a consequence we have boundedness of the discrete projections:
	\[
		\Norm{\pi_h}{C^\infty(\hat{K})} \leq C,\quad \Norm{\pi_h^{-1}}{C^\infty(K_h)} \leq C
	\]
	for $\hat{K}\in\hat{\K}_h$ and $K_h\in\Kh$ with generic constants $C$ independent of $h$.
\end{lemma}
\begin{proof}
	The boundedness of $\pi_h$ follows directly from continuity of the interpolation operator $I_h^{\kg}$ and smoothness of $\pi$.
For $\hat{B}^{-1}_h$ we write
	\begin{align*}
		\hat{B}^{-1}_h(\pi_h(\hat{x}))
		& = \hat{B}^{-1}(\pi(\hat{x}))\left(\hat{B}(\hat{x}) - \hat{B}_h(\hat{x})\right)\hat{B}^{-1}_h(\pi_h(\hat{x})) + \hat{B}^{-1}(\pi(\hat{x}))\Ph(\pi_h(\hat{x})),\quad\text{for }\hat{x}\in\Gh
	\end{align*}
	and estimate using interpolation-error estimates
	\begin{align*}
		\Norm{\hat{B}^{-1}_h}{C^0(K_h)}
			&\leq \Norm{\hat{B}^{-1}}{C^0(K)}\Norm{\hat{B}_h - \hat{B}}{C^0(\hat{K})}\Norm{\hat{B}_h^{-1}}{C^0(K_h)} + \Norm{\hat{B}^{-1}}{C^0(K)}\\
			&\leq C\;h^{\kg}\Norm{\hat{B}_h^{-1}}{C^0(K_h)} + \Norm{\hat{B}^{-1}}{C^0(K)}\,.
	\end{align*}
	For $h$ small enough, we can compensate the first term into the left-hand side and by \eqref{eq:boundedness-B-hat} get the desired boundedness. Higher order estimates then follow from the chain rule.
\end{proof}

\subsection{Tensor fields}\label{sec:Tensor-fields}
While intrinsic tensor fields are described in terms of the surface intrinsic metric and a basis of the tangential space, we aim for an extrinsic description in terms of the basis of the $\R^{d+1}$. This allows for a global simple \EN{parametrisation}{parametrization} and representation of continuous fields and gives rise to a corresponding representation of non-intrinsic fields with normal components.

\subsubsection{Tensors and tensor fields}\label{sec:tensors-and-tensor fields}

We denote by $\mathbf{e}_1,\ldots,\mathbf{e}_{d+1}\in\R^{d+1}$ the canonical basis of $\R^{d+1}$. Their projections into $\TangentBundle{\G}$ introduces a local spanning set $\mathbf{p}_1(x)\colonequals\P(x)\mathbf{e}_1,\ldots,\mathbf{p}_{d+1}(x)\colonequals\P(x)\mathbf{e}_{d+1}$ for $x\in\G$. For each tangent vector field $V\in\TangentBundle{\G}$ we denote the representation in extrinsic coordinates $V(x) = V^i(x) \mathbf{e}_i = V^i(x) \mathbf{p}_i(x)$ by $\underline{V}(x)=[V^i(x)]_i$ with coordinate functions $V^i : \G\to\R$. This extrinsic representation induces a canonical (Frobenius) inner product $\langle V,\,W\rangle \colonequals\sum_i V^i\, W^i$.

Similarly, elements $\TT{f}$ of the tangential tensor bundle $\TensorBundleTan$ of rank-$n$ contravariant tensors have a representation in extrinsic coordinates by
\begin{align*}
	\TT{f}(x) &= f^{i_1,\ldots,i_n}(x)\, \mathbf{e}_{i_1}\otimes\cdots\otimes \mathbf{e}_{i_n}\\
					  &= f^{i_1,\ldots,i_n}(x)\, \mathbf{p}_{i_1}(x)\otimes\cdots\otimes \mathbf{p}_{i_n}(x)\,,
\end{align*}
with $\underline{\TT{f}}(x) = [f^{i_1,\ldots,i_n}(x)]_{i_1,\ldots,i_n}$. This representation induces the canonical inner product $\langle V_{1}\otimes V_{2},W_{1}\otimes W_{2}\rangle\colonequals \langle V_{1},W_{1}\rangle \langle V_{2},W_{2}\rangle$.
The tangent bundle $\TangentBundle{\G}\equiv\TensorBundleTan[1]$ is represented by the rank-1 tensor bundle.
The space of tensors with smooth coefficient functions $f^{i_1,\ldots,i_n}\in C^\infty(\G,\R)$ defines the space of smooth tensor fields $\TT{f}\in\TensorFieldTan[n]$.

For each $x\in\G$ we denote by $\TensorBundleAmb[n][x]$ the non-tangential tensors that can be seen as product of Euclidean spaces, i.e., $\TensorBundleAmb[n][x]\cong \R^{(d+1)^n}$ and the corresponding tensor bundle is denoted by $\TensorBundleAmb[n]$. With $\TangentBundle{\R^{d+1}}\equiv\TensorBundleAmb[1]$ the ambient tangent space is given.

\begin{remark}
	We do not distinguish between covariant and contravariant tensors here in this manuscript. All the tensor fields are represented w.r.t. an Euclidean basis. Thus, raising and lowering of indices is an index contraction with the Kronecker delta tensor. For simplicity of notation we mostly work with contravariant index position here and just write the full tensor rank $n$ instead of separation by covariant and contravariant type.
\end{remark}

\begin{remark}
	In most parts of this manuscript we implicitly assume $n > 0$ and thus exclude scalar fields. Those are special cases \EN{analysed}{analyzed} elsewhere, e.g., \cite{De2009Higher}.
\end{remark}

\begin{definition}[Tensor projection]\label{def:tensor-projection}
	The tangential projection $\P=\Id - \nn\otimes\nn$ is extended to a projection of (non-tangential) tensor fields $\TT{f} = f^{i_1,\ldots,i_n} \mathbf{e}_{i_1}\otimes\ldots\otimes\mathbf{e}_{i_n}\in\TensorFieldAmb[n]$ by
\begin{align*}
		\PP\TT{f} \colonequals&\, f^{i_1,\ldots,i_n} \P\mathbf{e}_{i_1}\otimes\ldots\otimes\P\mathbf{e}_{i_n}
							=\, f^{i_1,\ldots,i_n} \mathbf{p}_{i_1}\otimes\ldots\otimes\mathbf{p}_{i_n}\,.
	\end{align*}
\end{definition}

	The tensor projection is defined as an operator projecting all components of the tensor using the tangential projection. The projected tensor is a tangential tensor. For vector fields $V$ the tensor projection corresponds to the usual tangential projection, $\PP V=\P V$, for 2-tensor fields $\TT{f}$ it corresponds to a row and column projection, $\PP\TT{f}=\P\TT{f}\P$.

\begin{definition}[Discrete tensor projection]
	We denote by $\PPh$ the tensor projection into the tangent space of $\Gh$, defined in terms of $\Ph=\Id-\nnh\otimes\nnh$ elementwise similar to \cref{def:tensor-projection}. Those elementwise projections are formally glued together to define a projection operator for tensor fields on $\Gh$.
\end{definition}
\begin{remark}
	Note, on the facets in the intersection of two elements, there is no tangent space defined and thus also no tangential projection. In the finite element method the (bi-)linear forms are defined elementwise and quadrature typically is performed in the inside of the elements. Thus, this definition is no restriction for the practical implementation. However, in the analysis one has to be careful, e.g., when considering interpolation of projected fields. Often continuity is required for interpolation operators.
\end{remark}

In order to compare tensor fields on $\G$ and $\Gh$ we choose particular extensions and lifts. Those are defined in terms of the closest-point projection $\pi:U_\delta\to\G$.

\begin{definition}[Extension and lifting of tensor fields]\label{def:extension}
	The extension of tensor fields $\TT{f}:\G\to\TensorBundleAmb$ from $\G$ to $U_\delta$ is defined by
\begin{equation*}
		\TT{f}^{e} \colonequals \TT{f}\circ\pi:U_\delta\to\TensorBundleAmb\,.
	\end{equation*}
Analogously, we define the lifting of tensor fields $\TT{f}_{h}:\Gh\to\TensorBundleAmb$ from $\Gh$ to $\G$ as the inverse of the extension restricted to $\Gh$,
\begin{equation*}
		\TT{f}_{h}^{\ell} \colonequals \TT{f}_{h}\circ \pi|_{\Gh}^{-1} :\G\to \TensorBundleAmb\,.
	\end{equation*}
	To extend a tensor field $\TT{f}_{h}:\Gh\to\TensorBundleAmb$ from $\Gh$ to $U_\delta$, we first lift it to $\G$ and then extend, i.e.,
	\begin{equation*}
		\TT{f}_{h}^{e} \colonequals (\TT{f}_{h}^{\ell})^{e} :U_{\delta}\to \TensorBundleAmb\,.
	\end{equation*}
\end{definition}

\subsubsection{Derivative of tensor fields}\label{sec:derivative-of-tensor fields}

We introduce derivatives of surface tensor fields, first the Euclidean derivative of extended tensor fields, second the covariant derivative. Those derivative definitions are just collections of standard definitions to introduce notation, see, e.g., \cite{DE2013Finite,HL2020Analysis}.

\begin{definition}[Euclidean derivative]\label{def:euclidean-derivative}
  Let the tensor field $\TT{f}:\G\to\TensorBundleAmb[n]$ be given in the basis notation $\TT{f} = f^{i_1,\ldots,i_n} \mathbf{e}_{i_1}\otimes\ldots\otimes\mathbf{e}_{i_n}$. We denote by $\bar{\TT{f}}$ a smooth extension of $\TT{f}$ in the \EN{neighbourhood}{neighborhood} of the surface $\G$ with $\bar{\TT{f}}|_\G=\TT{f}$. The Euclidean derivative of this tensor field is defined by
	\[
		\DDer\bar{\TT{f}} \colonequals \partial_{\mathbf{e}_j}\bar{f}^{i_1,\ldots,i_n}\mathbf{e}_{i_1}\otimes\ldots\otimes\mathbf{e}_{i_n}\otimes\mathbf{e}_{j}\,.
	\]
  The derivative $\DDer\bar{\TT{f}}$ is a rank-$(n+1)$ tensor field on $\G$.
\end{definition}

The Euclidean derivative of surface tensor fields $\TT{f}$ depends on the extension $\bar{\TT{f}}$ in the embedding space. In order to fix this, we consider the extension introduced before in \cref{def:extension}. For this an equivalent definition can be formulated in terms of tangential surface derivatives. This can be formulated intrinsically and thus does not need the notation of an extension explicitly.

\begin{definition}[Surface derivative]\label{def:surface-derivative}
  Let the tensor field $\TT{f}:\G\to\TensorBundleAmb[n]$ be given in the basis notation $\TT{f} = f^{i_1,\ldots,i_n} \mathbf{e}_{i_1}\otimes\ldots\otimes\mathbf{e}_{i_n}$ and let $\mathbf{p}_j=\P\mathbf{e}_j$. The tangential surface derivative of this tensor field is defined by
	\[
		\DDer_\G\TT{f} \colonequals \partial_{\mathbf{p}_j}f^{i_1,\ldots,i_n}\mathbf{e}_{i_1}\otimes\ldots\otimes\mathbf{e}_{i_n}\otimes\mathbf{e}_{j}\,.
	\]
	with $\partial_{\mathbf{p}_j}f=\mathbf{p}_j\cdot\DDer\bar{f}$ a tangential directional derivative, i.e., $\nabla_\G f = (\partial_{\mathbf{p}_j}f)\mathbf{e}_{j} = \P\DDer\bar{f}$ is the surface gradient.
  The derivative $\DDer_\G\TT{f}$ is a rank-$(n+1)$ tensor field on $\G$ but not necessarily tangential.
\end{definition}

\begin{remark}
	The surface derivative corresponds to the Euclidean derivative applied to an extended tensor field $\TT{f}^e=\TT{f}\circ\pi$ that is extended constant in normal direction, i.e., $\DDer_\G\TT{f}=\DDer\TT{f}^e|_\G$. For vector fields $V$ this corresponds to a one-sided projection of the full Euclidean derivative, i.e., $\DDer_\G V= \DDer\bar{V}|_\G\P$. For scalar fields $f$ it is equivalent to the surface gradient, i.e.,  $\DDer_\G f = \nabla_\G f$. This derivative operator can be defined intrinsically in terms of derivatives of a local \EN{parametrisation}{parametrization} and metric tensor and is implemented directly by surface finite elements.
\end{remark}

\begin{definition}[Covariant derivative]\label{def:covariant-derivative}
  Let $\TT{f}\in\TensorFieldTan[n]$ denote a tangential tensor field on the surface $\G$. Denote by $\bar{\TT{f}}$ a smooth extension of $\TT{f}$ as in \cref{def:euclidean-derivative} then the covariant derivative is defined by
	\[
		\CDer\TT{f}\colonequals \PP\DDer\bar{\TT{f}}|_\G\,.
	\]
\end{definition}

\begin{remark}
  The definition of the covariant derivative is independent of the choice of the extension, cf. \cite[Chapter 3]{DC1992Riemannian}\EN{}{,} and can also be written in terms of the tangential surface derivative, i.e., $\CDer\TT{f} = \PP\DDer_\G\TT{f}$.
  For scalar fields $f:\G\to\R$ this corresponds to the classical surface gradient, i.e., $\CDer f = \nabla_\G f = \P\DDer\bar{f}|_\G$. For vector fields $V$ the covariant derivative can be expressed using the projection matrix $\P$ as $\CDer V = \P\DDer\bar{V}|_\G\P$.

\end{remark}

\begin{definition}[Higher-order derivatives]
	For a tensor field $\TT{f}$ and its extension $\bar{\TT{f}}$ as above, we introduce the $s$-order derivatives as
	\[
		\DDer^s	\bar{\TT{f}} \colonequals \DDer\ldots\DDer\bar{\TT{f}}\,,\quad
		\DDer_\G^s \TT{f} \colonequals \DDer_\G\ldots\DDer_\G\TT{f}\,,\text{ and}\quad
		\CDer^s	\PP\TT{f} \colonequals \CDer\ldots\CDer\PP\TT{f}\,,
	\]
	by $s$-times application of the first-order derivatives $\DDer$, $\DDer_\G$\EN{}{,} and $\CDer$, respectively.
\end{definition}

\begin{remark}
	The definitions of the covariant derivative on $\G$ can directly be applied to define a covariant derivative inside the (smooth) elements $K\in\mathcal{K}_{h}$ of the discrete surface $\Gh$. These local definitions are glued together to formally write the discrete surface derivatives of a discrete tensor field $\TT{f}_h:\Gh\to\TensorBundleAmb$ as $\ChDer\PP_h\TT{f}_h(x) \colonequals \CovariantDerivative{K}\PP_h\TT{f}_h(x)$ for $x\in K$. The Euclidean derivative is directly defined by the derivative of an extended tensor field $\bar{\TT{f}}_h$ of $\TT{f}_h$ with $\bar{\TT{f}}_h|_{\Gh}=\TT{f}_h$ and the surface derivative is intrinsically defined in the smooth element $K$, denoted by $\DDer_K$.
\end{remark}

\subsection{Function spaces and norms}\label{sec:function-spaces-and-norms}

In the following we define function spaces for tensor fields on $\G$ and on $\Gh$ for tangential and non-tangential tensor fields. Due to the local definition of (discrete) projections and derivatives, we define all spaces and norms elementwise summed over the elements of the corresponding triangulations. This is denoted by $\K$ and $\Kh$ in the definitions. Elementwise defined spaces are called \emph{broken Sobolev spaces} \cite{HL2020Analysis,ER2020unified}.

\begin{definition}[$L^p$-spaces]
We denote by $\LSpace[p]{\K, \TensorBundle{M}}$, for $M$ either $\G$ or $\R^{d+1}$, the space of tensor fields $\TT{f}:\G\to\TensorBundle{M}$ with finite $L^p$-norm with elementwise integrals, i.e.,
\[
  \LNorm[p]{\TT{f}}{\K}^p\coloneqq\sum_{K\in\K} \int_K \FNorm{\TT{f}}^p\,\textrm{d}\G
\]
with $\FNorm{\TT{f}}$ the Frobenius norm induced by the Frobenius inner product $\FInner{\cdot}{\cdot}$.
\end{definition}

The norm is defined elementwise in order to allow the evaluation of fields that are only defined piecewise, like projected tensors on the discrete surface or derivatives of only piecewise smooth functions. Analogously, the $L^p$-space and norm is defined for functions on the discrete surface $\Gh$ by integrals over the corresponding elements in the grid $\Kh$. The space and corresponding norm is denoted by $\LSpace[p]{\K_h, \R^{d+1}}$.

We will use the shortcut notation $\Norm{\cdot}{\K}\doteq\LNorm[2]{\TT{f}}{\K}$ for the $L^2$-norm. The corresponding $L^2$-inner product will be denoted by $\Inner{\cdot}{\cdot}{\K}\doteq\Inner{\cdot}{\cdot}{\LSpace[2]{\K}}$. The limiting spaces $\LInfSpace{\K}$ and $\LInfSpace{\Kh}$ are defined by the essential supremum over the local tensor Frobenius norm.

We define Sobolev spaces over smooth tensor fields recursively. Therefore, we require that certain derivatives are in a corresponding $L^p$-space. Since covariant derivatives are defined for tangential tensor fields only, we introduce two different spaces with their norms, for tangential and non-tangential tensor fields, respectively.
We set $\SobolevSpace[0][p]{\K, \TensorBundle{M}}\coloneqq\LSpace[p]{\K, \TensorBundle{M}}$ for both, tangential and non-tangential tensor fields, with $M=\G$ or $M=\R^{d+1}$, respectively.

\begin{definition}[Sobolev spaces for tangential tensor fields]
The higher-order Sobolev spaces $\SobolevSpaceTan[s][p]{\K}\doteq\SobolevSpace[s][p]{\K, \TensorBundleTan}$, for $s>0$, of tangential tensor fields is defined as the set of functions $\TT{f}\in\SobolevSpace[s-1][p]{\K, \TensorBundleTan}$ with $s$-order covariant derivative $\CDer^{s}\TT{f}\in\LSpace[p]{\K, \TensorBundle[n+s]{\G}}$. The associated norm $\SobolevNormTan[s][p]{\cdot}{\K}$ is defined by
\[
  \SobolevNormTan[s][p]{\TT{f}}{\K}^p \coloneqq \sum_{j=0}^s \LNorm[p]{\CDer^{j} \TT{f}}{\K}^p\,.
\]
\end{definition}

Similarly, the Sobolev spaces for tensor fields in the ambient space are defined using the surface derivative $\DDer_\G$ instead of the covariant derivative $\CDer$.

\begin{definition}[Sobolev spaces for non-tangential tensor fields]
The Sobolev space $\SobolevSpaceAmb[s][p]{\K}\doteq\SobolevSpace[s][p]{\K, \TensorBundleAmb}$ of tensor fields in the ambient space is defined as the set of functions $\TT{f}\in\SobolevSpace[s-1][p]{\K, \TensorBundleAmb}$ with $s$-order Euclidean derivative $\DDer_\G^{s}\TT{f}\in\LSpace[p]{\K, \TensorBundleAmb[n+s]}$ and norm defined by
\[
  \SobolevNormAmb[s][p]{\TT{f}}{\K}^p \coloneqq \sum_{j=0}^s \LNorm[p]{\DDer_\G^{j} \TT{f}}{\K}^p\,.
\]
\end{definition}

For fields on the discrete surface $\Gh$, tangential tensors are only defined in the interior of the surface elements $K\in\Kh$. So are its (covariant) derivatives. The Sobolev spaces and corresponding norms on $\Gh$ are denoted by $\SobolevSpace[s][p]{\Kh, M}$ and are defined in terms of $\ChDer$ and $\DDer_{\Gh}$ for $M=\Gh$ and $M=\R^{d+1}$, respectively.

If the tensor fields on $\Gamma$ are continuous, we also use the Sobolev spaces $\SobolevSpace[s][p]{\G, M}$ defined by integrals over $\G$ directly. Since the integration over $\G$ can be decomposed into integration over the elements of $\K$, it follows that $\SobolevSpace[s][p]{\G, M}\subset\SobolevSpace[s][p]{\K, M}$.

We denote the spaces and norm with $p=2$ by $\myBold{H}^s_{\text{tan}}\doteq\myBold{W}^{s,2}_{\text{tan}}$ and $\myBold{H}^s\doteq\myBold{W}^{s,2}$.

\subsection{Tensor projections}\label{sec:tensor-projections}

\begin{lemma}[Estimates for tensor projections]\label{lem:P}
	Let $\PP$ and $\PPh$ are tensor projections for rank-$n$ tensor fields on $\G$ and $\Gh$, respectively. Then we have the following estimates for the projections of tensors
\begin{align*}
    \LInfNorm{\PP - \PPh}{\Kh}
      &\leqC  \LInfNorm{\nn-\nnh}{\Kh}\,,\\
    \LInfNorm{\PP - \PP\PPh\PP}{\Kh}
      &\leqC  \LInfNorm{\nn-\nnh}{\Kh}^{2}\,,\\
    \LInfNorm{\PPh - \PPh\PP\PPh}{\Kh}
      &\leqC  \LInfNorm{\nn-\nnh}{\Kh}^{2}\,.
  \end{align*}
\end{lemma}
\begin{proof}
	The first estimate follows by induction over the tensor rank $n$, from the local pointwise estimate
\[
    \FNorm{\P - \Ph}^{2} = \FNorm{\nn\otimes \nn - \nnh\otimes\nnh}^{2}
                       = 2(1-\FInner{\nn}{\nnh}^{2})
                    \leq 2\FNorm{\nn-\nnh}^{2}\,.
  \]
Let $\TT{f}^{(n)} = \TT{f}^{(n-1)}\otimes\TT{f}^{(1)}$ be a rank-$n$ tensor on $\Gh$ represented as tensor product of a rank-$(n-1)$ and a rank-$1$ tensor. Then we have the induction relation
\begin{align*}
		\PPh\TT{f}^{(n)} &= \PPh\TT{f}^{(n-1)}\otimes\Ph\TT{f}^{(1)} \\
		(\PP-\PPh)\TT{f}^{(n)} &= \PP\TT{f}^{(n)} - \PPh\TT{f}^{(n)} \\
		&= (\PP-\PPh)\TT{f}^{(n-1)}\otimes\P\TT{f}^{(1)} + \PPh\TT{f}^{(n-1)}\otimes(\P-\Ph)\TT{f}^{(1)}\,.
	\end{align*}
Using triangle inequality and extracting the operator norm of the projection gives by recursion the first result.

	The second estimate follows by induction from the equation for vector fields $\VV{v}$ on $\Gh$, i.e., $(\P - \P\Ph\P)\VV{v} = \langle \P\nnh,\VV{v}\rangle \P\nnh$ and the estimate
	\[
		\FNorm{\P\nnh} \leq \big(1+\frac{1}{2}\FNorm{\nn-\nnh}\big)\FNorm{\nn-\nnh}\,.
	\]
	We can interchange the roles of $\nn$ and $\nnh$ to obtain the last estimate.
\end{proof}

\begin{lemma}[Estimates for the derivatives of tensor projections]\label{lem:dP}
	Let $\PP$ and $\PPh$ are the tensor projections for rank-$n$ tensor fields on $\G$ and $\Gh$, respectively, properly extended to $U_\delta$. Then we have
\begin{align*}
    \LInfNorm{\DDer(\PP - \PPh)}{\Kh}
      &\leqC  h^{\kg-1}\,.
  \end{align*}
\end{lemma}
\begin{proof}
	For tangential projections $\P$ and $\Ph$ we can write locally
	\begin{align*}
		\DDer(\nn\otimes\nn - \nnh\otimes\nnh)
			&= [w_h^{ik} n_h^j + n_h^i w_h^{jk} - w^{ik} n^j - n^i w^{jk}]\mathbf{e}_{i}\otimes\mathbf{e}_{j}\otimes\mathbf{e}_{k} \\
			&= [w_h^{ik} (n_h^j - n^j) + (n_h^i - n^i) w_h^{jk} \\
			&\quad + (w_h^{ik} - w^{ik}) n^j + n^i (w_h^{jk} - w^{jk})]\mathbf{e}_{i}\otimes\mathbf{e}_{j}\otimes\mathbf{e}_{k}\,,
	\end{align*}
	with $\Weingarten=-\DDer\nn=w^{ij}\mathbf{e}_i\otimes\mathbf{e}_j$ and $\Weingarten_h=-\DDer\nnh=w_h^{ij}\mathbf{e}_i\otimes\mathbf{e}_j$ the continuous and discrete surface Weingarten maps.
	The estimate then follows from  \eqref{eq:n-nh} and \eqref{eq:w-wh}. For tensor projections, the assertion follows by induction over the tensor rank, see the proof of \cref{lem:P}.
\end{proof}

\subsection{Comparison of tensor fields on the discrete and continuous surface}\label{sec:l2-estimates-for-tensor fields}
The relation between tensor fields and its extensions and lifts in norms and inner products is given by the following Lemma.

\begin{lemma}[$L^2$-norm equivalence]\label{lem:l2-norm-equiv}
	Let $\TT{f},\TT{g}\in \LSpace[2]{\K,\TensorBundleAmb}$ be ambient tensor fields. Then we have the estimate.
	\begin{align*}
		\Abs{\Inner{\TT{f}}{\TT{g}}{\K} - \Inner{\TT{f}^{e}}{\TT{g}^{e}}{\Kh}}
		&\leqC h^{\kg+1}\Norm{\TT{f}^{e}}{\Kh}\Norm{\TT{g}^{e}}{\Kh}\,.
	\end{align*}
	For $h<h_0$ small enough this implies $L^2$-norm equivalence, i.e.,
	\begin{align*}
		\Norm{\TT{f}^{e}}{\Kh} \leqC \Norm{\TT{f}}{\K} \leqC \Norm{\TT{f}^{e}}{\Kh}\,.
	\end{align*}
\end{lemma}

\begin{proof}
	The estimate is analogous to the well-known estimate for the lifting of functions \cite{De2009Higher}.
	We write
	\begin{align*}
		\Abs{\Inner{\TT{f}}{\TT{g}}{\K} - \Inner{\TT{f}^{e}}{\TT{g}^{e}}{\Kh}}
		&= \Abs{\Inner{\Abs{\det(B)}\TT{f}^{e}}{\TT{g}^{e}}{\Kh} - \Inner{\TT{f}^{e}}{\TT{g}^{e}}{\Kh}}\\
		&\leq \LInfNorm{1-\Abs{\det(B)}}{\Kh}\Norm{\TT{f}^{e}}{\Kh}\Norm{\TT{g}^{e}}{\Kh}\,.
	\end{align*}
	The estimate follows from properties of $B$, see \cref{lem:B}.
\end{proof}

In the following we will often make use of the fact that while the discrete normal part of a tangential vector field does not vanish, it at least converges to zero with an order that is determined by the geometric approximation order $\kg$. For inner products this order of convergence is even one order better than for norm estimates which is proven as a so-called ``non-standard geometry estimate'' for vector fields in \cite{HL2020Analysis}.

\begin{lemma}[Standard and non-standard estimates for mixed projections]\label{lem:mixed_projections}
	For a tensor field $\TT{f}\in \LSpace[p]{\K, \TensorBundleAmb}$ we have for all $p\in[1,\infty]$
	\[
    \LNorm[p]{\QQh\PP \TT{f}}{\K} \leqC h^{k_g}\LNorm[p]{\PP \TT{f}}{\K}.
	\]
	For vector fields $\VV{v}\in\SobolevSpaceAmb[1][1]{\K, \TensorBundleAmb}$, we have the estimate
	\[
    \Inner{\VV{v}}{\P\nnh}{\K} \leqC h^{\kg+1}\SobolevNormAmb[1][1]{\VV{v}}{\K}\,.
	\]
	For tensor fields $\TT{f},\TT{g}\in \HSpaceAmb{\K, \TensorBundleAmb}$, we have
	\begin{align}\label{eq:very-evil-trick}
		\Inner{\QQh \TT{f}^e}{\PP\TT{g}^e}{\Kh}\leqC h^{k_g+1}\HNormAmb{\TT{f}}{\K}\HNormAmb{\TT{g}}{\K}.
	\end{align}
\end{lemma}
\begin{proof}
	The first estimate follows from \cref{lem:P} and orthogonality of $\QQh$ and $\PPh$. The second estimate is proven by \citet{HL2020Analysis}.
	The last estimate follows for vector fields $\VV{v}$ and $\VV{w}$ on $\G$ from
\begin{align*}
    \Inner{\Qh \VV{v}^e}{\P \VV{w}^e}{\Kh}
			&= \Inner{\P\nnh}{\P \FInner{\VV{v}^e}{\nnh}\VV{w}^e}{\Kh} \\
			&\leqC h^{\kg+1}\SobolevNormAmb[1][1]{\FInner{\VV{v}}{\nnh}\VV{w}}{\K} \\
			&\leqC h^{\kg+1}\HNormAmb{\VV{v}}{\K}\HNormAmb{\VV{w}}{\K}\,,
  \end{align*}
using the norm equivalence of \cref{lem:l2-norm-equiv}. For tensor fields the estimate follows by induction.
\end{proof}

\subsection{Comparison of covariant derivatives of tensor fields on the discrete and continuous surface}\label{sec:h1-estimates-for-tensor fields}

We now consider gradients of tensor fields. Given a non-tangential tensor field $\TT{f}$ on $\G$ we write out the tangential gradient of the tangential part of $\TT{f}$ by
\begin{align}
	\CDer\PP\TT{f} &= \PP\DDer\PP\TT{f}^{e} \\
								 &= \PP\DDer{\TT{f}}^{e} - \PP\DDer\QQ\TT{f}^{e} \notag
\end{align}

\begin{remark}
	Analogously, we introduce the derivative of the tangential part of tensor fields $\TT{f}_h$ on $\Gh$, locally inside the elements of the triangulation $K\in\mathcal{K}_h$, by
\begin{align*}
		\CovariantDerivative{K}\PPh\TT{f}_h
			&= \PPh\DDer\TT{f}_h^{e} - \PPh\DDer\QQh\TT{f}_h^{e} \\
			&= \PPh\DDer_K\TT{f}_h - \PPh\DDer_K\QQh\TT{f}_h\,.
	\end{align*}
The second form is typically used in surface finite element codes.
\end{remark}

\begin{lemma}\label{lem:pdq}
	For (non-tangential) tensor fields  $\TT{f}:\G\to\TensorBundleAmb$ the following estimate holds for the derivative of the normal part:
\begin{align}\label{eq:pdq}
		\Norm{\PP\DDer\QQ\TT{f}^{e}}{\K}
			\leqC \LInfNorm{\Weingarten}{\K}\Norm{\QQ\TT{f}}{\K} \,,
	\end{align}
	with $\Weingarten=-\DDer\nn=w^{ij}\mathbf{e}_i\otimes\mathbf{e}_j$ the continuous extended Weingarten map.
Interchanging the roles of $\QQ$ and $\PP$, we estimate analogously
\begin{align}\label{eq:qdp}
		\Norm{\QQ\DDer\PP\TT{f}^{e}}{\K}
			\leqC \LInfNorm{\Weingarten}{\K}\Norm{\PP\TT{f}}{\K} \,.
	\end{align}
\end{lemma}
\begin{proof}
	By the product rule, we have for a tensor
	\begin{align*}
		\PP\DDer\QQ\TT{f}^{e} = (\PP\DDer\QQ)\QQ\TT{f}^{e}.
	\end{align*}
	As the norm of $\DDer\Q$ is bounded, i.e., $\FNorm{\DDer \Q}\leqC\FNorm{\Weingarten}$, so are the norms of $\DDer\P$\EN{}{,} and by induction also the norms of $\DDer\PP$ and $\DDer\QQ$. Thus, the estimate \eqref{eq:pdq} follows. Analogously, we can write
	\begin{align*}
		\QQ\DDer\PP\TT{f}^{e} = (\QQ\DDer\PP)\PP\TT{f}^{e}.
	\end{align*}
\end{proof}

\begin{remark}\label{rem:pdq_h}
	We obtain on each element $K\in\mathcal{K}_h$ with similar arguments as in \cref{lem:pdq} the estimates for discrete projections
	\begin{align}
		\Norm{\PPh\DDer\QQh\TT{f}_h^{e}}{K}
		& \leqC \LInfNorm{\Weingarten_h}{K}\Norm{\QQh\TT{f}_h}{K}\,,\label{eq:pdq-disc}\\
		\Norm{\QQh\DDer\PPh\TT{f}_h^{e}}{K}
		& \leqC \LInfNorm{\Weingarten_h}{K}\Norm{\PPh\TT{f}_h}{K} \,,\label{eq:qdp-disc}
	\end{align}
	where $\Weingarten_h=-\DDer\nnh$ denotes the discrete Weingarten map on the elements of $\Gh$.
\end{remark}

\begin{remark}
	For tangential tensor fields $\TT{f}\in \HSpaceTan[1]{\K}$ \cref{lem:pdq} implies
	\begin{align*}
		\Norm{\DDer\TT{f}^{e}}{\K} & = \Norm{\PP\DDer\PP\TT{f}^{e}}{\K} +\Norm{\QQ\DDer\PP\TT{f}^{e}}{\K} \\
		&\leqC \Norm{\CDer\TT{f}}{\K} + \LInfNorm{\Weingarten}{\K}\Norm{\TT{f}}{\K}\\
		&\leqC \HNormTan[1]{\TT{f}}{\K}.
	\end{align*}
\end{remark}

For higher order derivatives a similar estimate holds:

\begin{lemma}\label{lem:gradient-estimate}
	For $\TT{f}\in\HSpaceTan[m]{\K}$ with $m\geq 1$, the ambient Sobolev norm is bounded by the tangential Sobolev norm:
	\begin{align*}
		\Norm{\DDer^m\TT{f}^e}{\K} = \Norm{\DDer_\G^m\TT{f}}{\K}
		&\leqC \HNormTan[m]{\TT{f}}{\K}\,.
	\end{align*}
\end{lemma}
\begin{proof}
	This is shown by \citet[Lemma 2.3]{HL2020Analysis}.
\end{proof}

The following Lemma gives estimates for the error in derivatives of extended and projected tensor fields.

\begin{lemma}[$H^{1}$-estimates]\label{lem:h1-estimates}
	Let $\TT{f} \in\HSpace[1]{\K, \TensorBundleAmb}$ be a non-tangential tensor field. Then we have
	\begin{align}\label{eq:h1-estimates}
		\Norm{\left(\CDer\PP\TT{f}\right)^{e} - \ChDer\PPh\TT{f}^{e}}{\Kh}\leqC h^{\kg}\left(\Norm{\DDer\TT{f}^e}{\Kh}+ \Norm{\PPh\TT{f}^{e}}{\Kh} + h^{-1}\Norm{\QQh\TT{f}^{e}}{\Kh}\right).
	\end{align}
\end{lemma}
\begin{proof}
We write almost everywhere
\begin{align*}
	\left(\CDer\PP\TT{f}\right)^{e} - \ChDer\PPh\TT{f}^{e}
	&= \PP\DDer\TT{f}^{e} - \PPh\DDer\TT{f}^{e} + \PPh\DDer\QQh\TT{f}^{e} - \PP\DDer\QQ\TT{f}^{e}\,.
\end{align*}
Thus, we can estimate
\begin{align*}
	\Norm{\left(\CDer\PP\TT{f}\right)^{e} - \ChDer\PPh\TT{f}^{e}}{\Kh} &
	\leqC h^{\kg}\Norm{\DDer\TT{f}^e}{\Kh}
	+ \Norm{\PPh\DDer\QQh\TT{f}^{e} - \PP\DDer\QQ\TT{f}^{e}}{\Kh}\,.
\end{align*}
As in the proof of \cref{lem:pdq} we write almost everywhere on $\Gh$
\if\isimastyle1
\begin{align*}
	\PPh\DDer\QQh\TT{f}^{e} - \PP\DDer\QQ\TT{f}^{e}
	&= (\PPh\DDer\QQh)\QQh\TT{f}^{e} - (\PP\DDer\QQ)\QQ\TT{f}^{e} \\
	&= (\PPh(\DDer\QQh - \DDer\QQ))\QQh\TT{f}^{e} + (\PPh\DDer\QQ)(\QQh-\QQ)\TT{f}^{e} +  ((\PPh-\PP)\DDer\QQ)\QQ\TT{f}^{e}.
\end{align*}
\else
\begin{align*}
&	\PPh\DDer\QQh\TT{f}^{e} - \PP\DDer\QQ\TT{f}^{e} \\
&		\quad= (\PPh\DDer\QQh)\QQh\TT{f}^{e} - (\PP\DDer\QQ)\QQ\TT{f}^{e} \\
&		\quad= (\PPh(\DDer\QQh - \DDer\QQ))\QQh\TT{f}^{e} + (\PPh\DDer\QQ)(\QQh-\QQ)\TT{f}^{e} +  ((\PPh-\PP)\DDer\QQ)\QQ\TT{f}^{e}.
\end{align*}
\fi
As $\Norm{\DDer\QQh - \DDer\QQ}{}\leqC h^{\kg-1}$, see \cref{lem:dP}, we obtain
\begin{align*}
	\Norm{\PPh\DDer\QQh\TT{f}^{e} - \PP\DDer\QQ\TT{f}^{e}}{\Kh}
	& \leqC h^{\kg-1}\Norm{\QQh\TT{f}^{e}}{\Kh} + h^{\kg}\Norm{\TT{f}^{e}}{\Kh}\\
	& \leqC h^{\kg} \left(h^{-1}\Norm{\QQh\TT{f}^{e}}{\Kh} + \Norm{\PPh\TT{f}^{e}}{\Kh}\right).
\end{align*}
\end{proof}

\begin{remark}\label{rem:h1-norm-equivalence}
	In analogy to the estimate in \cref{lem:h1-estimates}, we obtain
	\begin{align*}
		\Norm{\left(\CDer\PP\TT{f}\right)^{e} - \ChDer\PPh\TT{f}^{e}}{\Kh}\leqC h^{\kg}\left(\Norm{\DDer\TT{f}^{e}}{\K} + \Norm{\PP\TT{f}}{\K} + h^{-1}\Norm{\QQ\TT{f}}{\K}\right).
	\end{align*}
	with norms on $\G$.
	Thus, we immediately see that for tangential tensor fields $\TT{f}\in \HSpaceTan[1]{\K}$
	\begin{align*}
		\Norm{\left(\CDer\PP\TT{f}\right)^{e} - \ChDer\PPh\TT{f}^{e}}{\Kh}\leqC h^{\kg}\HNormTan[1]{\TT{f}}{\K}.
	\end{align*}
\end{remark}

We introduce a new mixed norm motivated by the structure of the right hand side of~\eqref{eq:h1-estimates}.

\begin{definition}\label{def:GNorm}
	For $\TT{f}_h \in\HSpace[1]{\Kh, \TensorBundleAmb}$ we define the norm
	\begin{align*}
		\GNorm{\TT{f}_h} \colonequals \Norm{\CovariantDerivative{\Gh}\PP_h\TT{f}_h}{\Kh} + \Norm{\DDer\TT{f}_h^e}{\Kh}+ \Norm{\PPh\TT{f}_h}{\Kh} + h^{-1}\Norm{\QQh\TT{f}_h}{\Kh}.
	\end{align*}
\end{definition}

\begin{remark}
	With \cref{rem:pdq_h} we have for $\TT{f}_h \in\HSpace[1]{\Kh, \TensorBundleAmb}$ the estimate
	\begin{align}\label{eq:gnorm-to-h1}
		\GNorm{\TT{f}_h} \leqC \HNormAmb{\TT{f}_h}{\Kh} + h^{-1}\Norm{\QQh\TT{f}_h}{\Kh}.
	\end{align}
From \cref{lem:h1-estimates} it follows that for $\TT{f} \in\HSpace[1]{\K, \TensorBundleAmb}$
	\begin{align}\label{eq:cder-to-gnorm-est}
		\Norm{\CDer\PP\TT{f}}{\K}\leqC \GNorm{\TT{f}^{e}}.
	\end{align}
For tangential tensor fields $\TT{f} \in\HSpaceTan[1]{\K}$ we have
	\begin{align}\label{eq:gnorm-to-h1tan}
		\GNorm{\TT{f}^{e}}\leqC \HNormTan[1]{\TT{f}}{\K}.
	\end{align}
\end{remark}

\section{Finite element \EN{discretisation}{discretization} of tensor equation}\label{sec:finite-element-discretization-of-tensor-equation}

In order to describe approximations of the function space $\HSpaceTan[1]{\K}$, we introduce discrete finite element spaces in the following. A family of finite element spaces on $\Gh$ is based on lifting functions in $\hat{V}_h^m$ from $\hat{\G}_h$ to the \EN{parametrised}{parametrized} surface $\Gh$, i.e.,
\begin{equation}\label{eq:scalar-lagrange-space}
  V_{h}^m \doteq V_{h}^m(\Gh,\R) \coloneqq \Set{ \tilde{v}\in C^0(\Gh) \mid \tilde{v} = \hat{v}\circ\pi_{h}^{-1},\text{ for some }\hat{v}\in \hat{V}_h^m}\,.
\end{equation}

Similarly to the space $\hat{V}_h^m(\hat{\G}_h,\R^N)$ for vector-valued maps on $\hat{\G}_h$, see \eqref{eq:vector-flat-lagrange-space}, we introduce the finite element space for tensor-valued functions on $\Gh$ as product space over $V_{h}^m$ by
\begin{equation}\label{eq:tensor-lagrange-space}
  \Vh[h]^m\doteq V_{h}^m(\Gh,\TensorBundleAmb)\coloneqq \TensorProductSpace{V_{h}^m}{N}\,,
\end{equation}
with $N=(d+1)^n$.
Since each component is continuous and locally differentiable, we have
\[
  \Vh[h]^m\subset \HSpace{\Gh, \TensorBundleAmb}\cap \TT{C}^0(\Gh, \TensorBundleAmb)\,.
\]
In the following, we use the shortcut $\Vh[h]\doteq\Vh[h]^{\ku}$ to denote the tensor finite element space of local polynomial order $\ku$ on $\Gh$.

We introduce a tensor interpolation operators $\Ih$ and $\TT{I}$ as the usual Lagrange interpolation for continuos tensor fields, by using a componentwise scalar Lagrange interpolation operator. For a continuous scalar field on $\hat{\G}_h$, we have the classical Lagrange interpolation $\hat{I}_h^m:C^0(\hat{\G}_h)\to\hat{V}_h^m$, see, e.g., \cite{DE2013Finite}. This induces a lifted scalar interpolation $I_{h}^m:C^0(\Gh)\to V_{h}^m$ for scalar fields $\tilde{v}\in C^0(\Gh)$ by
\begin{equation}\label{eq:scalar-interpolation-operator}
	(I_{h}^m \tilde{v})(\pi_h(x)) = (\hat{I}_h^m \tilde{v}\circ \pi_{h})(x)\text{ for }x\in\hat{\G}_h\,.
\end{equation}

A componentwise scalar interpolation operators defines the tensor Lagrange interpolation $\hat{\TT{I}}_h^m$ and $\TT{I}_h^m$. We write the corresponding interpolation operator for tensor fields $\TT{v}\in\TT{C}^0(\G,\TensorBundleAmb)$ on the continuous surface $\Gamma$ by
\begin{equation}\label{eq:tensor-interpolation-operator}
	(\TT{I}^m\TT{v})^e \coloneqq \TT{I}_{h}^m\TT{v}^e\,,\text{ and }
	(\hat{\TT{I}}^m\TT{v})^e \coloneqq \hat{\TT{I}}_{h}^m\TT{v}^e\,,
\end{equation}
respectively.

A discrete formulation of the continuous variational problem \eqref{eq:continuous_model_problem} can be obtained by approximating the surface $\G$ by $\Gh$, the function space $\HSpaceTan[1]{\K}$ by $\Vh$\EN{}{,} and by projecting tensor fields in the ambient space to the tangent space. The tangentiality of the solution can be obtained by introducing an additional \EN{penalisation}{penalization} term $s_h$ into the bilinear form.

The resulting discrete variational problem reads: Find $\TT{u}_h\in\Vh$, such that,
\begin{equation}\label{eq:discrete_variational_problem}
	A_h(\TT{u}_h, \TT{v}_h) = l_h(\TT{v}_h),\quad\forall\TT{v}_h\in\Vh\,,
\end{equation}
with discrete forms given by
\begin{align*}
	A_h(\TT{u}_h, \TT{v}_h) &\colonequals a_h(\TT{u}_h, \TT{v}_h) + s_h(\TT{u}_h, \TT{v}_h)\,,\\
	a_h(\TT{u}_h, \TT{v}_h) &\colonequals \Inner{\ChDer\PPh\TT{u}_h}{\ChDer\PPh\TT{v}_h}{\Kh}\! + \Inner{\PPh\TT{u}_h}{\PPh\TT{v}_h}{\Kh}\,,\\
	s_h(\TT{u}_h, \TT{v}_h) &\colonequals \beta h^{-2\alpha}\Inner{\QQhTilde\TT{u}_h}{\QQhTilde\TT{v}_h}{\Kh}\,, \\
	l_h(\TT{v}_h)					&\colonequals \Inner{\TT{f}^e}{\PPh\TT{v}_h}{\Kh}\,,
\end{align*}
where $\alpha\in[0,1]$ and $\beta>0$ are \EN{penalisation}{penalization} parameters. The normal projection operator $\QQhTilde$ is defined over a normal vector $\nnhTilde$ that \EN{fulfils}{fulfills} a (possibly) higher-order approximation of the exact surface normal $\nn$, i.e.,
\begin{equation}\label{eq:better-normal}
	\LNorm[\infty]{\nn - \nnhTilde}{\Kh}	\leq C\; h^{\kp}
\end{equation}
with $\kp\geq\kg$ a penalty-term integer order.

\begin{remark}\label{rem:better-normal}
	The normal $\nnhTilde$ could be either given as an interpolation of the exact surface normal, i.e., $\nnhTilde=I_{h,\kg}^{\kp-1}\nn$, or be given as the normal of a discrete surface $\G_{h,\kp}$ approximating $\G$ with higher order, extended to $U_\delta$. We will also \EN{analyse}{analyze} the case $\kp=\kg$ that is especially relevant if you just have a discrete surface given and thus are forced to use $\nnhTilde=\nnh$.

	For the case $\kg=1$ improved normals can be constructed using normal-vector averages or surface reconstruction, cf. \cite{MW2000Surface,XX2005Convergence}, or $L^2$-projections of piecewise constant normal fields into the space $V_h^1$ of piecewise linear functions, cf. \cite{Fritz2013Finite,CH2020Finite}.
\end{remark}

\section{Error analysis}\label{sec:error-analysis}
In this section we prove different types of error estimates culminating in \cref{sec:discretization-errors} in the \EN{discretisation}{discretization} error estimates described in \cref{sec:main-results}. Following the classical procedure we begin with interpolation-errors estimates in \cref{sec:interpolation-errors}. The following up \cref{sec:geometric-errors} discusses differences in discrete and continuous bilinear and linear forms that emerge from the geometry approximation. These are prerequisites for the proofs in \cref{sec:discretization-errors}. The estimates in this final subsection are formulated in various norms. Among these are the energy norms of the associated problems that are defined as follows.

For a bilinear form $\Omega$, we define an associated energy-norm $\EnergyNorm{\cdot}{\Omega}$ by $\EnergyNorm{\TT{v}}{\Omega}^2\colonequals\Omega(\TT{v},\TT{v})$. The discrete energy norm $\EnergyNorm{\cdot}{A_h}$ can thereby be split into parts:
\begin{align*}
	\EnergyNorm{\TT{v}_h}{A_h}^2 \colonequals& A_{h}(\TT{v}_h, \TT{v}_h) = a_{h}(\TT{v}_h, \TT{v}_h) + s_{h}(\TT{v}_h, \TT{v}_h)\\
	=& \underbrace{\Norm{\ChDer\PPh\TT{v}_h}{\Kh}^2 + \Norm{\PPh\TT{v}_h}{\Kh}^2}_{\EnergyNorm{\TT{v}_h}{a_h}^2} + \underbrace{\beta h^{-2\alpha}\Norm{\QQhTilde\TT{v}_h}{\Kh}^2}_{\EnergyNorm{\TT{v}_h}{s_h}^2}
\end{align*}
and the continuos energy norm is written for  $\TT{v}\in\HSpaceAmb{\K}$ using projection operators:
\[
  \EnergyNorm{\TT{v}}{a}^2 \colonequals a(\TT{v}, \TT{v}) \colonequals \Norm{\CDer\PP\TT{v}}{\K}^2 + \Norm{\PP\TT{v}}{\K}^2\,.
\]
For $\TT{v}\in\HSpaceTan{\K}$ this corresponds to the tangential $H^1$-norm.

\subsection{Interpolation errors}\label{sec:interpolation-errors}

Interpolation-error estimates can be deduced from componentwise estimates for the scalar interpolation operator, cf. \cite{De2009Higher}.

\begin{lemma}\label{lem:interpolation-error-estimates-amb}
	For non-tangential tensor fields $\TT{f}\in\HSpaceAmb[m+1]{\K}\cap\TT{C}^0(\G)$ with $m\geq 1$ we have the following interpolation error estimates for $\TT{I}^k\TT{f}=(\TT{I}_h^k\TT{f}^e)^\ell$ the tensor Lagrange interpolation with $0\leq l\leq m+1\leq k+1$,
	\begin{align}
		\HNormAmb[l]{\TT{f} - \TT{I}^k\TT{f}}{\K}
			&\leqC h^{m+1-l}\HNormAmb[m+1]{\TT{f}}{\K}\,, \label{eq:interpolation_L2} \\
		\HNormAmb[1]{\PP(\TT{f} - \TT{I}^k\TT{f})}{\K}
			&\leqC h^{m}\HNormAmb[m+1]{\TT{f}}{\K}\,, \label{eq:interpolation_H1}
	\end{align}
\end{lemma}
\begin{proof}
	We have the requirement $1\leq d \leq 3$ for the dimension of the geometry and $m\geq 1$. Thus, the first statement follows from standard scalar local interpolation error estimates, applied componentwise and elementwise, cf. \cite[Thm. 3.1.6]{Ciarlet1978The}.

	The second statement can be brought into a compatible form:
	\if\isimastyle1
	\begin{align*}
		\HNormAmb[1]{\PP(\TT{f} - \TT{I}^m\TT{f})}{\K}
			&\leq \Norm{\PP(\TT{f} - \TT{I}^m\TT{f})}{\K} + \Norm{\DDer_{\G}\PP(\TT{f} - \TT{I}^m\TT{f})}{\K} \\
			&\leqC \Norm{\TT{f} - \TT{I}^m\TT{f}}{\K}
					 + \Norm{\DDer(\TT{f} - \TT{I}^m\TT{f})^e}{\K}
					 + \Norm{\DDer\QQ(\TT{f} - \TT{I}^m\TT{f})^e}{\K} \\
			&\leqC h^{m+1}\HNormAmb[m+1]{\TT{f}}{\K}
									+ h^{m}\HNormAmb[m+1]{\TT{f}}{\K}
									+ \LInfNorm{\Weingarten}{\K} \Norm{\QQ(\TT{f} - \TT{I}^m\TT{f})}{\K} \\
			&\leqC \left(h^{m+1} + h^{m}\right)\HNormAmb[m+1]{\TT{f}}{\K}
	\end{align*}
	\else
	\begin{multline*}
		\HNormAmb[1]{\PP(\TT{f} - \TT{I}^m\TT{f})}{\K}
			\leq \Norm{\PP(\TT{f} - \TT{I}^m\TT{f})}{\K} + \Norm{\DDer_{\G}\PP(\TT{f} - \TT{I}^m\TT{f})}{\K} \\
			\begin{aligned}
			&\leqC \Norm{\TT{f} - \TT{I}^m\TT{f}}{\K}
					 + \Norm{\DDer(\TT{f} - \TT{I}^m\TT{f})^e}{\K}
					 + \Norm{\DDer\QQ(\TT{f} - \TT{I}^m\TT{f})^e}{\K} \\
			&\leqC h^{m+1}\HNormAmb[m+1]{\TT{f}}{\K}
									+ h^{m}\HNormAmb[m+1]{\TT{f}}{\K}
									+ \LInfNorm{\Weingarten}{\K} \Norm{\QQ(\TT{f} - \TT{I}^m\TT{f})}{\K} \\
			&\leqC \left(h^{m+1} + h^{m}\right)\HNormAmb[m+1]{\TT{f}}{\K}
		\end{aligned}
	\end{multline*}
	\fi
using \eqref{eq:interpolation_L2}, an argument following \cref{lem:pdq} to extract the Weingarten map, boundedness of the Weingarten map\EN{}{,} and componentwise scalar estimates.
\end{proof}

\begin{lemma}\label{lem:interpolation-error-estimates}
	For tangential tensor fields $\TT{f}\in\HSpaceTan[m+1]{\K}\cap\TT{C}^0(\G)$ with $m\geq 1$ we have the following interpolation error estimates for $\TT{I}^m\TT{f}=(\TT{I}_h^m\TT{f}^e)^\ell$ the tensor Lagrange interpolation,
	\begin{align}
		\Norm{\TT{f} - \TT{I}^m\TT{f}}{\K}
			&\leqC h^{m+1}\HNormTan[m+1]{\TT{f}}{\K}\,, \label{eq:interpolation_tan_L2} \\
		\HNormTan[1]{\PP(\TT{f} - \TT{I}^m\TT{f})}{\K}
			&\leqC h^{m}\HNormTan[m+1]{\TT{f}}{\K}\,, \\
		\EnergyNorm{\TT{f}^e - \TT{I}_h^m\TT{f}^e}{A_h}
			&\leqC h^{m}\HNormTan[m+1]{\TT{f}}{\K}\,,\\
		\GNorm{\TT{f}^e - \TT{I}_h^m\TT{f}^e}
			&\leqC h^{m}\HNormTan[m+1]{\TT{f}}{\K}\,.
	\end{align}
\end{lemma}
\begin{proof}
	The first statement follows from \eqref{eq:interpolation_L2} using \cref{lem:gradient-estimate} to obtain the estimate w.r.t. the tangential Sobolev norm.

	The second statement can be brought into a compatible form similar to the proof of \eqref{eq:interpolation_H1}
	\if\isimastyle1
	\begin{align*}
		\HNormTan[1]{\PP(\TT{f} - \TT{I}^m\TT{f})}{\K}
			&\leq \Norm{\PP(\TT{f} - \TT{I}^m\TT{f})}{\K} + \Norm{\CDer\PP(\TT{f} - \TT{I}^m\TT{f})}{\K} \\
			&\leqC \Norm{\TT{f} - \TT{I}^m\TT{f}}{\K}
					 + \Norm{\PP\DDer(\TT{f} - \TT{I}^m\TT{f})^e}{\K}
					 + \Norm{\PP\DDer\QQ(\TT{f} - \TT{I}^m\TT{f})^e}{\K} \\
			&\leqC h^{m+1}\HNormTan[m+1]{\TT{f}}{\K}
									+ \Norm{\DDer(\TT{f} - \TT{I}^m\TT{f})^e}{\K}
									+ \LInfNorm{\Weingarten}{\K} \Norm{\QQ(\TT{f} - \TT{I}^m\TT{f})}{\K} \\
			&\leqC \left(h^{m+1} + h^{m}\right)\HNormTan[m+1]{\TT{f}}{\K}
	\end{align*}
	\else
	\begin{multline*}
		\HNormTan[1]{\PP(\TT{f} - \TT{I}^m\TT{f})}{\K}
			\leq \Norm{\PP(\TT{f} - \TT{I}^m\TT{f})}{\K} + \Norm{\CDer\PP(\TT{f} - \TT{I}^m\TT{f})}{\K} \\
			\begin{aligned}
			&\leqC \Norm{\TT{f} - \TT{I}^m\TT{f}}{\K}
					 + \Norm{\PP\DDer(\TT{f} - \TT{I}^m\TT{f})^e}{\K}
					 + \Norm{\PP\DDer\QQ(\TT{f} - \TT{I}^m\TT{f})^e}{\K} \\
			&\leqC h^{m+1}\HNormTan[m+1]{\TT{f}}{\K}
									+ \Norm{\DDer(\TT{f} - \TT{I}^m\TT{f})^e}{\K}
									+ \LInfNorm{\Weingarten}{\K} \Norm{\QQ(\TT{f} - \TT{I}^m\TT{f})}{\K} \\
			&\leqC \left(h^{m+1} + h^{m}\right)\HNormTan[m+1]{\TT{f}}{\K}
		\end{aligned}
	\end{multline*}
	\fi
using \eqref{eq:interpolation_tan_L2}.

	For the third estimate we use $L^2$-norm equivalence, the discrete $\PPh\DDer\QQh$ estimate \eqref{eq:pdq-disc}\EN{}{,} and similar arguments as before, to obtain
	\if\isimastyle1
	\begin{align*}
		\EnergyNorm{\TT{f}^e - \TT{I}_h^m\TT{f}^e}{A_h}
			&\leqC \Norm{\ChDer\PPh(\TT{f} - \TT{I}^m\TT{f})^e}{\Kh}\!+ \Norm{\PPh(\TT{f} - \TT{I}^m\TT{f})^e}{\Kh}\!+ h^{-\alpha}\Norm{\QQhTilde(\TT{f} - \TT{I}^m\TT{f})^e}{\Kh} \\
			&\leqC \Norm{\PPh\DDer(\TT{f} - \TT{I}^m\TT{f})^e}{\Kh}
										+ \Norm{\QQh(\TT{f} - \TT{I}^m\TT{f})^e}{\Kh}
				+ \Norm{(\TT{f} - \TT{I}^m\TT{f})^e}{\Kh} + h^{-\alpha}\Norm{(\TT{f} - \TT{I}^m\TT{f})^e}{\Kh} \\
			&\leqC \big(h^{m} + h^{m+1} + h^{m+1-\alpha}\big)\HNormTan[m+1]{\TT{f}}{\K}\,.
	\end{align*}
	\else
	\begin{multline*}
		\EnergyNorm{\TT{f}^e - \TT{I}_h^m\TT{f}^e}{A_h}\\
		\begin{aligned}
			&\leqC \Norm{\ChDer\PPh(\TT{f} - \TT{I}^m\TT{f})^e}{\Kh}\!+ \Norm{\PPh(\TT{f} - \TT{I}^m\TT{f})^e}{\Kh}\!+ h^{-\alpha}\Norm{\QQhTilde(\TT{f} - \TT{I}^m\TT{f})^e}{\Kh} \\
			&\leqC \Norm{\PPh\DDer(\TT{f} - \TT{I}^m\TT{f})^e}{\Kh}
										+ \Norm{\QQh(\TT{f} - \TT{I}^m\TT{f})^e}{\Kh}\\
			&\qquad\qquad	+ \Norm{(\TT{f} - \TT{I}^m\TT{f})^e}{\Kh} + h^{-\alpha}\Norm{(\TT{f} - \TT{I}^m\TT{f})^e}{\Kh} \\
			&\leqC \big(h^{m} + h^{m+1} + h^{m+1-\alpha}\big)\HNormTan[m+1]{\TT{f}}{\K}\,.
		\end{aligned}
	\end{multline*}
	\fi
With $0\leq\alpha\leq 1$ the assertion follows.

	The last estimate for the norm $\GNorm{\cdot}$, introduced in \cref{def:GNorm}, follows the argumentation as before for the norm $\EnergyNorm{\cdot}{A_h}$, with $\alpha=1$ and $\QQhTilde$ replaced by $\QQh$. The additional derivative term is already contained in the estimate of the other terms.
\end{proof}

\begin{lemma}[Inverse estimates]\label{lem:inverse-estimates}
	Let $\TT{v}_h\in\Vh^{\ku}$  and $0\leq l \leq m\leq \ku+1$ then
	\begin{align*}
		\HNormAmb[m]{\TT{v}_h^e}{\K}\leqC h^{-(m-l)} \HNormAmb[l]{\TT{v}_h^e}{\K}.
	\end{align*}
\end{lemma}
\begin{proof}
	Similar to the proof of interpolation estimates, the assertion follows directly from the inverse estimates of the componentwise scalar interpolation operator and the boundedness of the piecewise polynomial surface \EN{parametrisation}{parametrization} and its inverse, see \cref{lem:pi_h-pi_hinv}.
\end{proof}

\begin{lemma}\label{lem:inverse-pp-pph}
	Let $\TT{v}_h\in\Vh^{\ku}$ then
	\begin{align*}
		\Norm{\DDer(\PPh-\PP)\TT{v}_h}{\Kh}\leqC h^{\kg-1}\Norm{\TT{v}_h}{\Kh}.
	\end{align*}
\end{lemma}
\begin{proof}
	This follows from \cref{lem:dP} and \cref{lem:inverse-estimates}.
\end{proof}

\begin{lemma}[Super-approximation]\label{lem:superApproximation}
	Let $\TT{v}\in\HSpaceTan[m]{\K}\cap \TT{C}^0(\G)$. Then for $\ku\geq m$
	\begin{align}\label{eq:super-approx-1}
		\Norm{\QQ \I^{\ku}\TT{v}}{\K}\leqC  h^{m+1} \HNormAmb[m]{\I^{\ku}\TT{v}}{\K},
	\end{align}
	where the constant depends on up to $\ku+1$ order derivatives of $\QQ$.

	For $\TT{v}_h\in \Vh^{\ku}$, $l=0,1$ and $s\leq \ku$ we have
	\begin{align}\label{eq:super-approx-2}
		\Norm{\DDer^{l}(\PP \TT{v}_h^{e} - \I^{\ku} \PP \TT{v}_h^{e})}{\K}\leqC h^{s+1-l} \HNormAmb[s]{\TT{v}_h^{e}}{\K},
	\end{align}
	where the constant depends on up to $\ku+1$ order derivatives of $\PP$. The same estimate holds for $\QQ$ instead of $\PP$.
\end{lemma}
\begin{proof}
	As $\QQ \I^{\ku}\TT{v}$ is continuous, interpolation is well-defined, we have
	\[\I^{\ku} \QQ \I^{\ku}\TT{v}=0\]
	as the function vanishes at every Lagrange node. $\QQ \I^{\ku}\TT{v}$ is not globally smooth, but on every element $K\in \mathcal{K}$ we can estimate
	\begin{align*}
		\Norm{\QQ \I^{\ku}\TT{v}}{K}
		&= \Norm{\QQ \I^{\ku}\TT{v} - \I^{\ku} \QQ \I^{\ku}\TT{v}}{K}\\
		&\leqC h^{\ku+1} \HNorm[\ku+1]{\QQ (\I^{\ku}\TT{v})^{e}}{K}\\
		&\leqC h^{\ku+1} \Norm{\QQ}{C^{\ku+1}} \HNorm[\ku+1]{(\I^{\ku}\TT{v})^{e}}{K}.
	\end{align*}
	As derivatives of order $\ku+1$ of $(\I^{\ku}\TT{v})^{e}$ vanish we can use inverse estimates to obtain
	\begin{align*}
		\Norm{\QQ \Ih^{\ku}\TT{v}}{K}
		&\leqC  h^{\ku+1} \HNorm[\ku]{(\Ih^{\ku}\TT{v})^{e}}{K}\\
		&\leqC  h^{m+1} \HNorm[m]{(\Ih^{\ku}\TT{v})^{e}}{K}.
	\end{align*}
	Summation over the elements yield the first estimate.

	The second estimate follows from a similar argument.
	We have by interpolation, vanishing of the highest order derivatives\EN{}{,} and inverse estimates on each element
	\begin{align*}
		\Norm{\DDer^l(\PP \TT{v}_h^{e} - \I^{\ku} \PP \TT{v}_h^{e})}{K}
		&\leqC h^{\ku+1-l} \HNormAmb[\ku+1]{\PP \TT{v}_h^{e}}{K}\\
		&\leqC h^{\ku+1-l} \Norm{\PP}{C^{\ku+1}} \HNormAmb[\ku+1]{\TT{v}_h^{e}}{K}\\
		&\leqC h^{\ku+1-l} \HNormAmb[\ku]{\TT{v}_h^{e}}{K}\\
		&\leqC h^ {1-l+s} \HNormAmb[s]{\TT{v}_h^{e}}{K}.
	\end{align*}
\end{proof}

We introduce inverse estimates for discrete projections $\PPh\TT{v}_h$ and $\QQh\TT{v}_h$ of discrete tensor fields $\TT{v}_h\in \Vh$ in the following lemma.

\begin{lemma}\label{lem:super-inverse}
	For $\TT{v}_h\in \Vh^{\ku}$, $s\leq l=0,1$, we have
	\begin{align*}
		\Norm{\DDer^{l}\PPh\TT{v}_h^{e}}{\Kh} \leqC h^{-l+s}\HNormAmb[s]{\PPh\TT{v}_h^e}{\K} + h^{-l+1}\Norm{\TT{v}_h}{\Kh}\,.
	\end{align*}
	The same holds for $\QQh$ instead of $\PPh$.
\end{lemma}

\begin{proof}
	We write
	\begin{align*}
		\Norm{\DDer^l\PPh\TT{v}_h^{e}}{\Kh}\leq
		\Norm{\DDer^l(\PPh-\PP)\TT{v}_h^{e}}{\Kh}
		+ \Norm{\DDer^l\left(\PP\TT{v}_h^{e} - \Ih^{\ku}\PP\TT{v}_h^{e}\right)}{\Kh}
		+ \Norm{\DDer^l\Ih^{\ku}\PP\TT{v}_h^{e}}{\Kh}\,.
	\end{align*}
	The first term can be estimated by \cref{lem:P} for $l=0$ or \cref{lem:inverse-pp-pph} for $l=1$
	\begin{align*}
		\Norm{\DDer^l(\PPh-\PP)\TT{v}_h^{e}}{\Kh}
		\leqC h^{\kg-l}\Norm{\TT{v}_h}{\Kh}\leq h^{-l+1} \Norm{\TT{v}_h}{\Kh}\,.
		\end{align*}
	For the second term we use the super-approximation result \eqref{eq:super-approx-2} with $s=0$ to obtain
	\begin{align*}
		\Norm{\DDer^l\left(\PP\TT{v}_h^{e} - \Ih^{\ku}\PP\TT{v}_h^{e}\right)}{\Kh}
		\leqC h^{-l+1} \Norm{\TT{v}_h}{\Kh}\,.
	\end{align*}
	For the last term, we use inverse estimates and also super-approximation \eqref{eq:super-approx-2}
	\begin{align*}
		\Norm{\DDer^l\Ih^{\ku}\PP\TT{v}_h^{e}}{\Kh}
		&\leqC h^{-(l-s)}\HNormAmb[s]{\Ih^{\ku}\PP\TT{v}_h}{\Kh}\\
		&\leqC h^{-(l-s)}\left(\HNormAmb[s]{\PP\TT{v}_h}{\Kh} + \HNormAmb[s]{\Ih^{\ku}\PP\TT{v}_h - \PP\TT{v}_h}{\Kh}\right) \\
		&\leqC h^{-(l-s)}\left(\HNormAmb[s]{\PP\TT{v}_h}{\Kh} + h^{-s+1} \Norm{\TT{v}_h}{\Kh}\right) \\
		&\leqC h^{-(l-s)}\HNormAmb[s]{\PP\TT{v}_h}{\Kh}+h^{-l+1} \Norm{\TT{v}_h}{\Kh}\,.
	\end{align*}
	By \cref{lem:inverse-pp-pph} we get the assertion
	\begin{align*}
		\HNormAmb[s]{\PP\TT{v}_h}{\Kh}
		&\leqC \HNormAmb[s]{\PPh\TT{v}_h}{\Kh} +\HNormAmb[s]{(\PPh-\PP)\TT{v}_h}{\Kh}\\
		&\leqC \HNormAmb[s]{\PPh\TT{v}_h}{\Kh} +h^{\kg-s}\Norm{\TT{v}_h}{\Kh}.
	\end{align*}
\end{proof}

The following lemma allows us to estimate the norm $\GNorm{\cdot}$ introduced in \cref{def:GNorm} of discrete tensor fields in terms of the discrete energy norm.

\begin{lemma}\label{lem:disc-h1-bounds}
	For $\TT{v}_h\in \LSpace{\Gh}$, $h\in(0,h_{0}]$ with $h_{0}$ small enough, we have
	\begin{align}
		\label{eq:disc-h1-bounds-1}
		\Norm{\QQh\TT{v}_h}{\Kh}
		&\leqC h^{\kp}\Norm{\PPh\TT{v}_h}{\Kh}+\;h^{\alpha}\EnergyNorm{\TT{v}_h}{s_h}.
	\end{align}
	If $\TT{v}_h\in\Vh$, then
	\begin{align}
		\label{eq:disc-h1-bounds-2}
		\GNorm{\TT{v}_h}&\leqC \EnergyNorm{\TT{v}_h}{a_h} + h^{-1}\Norm{\QQh\TT{v}_h}{\Kh} \leqC \EnergyNorm{\TT{v}_h}{a_h} + h^{-1+\alpha}\EnergyNorm{\TT{v}_h}{s_h}.
	\end{align}
\end{lemma}
\begin{proof}
	For the first estimate, we introduce the normal projection $\QQhTilde$ and get
	\begin{align*}
		\Norm{\QQh\TT{v}_h}{\Kh}
		&\leq \Norm{\QQhTilde\TT{v}_h}{\Kh} + \Norm{\QQh\TT{v}_h-\QQhTilde\TT{v}_h }{\Kh} \\
		&\leq \Norm{\QQhTilde\TT{v}_h}{\Kh} + h^{\kp}\Norm{\TT{v}_h }{\Kh} \\
		&\leq \Norm{\QQhTilde\TT{v}_h}{\Kh} +h^{\kp}\Norm{\PPh\TT{v}_h }{\Kh}+ h^{\kp}\Norm{\QQh\TT{v}_h }{\Kh}\,.
\intertext{Thus, for $h\leq h_0$ small enough}
\Norm{\QQh\TT{v}_h}{\Kh}
		&\leqC h^{\kp}\Norm{\PPh\TT{v}_h }{\Kh} + \Norm{\QQhTilde\TT{v}_h}{\Kh}\,.
	\end{align*}
	For the second estimate, we consider only discrete maps $\TT{v}_h\in \Vh$. We note that it is enough to estimate $\Norm{\DDer\TT{v}_h^{e}}{\Kh}$.
	This can be done by first splitting $\TT{v}_h$ into tangential and normal parts
	\begin{align*}
		\Norm{\DDer\TT{v}_h^{e}}{\Kh}&\leq\Norm{\DDer\PPh\TT{v}_h^{e}}{\Kh} + \Norm{\DDer\QQh\TT{v}_h^{e}}{\Kh}\,.
	\end{align*}
	Using \eqref{eq:qdp-disc} we then obtain for the first term
	\begin{align*}
		\Norm{\DDer\PPh \TT{v}_h^{e}}{\Kh}&= \Norm{\ChDer\PPh \TT{v}_h^{e}}{\Kh} + \Norm{\QQh\DDer\PPh \TT{v}_h^{e}}{\Kh}\\
		&\leq \Norm{\ChDer\PPh \TT{v}_h}{\Kh} +C\;\Norm{\PPh\TT{v}_h}{\Kh}\\
		&\leqC \EnergyNorm{\TT{v}_h}{a_h}.
	\end{align*}
	For the second term, we use \cref{lem:super-inverse} with $l=1$, $s=0$\EN{}{,} and $\QQh$ instead of $\PPh$
	\begin{align*}
		\Norm{\DDer\QQh\TT{v}_h^{e}}{\Kh}\leqC
		h^{-1}\Norm{\QQh\TT{v}_h}{\Kh} + \Norm{\TT{v}_h}{\Kh}\,.
	\end{align*}
	Thus, we have
	\begin{align*}
		\Norm{\DDer\TT{v}_h^{e}}{\Kh}&\leqC
		\EnergyNorm{\TT{v}_h}{a_h} + \Norm{\TT{v}_h}{\Kh}
		+ h^{-1}\Norm{\QQh\TT{v}_h}{\Kh}.
	\end{align*}
	Choosing $h$ small enough yields estimate~\eqref{eq:disc-h1-bounds-2}.
\end{proof}

\begin{remark}
	Using \cref{lem:disc-h1-bounds}, we can reformulate the super-approximation result \eqref{eq:super-approx-2} for $s=1$ slightly:
	\begin{align}\label{eq:super-approx-2-modified}
		\HNorm{\PP \TT{v}_h^{e} - \I^{\ku} \PP \TT{v}_h^{e}}{\K}\leqC h \HNorm{\TT{v}_h^{e}}{\K}\leqC h\GNorm{\TT{v}_h} \leqC h\EnergyNorm{\TT{v}_h}{a_h} + \Norm{\QQh\TT{v}_h}{\Kh}.
	\end{align}
\end{remark}

\subsection{Geometric errors}\label{sec:geometric-errors}

There are two geometric error terms, namely the error in the quadratic form and the error in the linear form. The error in the quadratic form again splits into two parts, an $H^1$-part and an $L^2$-part. For non-tangential tensor fields $\TT{v},\TT{w}\in\HSpaceAmb{\K}$ we introduce the forms
\begin{align*}
	\delta_{\nabla}(\TT{v},\TT{w})
      &= \Inner{\CDer\PP\TT{v}}{\CDer\PP\TT{w}}{\K}
       - \Inner{\ChDer\PPh\TT{v}^e}{\ChDer\PPh\TT{w}^e}{\Kh}\\
	\delta_{m}(\TT{v},\TT{w})
      &= \Inner{\PP\TT{v}}{\PP\TT{w}}{\K}
       - \Inner{\PPh\TT{v}^e}{\PPh\TT{w}^e}{\Kh}\\
	\delta_{l}(\TT{v})
      &= \Inner{\TT{f}}{\TT{v}}{\K} - \Inner{\TT{f}^e}{\PPh\TT{v}^e}{\Kh}\,.
\end{align*}

Furthermore, we set
\begin{align}\label{eq:geom-error-delta}
	\delta(\TT{v})(\TT{w}) =\delta_{l}(\TT{w}) -\delta_{m}(\TT{v},\TT{w}) - \delta_{\nabla}(\TT{v},\TT{w}).
\end{align}

\begin{lemma}[The Linear Form]\label{lem:geometricQlError}
	For $\TT{v}_h\in \LSpace[2]{\Kh, \TensorBundleAmb}$, we have
	\begin{align*}
    \Abs{\delta_{l}(\TT{v}_h^\ell)}
        &\leqC h^{\kg+1}\Norm{\TT{f}}{\K}\left(\Norm{\PPh\TT{v}_h}{\Kh}
          + h^{-1}\Norm{\QQh\TT{v}_h}{\Kh}\right)\,.\end{align*}
\end{lemma}
\begin{proof}
	As $\TT{f}=\PP\TT{f}$, we have for all $\TT{v}_h\in \LSpace[2]{\Kh, \TensorBundleAmb}$
	\begin{align*}
		\Abs{\delta_{l}(\TT{v}_h^{\ell})}
		& = \Abs{\Inner{\TT{f}}{\TT{v}_h^{\ell}}{\K} - \Inner{\TT{f}^e}{\PPh\TT{v}_h}{\Kh}}\\
		&\leq \Abs{\Inner{\PP\TT{f}}{\QQh\TT{v}_h^{\ell}}{\K}}
		+ \Abs{\Inner{\TT{f}}{\PPh\TT{v}_h^{\ell}}{\K}
		- \Inner{\TT{f}^{e}}{\PPh\TT{v}_h}{\Kh}}\\
		&\leqC\;h^{\kg} \Norm{\QQh\TT{v}_h}{\Kh}\Norm{\TT{f}}{\K}
		+ h^{\kg+1} \Norm{\PPh\TT{v}_h}{\Kh}\Norm{\TT{f}}{\K}\,,
	\end{align*}
	which follows from \cref{lem:l2-norm-equiv} and \cref{lem:mixed_projections}.
\end{proof}

\begin{lemma}[The mass matrix term]\label{lem:geometricQmError}
	For $\TT{v}_h,\TT{w}_h\in\LSpace[2]{\Kh, \TensorBundleAmb}$ (non-tangential)	the following estimate holds for $h<h_0$ small enough:
	\if\isimastyle1
	\begin{align}\label{eq:geometricQmError1}
		\Abs{\delta_{m}(\TT{v}_h^\ell,\TT{w}_h^\ell)}
		&\leqC h^{\kg+1}
			\left(\Norm{\PPh\TT{w}_h}{\Kh} + h^{-1}\Norm{\QQh\TT{w}_h}{\Kh}\right)
		\! \left(\Norm{\PPh\TT{v}_h}{\Kh} + h^{-1}\Norm{\QQh\TT{v}_h}{\Kh}\right).
	\end{align}
	\else
	\begin{multline}\label{eq:geometricQmError1}
		\Abs{\delta_{m}(\TT{v}_h^\ell,\TT{w}_h^\ell)}\\
		\qquad \leqC h^{\kg+1}
			\left(\Norm{\PPh\TT{w}_h}{\Kh} + h^{-1}\Norm{\QQh\TT{w}_h}{\Kh}\right)
		\! \left(\Norm{\PPh\TT{v}_h}{\Kh} + h^{-1}\Norm{\QQh\TT{v}_h}{\Kh}\right).
	\end{multline}
	\fi

	For $\TT{v}\in\LSpace[2]{\K, \TensorBundleTan}$ (tangential) and $\TT{w}_h\in\LSpace[2]{\Kh, \TensorBundleAmb}$ (non-tangential)	the following estimate holds for $h<h_0$ small enough:
	\begin{equation}\label{eq:geometricQmError2}
			\Abs{\delta_{m}(\TT{v},\TT{w}_h^\ell)}
				\leqC h^{\kg+1}\left(\Norm{\PPh\TT{w}_h}{\Kh} + h^{-1}\Norm{\QQh\TT{w}_h}{\Kh}\right)\Norm{\TT{v}}{\K}.
	\end{equation}

	For $\TT{v}, \TT{w}\in\LSpace[2]{\K, \TensorBundleTan}$ (tangential) the following estimate holds for $h<h_0$ small enough:
	\begin{equation}\label{eq:geometricQmError3}
	\Abs{\delta_{m}(\TT{v},\TT{w})}
	\leqC h^{\kg+1}\Norm{\TT{w}}{\K}\Norm{\TT{v}}{\K}.
	\end{equation}
\end{lemma}
\begin{proof}
	We write for For $\TT{v}_h,\TT{w}_h\in\LSpace[2]{\Kh, \TensorBundleAmb}$
	\begin{align*}
		\Abs{\delta_{m}(\TT{v}_h^\ell,\TT{w}_h^\ell)}
		 		\leq & \Abs{\Inner{\PP\TT{v}_h^\ell}{\PP\TT{w}_h^\ell}{\K} - \Inner{\PP\TT{v}_h}{\PP\TT{w}_h}{\Kh}}\\
				& + \Abs{\Inner{(\PP-\PPh)\TT{v}_h}{(\PP-\PPh)\TT{w}_h}{\Kh}}\\
				& + \Abs{\Inner{\PPh\TT{v}_h}{(\PP-\PPh)\TT{w}_h}{\Kh}}\\
				& + \Abs{\Inner{(\PP-\PPh)\TT{v}_h}{\PPh\TT{w}_h}{\Kh}}\,.
	\end{align*}
	We use \cref{lem:l2-norm-equiv} to estimate the first term on the right hand side by
	\if\isimastyle1
	\begin{align*}
		\Abs{\Inner{\PP\TT{v}_h^\ell}{\PP\TT{w}_h^\ell}{\K} - \Inner{\PP\TT{v}_h}{\PP\TT{w}_h}{\Kh}}
		&\leqC h^{\kg+1} \Norm{\PP\TT{v}_h}{\Kh}\Norm{\PP\TT{w}_h}{\Kh}\\
		&\leqC h^{\kg+1} \left(\Norm{\PPh\TT{v}_h}{\Kh} + h^{\kg}\Norm{\QQh\TT{v}_h}{\Kh}\right)\left(\Norm{\PPh\TT{w}_h}{\Kh} + h^{\kg}\Norm{\QQh\TT{w}_h}{\Kh}\right).
	\intertext{The second term gives}
		\Abs{\Inner{(\PP-\PPh)\TT{v}_h}{(\PP-\PPh)\TT{w}_h}{\Kh}}
		&\leqC h^{2\kg} \Norm{\TT{v}_h}{\Kh}\Norm{\TT{w}_h}{\Kh}\\
		&\leqC h^{2\kg} \left(\Norm{\PPh\TT{v}_h}{\Kh} + \Norm{\QQh\TT{v}_h}{\Kh}\right)\left(\Norm{\PPh\TT{w}_h}{\Kh} + \Norm{\QQh\TT{w}_h}{\Kh}\right)\,.
	\intertext{For the third term we write}
		\Abs{\Inner{\PPh\TT{v}_h}{(\PP-\PPh)\TT{w}_h}{\Kh}}
		&\leq \Abs{\Inner{\PPh\TT{v}_h}{\PPh(\PP-\PPh)\PPh\TT{w}_h}{\Kh}}
		 + \Abs{\Inner{\PPh\TT{v}_h}{\PPh(\PP-\PPh)\QQh\TT{w}_h}{\Kh}}\\
		&\leqC h^{2\kg} \Norm{\PPh\TT{v}_h}{\Kh}\Norm{\TT{w}_h}{\Kh}
		 + h^{\kg}\Norm{\PPh\TT{v}_h}{\Kh}\Norm{\QQh\TT{w}_h}{\Kh}\\
		&\leqC h^{\kg+1}\Norm{\PPh\TT{v}_h}{\Kh}\left(\Norm{\PPh\TT{w}_h}{\Kh}+h^{-1}\Norm{\QQh\TT{v}_h}{\Kh}\right)\,.
	\end{align*}
	\else
	\begin{align*}
		&	\Abs{\Inner{\PP\TT{v}_h^\ell}{\PP\TT{w}_h^\ell}{\K} - \Inner{\PP\TT{v}_h}{\PP\TT{w}_h}{\Kh}} \\
		&\qquad \leqC h^{\kg+1} \Norm{\PP\TT{v}_h}{\Kh}\Norm{\PP\TT{w}_h}{\Kh}\\
		&\qquad \leqC h^{\kg+1} \left(\Norm{\PPh\TT{v}_h}{\Kh} + h^{\kg}\Norm{\QQh\TT{v}_h}{\Kh}\right)\left(\Norm{\PPh\TT{w}_h}{\Kh} + h^{\kg}\Norm{\QQh\TT{w}_h}{\Kh}\right).
	\intertext{The second term gives}
		&	\Abs{\Inner{(\PP-\PPh)\TT{v}_h}{(\PP-\PPh)\TT{w}_h}{\Kh}}\\
		&\qquad \leqC h^{2\kg} \Norm{\TT{v}_h}{\Kh}\Norm{\TT{w}_h}{\Kh}\\
		&\qquad \leqC h^{2\kg} \left(\Norm{\PPh\TT{v}_h}{\Kh} + \Norm{\QQh\TT{v}_h}{\Kh}\right)\left(\Norm{\PPh\TT{w}_h}{\Kh} + \Norm{\QQh\TT{w}_h}{\Kh}\right)\,.
	\intertext{For the third term we write}
		& \Abs{\Inner{\PPh\TT{v}_h}{(\PP-\PPh)\TT{w}_h}{\Kh}}\\
		&\qquad \leq \Abs{\Inner{\PPh\TT{v}_h}{\PPh(\PP-\PPh)\PPh\TT{w}_h}{\Kh}}
		 + \Abs{\Inner{\PPh\TT{v}_h}{\PPh(\PP-\PPh)\QQh\TT{w}_h}{\Kh}}\\
		&\qquad \leqC h^{2\kg} \Norm{\PPh\TT{v}_h}{\Kh}\Norm{\TT{w}_h}{\Kh}
		 + h^{\kg}\Norm{\PPh\TT{v}_h}{\Kh}\Norm{\QQh\TT{w}_h}{\Kh}\\
		&\qquad \leqC h^{\kg+1}\Norm{\PPh\TT{v}_h}{\Kh}\left(\Norm{\PPh\TT{w}_h}{\Kh}+h^{-1}\Norm{\QQh\TT{v}_h}{\Kh}\right)\,.
	\end{align*}
	\fi
	The fourth term can be estimated analogously with $\TT{v}_h$ and $\TT{w}_h$ interchanged. Combining everything yields~\eqref{eq:geometricQmError1}.

	The estimates~\eqref{eq:geometricQmError2} and~\eqref{eq:geometricQmError3} follow from~\eqref{eq:geometricQmError1} as for $\TT{v}\in\LSpace[2]{\K, \TensorBundleTan}$ we have by $L^2$-norm equivalence
	\begin{align*}
		\Norm{\PPh\TT{v}^{e}}{\Kh}+ h^{-1}\Norm{\QQh\TT{v}^{e}}{\Kh}
		&\leqC \Norm{\PPh\TT{v}}{\K}+ h^{-1}\Norm{\QQh\TT{v}}{\K}\\
		&\leqC \Norm{\TT{v}}{\K}+ h^{\kg-1}\Norm{\TT{v}}{\K}.
	\end{align*}
\end{proof}

\begin{lemma}[The derivative term]\label{lem:geometricQnablaError}
	For $\TT{v}_h,\TT{w}_h\in\HSpace[1]{\Kh, \TensorBundleAmb}$ we have
	\begin{equation}\label{eq:geometricQnablaError-1}
		\Abs{\delta_{\nabla}(\TT{v}_h^\ell,\TT{w}_h^\ell)} \leqC h^{\kg}
			\GNorm{\TT{w}_{h}}\GNorm{\TT{v}_{h}}\,.
	\end{equation}

	If $\TT{v}\in\HSpaceTan[1]{\K}$ and $\TT{w}_h\in\Vh$ and $h<h_0$ small enough, then we obtain
	\begin{equation}\label{eq:geometricQnablaError-2}
		\Abs{\delta_{\nabla}(\TT{v},\TT{w}_h^\ell)} \leqC h^{\kg}\left(\EnergyNorm{\TT{w}_{h}}{a_h} + h^{-1}\Norm{\QQh\TT{w}_h}{\Kh}\right)\HNormTan {\TT{v}}{\K}.
	\end{equation}

	For $\TT{v},\TT{w}\in \HSpaceTan[2]{\K}$ we have
	\begin{equation}\label{eq:geometricQnablaError-3}
		\Abs{\delta_{\nabla}(\TT{v},\TT{w})} \leqC h^{\kg+1}\HNormTan[2]{\TT{v}}{\K}\HNormTan[2]{\TT{w}}{\K}.
	\end{equation}
\end{lemma}
\begin{proof}
	We begin proving~\eqref{eq:geometricQnablaError-1} by estimating
	\begin{align}\label{eq:delta-nabla-split}
		\Abs{\delta_{\nabla}(\TT{v}_h^\ell,\TT{w}_h^\ell)}
		\leq&  \Abs{\Inner{\CDer\PP\TT{v}_h^\ell}{\CDer\PP\TT{w}_h^\ell}{\K}
		- \Inner{(\CDer\PP\TT{v}_h^\ell)^{e}}{(\CDer\PP\TT{w}_h^\ell)^{e}}{\Kh}}\\
		&+ \Abs{ \Inner{(\CDer\PP\TT{v}_h^\ell)^{e}-\ChDer\PPh\TT{v}_h}{(\CDer\PP\TT{w}_h^\ell)^{e}}{\Kh}} \notag\\
		&+ \Abs{\Inner{\ChDer\PPh\TT{v}_h}{(\CDer\PP\TT{w}_h^\ell)^{e}-\ChDer\PPh\TT{w}_h}{\Kh}
		}\,.\notag
	\end{align}
	We use \cref{lem:l2-norm-equiv} to estimate the first term on the right hand side by
	\if\isimastyle1
	\begin{align*}
		\Abs{\Inner{\CDer\PP\TT{v}_h^\ell}{\CDer\PP\TT{w}_h^\ell}{\K}
		- \Inner{(\CDer\PP\TT{v}_h^\ell)^{e}}{(\CDer\PP\TT{w}_h^\ell)^{e}}{\Kh}}
		&\leqC h^{\kg +1} \Norm{(\CDer\PP\TT{v}_h^\ell)^{e}}{\Kh}\Norm{(\CDer\PP\TT{w}_h^\ell)^{e}}{\Kh}\\
		& \leqC h^{\kg +1} \Norm{\CDer\PP\TT{v}_h^\ell}{\K}\Norm{\CDer\PP\TT{w}_h^\ell}{\K}\\
		& \leqC h^{\kg +1} \GNorm{\TT{v}_h}\GNorm{\TT{w}_h}\;,
	\intertext{where~\eqref{eq:cder-to-gnorm-est} was used for the last estimate. For the second term we use \cref{lem:h1-estimates} and~\eqref{eq:cder-to-gnorm-est} to obtain}
		\Abs{ \Inner{(\CDer\PP\TT{v}_h^\ell)^{e}-\ChDer\PPh\TT{v}_h}{(\CDer\PP\TT{w}_h^\ell)^{e}}{\Kh}}
		&\leq \Norm{(\CDer\PP\TT{v}_h^\ell)^{e}-\ChDer\PPh\TT{v}_h}{\Kh}\Norm{(\CDer\PP\TT{w}_h^\ell)^{e}}{\Kh}\\
		&\leqC h^{\kg} \GNorm{\TT{v}_h} \Norm{(\CDer\PP\TT{w}_h^\ell)^{e}}{\Kh}\\
		&\leqC h^{\kg} \GNorm{\TT{v}_h}\GNorm{\TT{w}_h}\;.
	\end{align*}
	\else
	\begin{align*}
		&\Abs{\Inner{\CDer\PP\TT{v}_h^\ell}{\CDer\PP\TT{w}_h^\ell}{\K}
		- \Inner{(\CDer\PP\TT{v}_h^\ell)^{e}}{(\CDer\PP\TT{w}_h^\ell)^{e}}{\Kh}}\\
		&\qquad \leqC h^{\kg +1} \Norm{(\CDer\PP\TT{v}_h^\ell)^{e}}{\Kh}\Norm{(\CDer\PP\TT{w}_h^\ell)^{e}}{\Kh}\\
		&\qquad \leqC h^{\kg +1} \Norm{\CDer\PP\TT{v}_h^\ell}{\K}\Norm{\CDer\PP\TT{w}_h^\ell}{\K}\\
		&\qquad \leqC h^{\kg +1} \GNorm{\TT{v}_h}\GNorm{\TT{w}_h}\;,
	\intertext{where~\eqref{eq:cder-to-gnorm-est} was used for the last estimate. For the second term we use \cref{lem:h1-estimates} and~\eqref{eq:cder-to-gnorm-est} to obtain}
		&\Abs{ \Inner{(\CDer\PP\TT{v}_h^\ell)^{e}-\ChDer\PPh\TT{v}_h}{(\CDer\PP\TT{w}_h^\ell)^{e}}{\Kh}}\\
		&\qquad \leq \Norm{(\CDer\PP\TT{v}_h^\ell)^{e}-\ChDer\PPh\TT{v}_h}{\Kh}\Norm{(\CDer\PP\TT{w}_h^\ell)^{e}}{\Kh}\\
		&\qquad \leqC h^{\kg} \GNorm{\TT{v}_h} \Norm{(\CDer\PP\TT{w}_h^\ell)^{e}}{\Kh}\\
		&\qquad \leqC h^{\kg} \GNorm{\TT{v}_h}\GNorm{\TT{w}_h}\;.
	\end{align*}
	\fi
	The third term can be estimated analogously.

	Estimate~\eqref{eq:geometricQnablaError-2} follows by~\eqref{eq:geometricQnablaError-1} using~\eqref{eq:gnorm-to-h1tan} and~\eqref{eq:disc-h1-bounds-2}.
	To obtain estimate~\eqref{eq:geometricQnablaError-3} we split $\delta_{\nabla}$ as before in~\eqref{eq:delta-nabla-split} and also estimate the first term in the same manner. Instead of using \cref{lem:h1-estimates} for the second term we split this part as in the proof of \cref{lem:h1-estimates} into
	\if\isimastyle1
	\begin{align*}
		\Abs{ \Inner{(\CDer\TT{v})^{e}-\ChDer\PPh\TT{v}^{e}}{(\CDer\TT{w})^{e}}{\Kh}}
		&\leq \Abs{ \Inner{(\PP-\PP_h)\DDer\TT{v}^{e}}{(\CDer\TT{w})^{e}}{\Kh}}
		+\Abs{ \Inner{\PPh\DDer\QQh\TT{v}^{e}}{(\CDer\TT{w})^{e}}{\Kh}}\;.
	\end{align*}
	\else
	\begin{multline*}
		\Abs{ \Inner{(\CDer\TT{v})^{e}-\ChDer\PPh\TT{v}^{e}}{(\CDer\TT{w})^{e}}{\Kh}}\\
		\leq \Abs{ \Inner{(\PP-\PP_h)\DDer\TT{v}^{e}}{(\CDer\TT{w})^{e}}{\Kh}}
		+\Abs{ \Inner{\PPh\DDer\QQh\TT{v}^{e}}{(\CDer\TT{w})^{e}}{\Kh}}\;.
	\end{multline*}
	\fi
	For the first term we estimate using~\eqref{eq:very-evil-trick} and \cref{lem:gradient-estimate}
	\begin{align*}
		\Abs{ \Inner{(\PP-\PPh)\DDer\TT{v}^{e}}{(\CDer\TT{w})^{e}}{\Kh}}
		& = \Abs{ \Inner{\QQh\DDer\TT{v}^{e}}{(\CDer\TT{w})^{e}}{\Kh}}\\
		& \leqC h^{\kg+1} \HNormAmb{\DDer\TT{v}^{e}}{\K}\HNormAmb{(\CDer\TT{w})^{e}}{\K}\\
		& \leqC h^{\kg+1} \HNormTan[1]{\TT{v}}{\K}\HNormTan[2]{\TT{w}}{\K}.
	\end{align*}
	The second term can be estimated using \eqref{eq:pdq-disc}. Thus we obtain
	\begin{align*}
		\Abs{ \Inner{(\CDer\TT{v})^{e}-\ChDer\PPh\TT{v}^{e}}{(\CDer\TT{w})^{e}}{\Kh}}
		& \leqC h^{\kg+1} \HNormTan[1]{\TT{v}}{\K}\HNormTan[2]{\TT{w}}{\K}
	\end{align*}
	as estimate for the second term in~\eqref{eq:delta-nabla-split}. For the third term in~\eqref{eq:delta-nabla-split} we write
	\begin{multline*}
		\Abs{\Inner{\ChDer\PPh\TT{v}^{e}}{(\CDer\TT{w})^{e}-\ChDer\PPh\TT{w}^{e}}{\Kh}
		}
		\leq \Abs{\Inner{(\CDer\TT{v})^{e}}{(\CDer\TT{w})^{e}-\ChDer\PPh\TT{w}^{e}}{\Kh}}\\
		+ \Abs{\Inner{\ChDer\PPh\TT{v}^{e}-(\CDer\TT{v})^{e}}{(\CDer\TT{w})^{e}-\ChDer\PPh\TT{w}^{e}}{\Kh}}.
	\end{multline*}
	The first term corresponds to the second term in~\eqref{eq:delta-nabla-split} with $\TT{v}$ and $\TT{w}$ interchanged and can be dealt with accordingly. The last term can be estimated using \cref{lem:h1-estimates} and~\eqref{eq:gnorm-to-h1tan} by
	\begin{align*}
		\Abs{\Inner{\ChDer\PPh\TT{v}^{e}-(\CDer\TT{v})^{e}}{(\CDer\TT{w})^{e}-\ChDer\PPh\TT{w}^{e}}{\Kh}}&\leqC h^{2\kg} \HNormTan[1]{\TT{v}}{\K} \HNormTan[1]{\TT{w}}{\K}.
	\end{align*}
	Combining everything, we indeed obtain estimate~\eqref{eq:geometricQnablaError-3}.
\end{proof}

\subsection{\EN{Discretisation}{Discretization} errors}\label{sec:discretization-errors}
We will now prove the main results of this paper as described in \cref{sec:main-results}. To this end we proceed as in the standard theory for Galerkin methods and analogously to the theory of scalar surface finite elements \cite{De2009Higher}. This means we first introduce a perturbed Galerkin orthogonality for the discrete energy that is then used to prove the error estimates in the energy norm. We then improve these estimates for the tangential part of the error essentially by repeating the argument for the continuous energy and making use of the already shown discrete estimate. For the $L^2$-norm estimates we employ an Aubin--Nitsche type argument based on the dual problem.

\subsubsection{Perturbed Galerkin orthogonality}\label{sec:the-perturbed-galerkin-orthogonality}
In the standard theory of numerics for partial differential equations, an important tool is the Galerkin orthogonality, i.e., the property that the error between the continuous and the discrete solution is orthogonal on the discrete test space. In our setting this is no longer the case, but the error that is made with respect to true orthogonality can be quantified.

\begin{lemma}\label{lem:GalerkinOrtho}
	 Let $\TT{u}\in\HSpaceTan{\K}$ be the solution to the continuous problem, $\TT{u}_h\in \Vh$ the solution to the discrete problem and $\TT{v}_h\in \Vh$ a test tensor field. Then
	 \begin{align}
	 	a(\TT{u}-\TT{u}_{h}^{\ell},\TT{v}_{h}^{\ell}) & = \delta(\TT{u}_{h}^{\ell})(\TT{v}_h^{\ell}) + s_{h}(\TT{u}_h,\TT{v}_h) \label{eq:GalerkinOrthoCont}\\
	 	A_{h}(\TT{u}^{e}-\TT{u}_{h},\TT{v}_{h}) & =  \delta(\TT{u})(\TT{v}_h^{\ell}) + s_{h}(\TT{u}^{e},\TT{v}_h)\label{eq:GalerkinOrthoDisc}.
	 \end{align}
\end{lemma}

\begin{proof}
	This follows directly by inserting the test tensor field into both the discrete and the continuous problems and subtracting and rearranging the terms. Note that this actually motivated the definition of $\delta$ in~\eqref{eq:geom-error-delta}.
\end{proof}

In order to quantify this deviation from true orthogonality, we will need the following estimates

\begin{lemma}\label{lem:discreteGalerkinOrtho}
	For $\TT{u} \in \HSpaceTan{\K}$ and $\TT{v}_{h}\in \Vh$, we have for $h<h_0$ small enough
	\begin{align*}
		\Abs{\delta(\TT{u})(\TT{v}_h^{\ell})}
		&\leqC h^{\kg}\left(\EnergyNorm{\TT{v}_h}{a_h} + h^{-1}\Norm{\QQh\TT{v}_h}{\Kh}\right)
		\left(\HNormTan[1]{\TT{u}}{\K} + \Norm{\TT{f}}{\K}\right) \\
		&\leqC h^{\kg-1+\alpha}\EnergyNorm{\TT{v}_h}{A_h}
		\left(\HNormTan[1]{\TT{u}}{\K} + \Norm{\TT{f}}{\K}\right)\,, \\
		\Abs{s_h(\TT{u}^{e}, \TT{v}_h)}
&\leqC h^{\kp-\alpha}\EnergyNorm{\TT{v}_h}{s_h}\Norm{\TT{u}}{\K}\,.
	\end{align*}
\end{lemma}
\begin{proof}
	By \cref{lem:geometricQlError}, \cref{lem:geometricQmError} \eqref{eq:geometricQmError2}\EN{}{,} and \cref{lem:geometricQnablaError} \eqref{eq:geometricQnablaError-2}, we have
	\begin{align*}
		\Abs{\delta(\TT{u})(\TT{v}_h^{\ell})} & \leq \Abs{\delta_l(\TT{v}_h^{\ell})} + \Abs{\delta_{m}(\TT{u},\TT{v}_h^{\ell})}
		+\Abs{\delta_{\nabla}(\TT{u},\TT{v}_h^{\ell})}\\
		&\leqC
		h^{\kg+1}\left(\EnergyNorm{\TT{v}_{h}}{a_h}
		+ h^{-1}\Norm{\QQh\TT{v}_h}{\Kh}\right)\left(\Norm{\TT{f}}{\K}+ \Norm{\TT{u}}{\K}+ h^{-1}\HNormTan {\TT{u}}{\K}\right)\\
		&\leqC
		h^{\kg}\left(\EnergyNorm{\TT{v}_{h}}{a_h}
		+ h^{-1}\Norm{\QQh\TT{v}_h}{\Kh}\right)\left(\HNormTan {\TT{u}}{\K}+ \Norm{\TT{f}}{\K}\right).
	\end{align*}
	Furthermore, we have by \eqref{eq:disc-h1-bounds-1}
	\begin{align*}
		\EnergyNorm{\TT{v}_{h}}{a_h}
		+ h^{-1}\Norm{\QQh\TT{v}_h}{\Kh}\leqC h^{-1+\alpha} \EnergyNorm{\TT{v}_{h}}{A_h}\,.
	\end{align*}
	The second estimate follows from
	\begin{align*}
		\Abs{s_h(\TT{u}^{e}, \TT{v}_h)}  &= \beta h^{-2\alpha} \Abs{\Inner{\QQhTilde \TT{u}^{e}}{\QQhTilde \TT{v}_h}{\Kh}}\\
		&\leq  \beta h^{-2\alpha} \Norm{\QQhTilde \TT{u}^{e}}{\Kh}\Norm{\QQhTilde \TT{v}_h}{\Kh}\\
		&\leqC \beta h^{\kp-\alpha} \Norm{\TT{u}^{e}}{\Kh}\EnergyNorm{ \TT{v}_h}{s_h}
	\end{align*}
	and the $L^2$-norm equivalence.
\end{proof}

This will be enough to prove the error estimates in the discrete energy norm. In order to improve our estimates later, we will also employ the following result:

\begin{lemma}\label{lem:contGalerkinOrtho}
	Let $\TT{u},\TT{v}\in \HSpaceTan[1]{\G}$ and $\TT{u}_{h},\TT{u}_{h}\in \Vh$
	discrete approximations of $\TT{u},\TT{v}$. Then
	\if\isimastyle1
	\begin{align*}
		\Abs{\delta(\TT{u}_{h}^{\ell})(\TT{v}_h^{\ell})}
		&\leqC h^{\kg}\big(\HNormTan[1]{\TT{v}}{\K} + \GNorm{\TT{v}^{e}-\TT{v}_h}\big)
		\big(\Norm{\TT{f}}{\K} + \HNormTan[1]{\TT{u}}{\K} +\GNorm{\TT{u}^{e}-\TT{u}_h}\big).
	\end{align*}
	\else
	\begin{multline*}
		\Abs{\delta(\TT{u}_{h}^{\ell})(\TT{v}_h^{\ell})}
		\leqC h^{\kg}\big(\HNormTan[1]{\TT{v}}{\K} + \GNorm{\TT{v}^{e}-\TT{v}_h}\big)
		\big(\Norm{\TT{f}}{\K} + \HNormTan[1]{\TT{u}}{\K} +\GNorm{\TT{u}^{e}-\TT{u}_h}\big).
	\end{multline*}
	\fi
	If $\TT{u},\TT{v}\in \HSpaceTan[2]{\G}$ then
	\if\isimastyle1
	\begin{align*}
		\Abs{\delta(\TT{u}_{h}^{\ell})(\TT{v}_h^{\ell})}
		&\leqC h^{\kg+1}\big(\HNormTan[2]{\TT{v}}{\K} + h^{-1}\GNorm{\TT{v}^{e}\!-\!\TT{v}_h}\big)
		 \big(\Norm{\TT{f}}{\K} + \HNormTan[2]{\TT{u}}{\K} + h^{-1}\GNorm{\TT{u}^{e}\!-\!\TT{u}_h}\big).
	\end{align*}
	\else
	\begin{multline*}
		\Abs{\delta(\TT{u}_{h}^{\ell})(\TT{v}_h^{\ell})}\\
		\qquad\leqC h^{\kg+1}\big(\HNormTan[2]{\TT{v}}{\K} + h^{-1}\GNorm{\TT{v}^{e}\!-\!\TT{v}_h}\big)
		 \big(\Norm{\TT{f}}{\K} + \HNormTan[2]{\TT{u}}{\K} + h^{-1}\GNorm{\TT{u}^{e}\!-\!\TT{u}_h}\big).
	\end{multline*}
	\fi
\end{lemma}
\begin{proof}
	We write again
	\begin{align*}
		\Abs{\delta(\TT{u}_h^{\ell})(\TT{v}_h^{\ell})} & \leq \Abs{\delta_l(\TT{v}_h^{\ell})} + \Abs{\delta_{m}(\TT{u}_h^{\ell},\TT{v}_h^{\ell})}
		+\Abs{\delta_{\nabla}(\TT{u}_h^{\ell},\TT{v}_h^{\ell})}.
	\end{align*}
	The first two terms are estimated using \cref{lem:geometricQlError} and \cref{lem:geometricQmError} \eqref{eq:geometricQmError1} by
	\if\isimastyle1
	\begin{align*}
		\Abs{\delta_l(\TT{v}_h^{\ell})}+ \Abs{\delta_{m}(\TT{u}_h^{\ell},\TT{v}_h^{\ell})}
		&\leqC h^{\kg+1}\left(\Norm{\PPh\TT{v}_h}{\Kh} + h^{-1}\Norm{\QQh\TT{v}_h}{\Kh}\right)\left(\Norm{\TT{f}}{\K} + \Norm{\PPh\TT{u}_h}{\Kh} + h^{-1}\Norm{\QQh\TT{u}_h}{\Kh}\right).
	\end{align*}
	\else
	\begin{multline*}
		\Abs{\delta_l(\TT{v}_h^{\ell})}+ \Abs{\delta_{m}(\TT{u}_h^{\ell},\TT{v}_h^{\ell})}\\
		 \leqC h^{\kg+1}\left(\Norm{\PPh\TT{v}_h}{\Kh} + h^{-1}\Norm{\QQh\TT{v}_h}{\Kh}\right)\left(\Norm{\TT{f}}{\K} + \Norm{\PPh\TT{u}_h}{\Kh} + h^{-1}\Norm{\QQh\TT{u}_h}{\Kh}\right).
	\end{multline*}
	\fi
	We then estimate using~\eqref{eq:gnorm-to-h1tan}
	\begin{align*}
		\Norm{\PPh\TT{v}_h}{\Kh} + h^{-1}\Norm{\QQh\TT{v}_h}{\Kh} \leq \GNorm{\TT{v}_h}
		&\leq \GNorm{\TT{v}^{e}} + \GNorm{\TT{v}^{e} - \TT{v}_h}\\
		&\leqC \HNormTan[1]{\TT{v}}{\K} + \GNorm{\TT{v}^{e} - \TT{v}_h}\;.
	\end{align*}
	For the last term we get
	\if\isimastyle1
	\begin{align*}
		\Abs{\delta_{\nabla}(\TT{u}_h^{\ell},\TT{v}_h^{\ell})}
		 &\leq \Abs{\delta_{\nabla}(\TT{u}_h^{\ell}-\TT{u},\TT{v}_h^{\ell}-\TT{v})}
		+ \Abs{\delta_{\nabla}(\TT{u},\TT{v}_h^{\ell}-\TT{v})}
		+\Abs{\delta_{\nabla}(\TT{u}_h^{\ell}-\TT{u},\TT{v})}
		+ \Abs{ \delta_{\nabla}(\TT{u},\TT{v})}
	\end{align*}
	\else
	\begin{multline*}
		\Abs{\delta_{\nabla}(\TT{u}_h^{\ell},\TT{v}_h^{\ell})}
		 \leq \Abs{\delta_{\nabla}(\TT{u}_h^{\ell}-\TT{u},\TT{v}_h^{\ell}-\TT{v})}
		+ \Abs{\delta_{\nabla}(\TT{u},\TT{v}_h^{\ell}-\TT{v})}
		+\Abs{\delta_{\nabla}(\TT{u}_h^{\ell}-\TT{u},\TT{v})}
		+ \Abs{ \delta_{\nabla}(\TT{u},\TT{v})}
	\end{multline*}
	\fi
	Estimating the first part using \eqref{eq:geometricQnablaError-1} and the second and third parts using \eqref{eq:geometricQnablaError-2} gives
	\begin{align*}
		\Abs{\delta_{\nabla}(\TT{u}_h^{\ell}-\TT{u},\TT{v}_h^{\ell}-\TT{v})}
		& \leqC h^{\kg}\GNorm{\TT{u}_h-\TT{u}^{e}}\GNorm{\TT{v}_h-\TT{v}^{e}}\\
		\Abs{\delta_{\nabla}(\TT{u},\TT{v}_h^{\ell}-\TT{v})}
		& \leqC h^{\kg} \GNorm{\TT{v}_h-\TT{v}^{e}}\HNormTan[1]{\TT{u}}{\K}\\
		\Abs{\delta_{\nabla}(\TT{u}_h^{\ell}-\TT{u},\TT{v})}
		& \leqC h^{\kg} \GNorm{\TT{u}_h-\TT{u}^{e}}\HNormTan[1]{\TT{v}}{\K}.
	\end{align*}
	For the last term we use either \eqref{eq:geometricQnablaError-1} and \eqref{eq:gnorm-to-h1tan} or \eqref{eq:geometricQnablaError-3} depending on the smoothness of $\TT{u}$ and $\TT{v}$ and thus obtain either
	\begin{align*}
		\Abs{ \delta_{\nabla}(\TT{u},\TT{v})}
		&\leqC h^{\kg}\HNormTan[1]{\TT{u}}{\K}\HNormTan[1]{\TT{v}}{\K},\text{ or }\\
		\Abs{ \delta_{\nabla}(\TT{u},\TT{v})}
		&\leqC h^{\kg+1}\HNormTan[2]{\TT{u}}{\K}\HNormTan[2]{\TT{v}}{\K}\,.
	\end{align*}
	Combing these estimates yields the assertions.
\end{proof}

\subsubsection{Estimates in the energy norm}\label{sec:estimates-in-the-energy-norm}

We employ the discrete perturbed Galerkin orthogonality \cref{lem:discreteGalerkinOrtho} to prove the following estimate for the error $\TT{e}\colonequals \TT{u}^{e}- \TT{u}_{h}$ in the energy norm. We split the error $\TT{e}$ in two parts, the interpolation error $\TT{e}_{I}\colonequals \TT{u}^{e}-\Ih\TT{u}^{e}$ with $\TT{I}_h=\Ih^{\ku}$ and a discrete remainder $\TT{e}_{h}\colonequals \Ih\TT{u}^{e}-\TT{u}_{h}$.
\begin{theorem}\label{thm:energy-norm-estimate}
	For the discrete solution $\TT{u}_{h}\in \Vh^{\ku}$ and the continuous solution $\TT{u}\in \HSpaceTan[k_u+1]{\K}$, we have for $h<h_0$ small enough the estimates
	\begin{align*}
		\EnergyNorm{\TT{u}^{e}-\TT{I}_h\TT{u}^{e}}{A_h}
			&\leqC\;h^{\ku}\HNormTan[\ku+1]{\TT{u}}{\K}\,, \\
		\EnergyNorm{\TT{I}_h\TT{u}^{e}-\TT{u}_{h}}{A_h}
			&\leqC\;h^{m}\left(\Norm{\TT{f}}{\K} + \HNormTan[\ku+1]{\TT{u}}{\K}\right)\,,
	\end{align*}
	with $m\colonequals \min\{\ku, \kg-1+\alpha, \kp-\alpha\}$ and hence
	\begin{align*}
		\EnergyNorm{\TT{u}^{e}- \TT{u}_{h}}{A_h}
		&\leqC\;h^{m}\left(\Norm{\TT{f}}{\K} + \HNormTan[\ku+1]{\TT{u}}{\K}\right)\,.
	\end{align*}
\end{theorem}

\begin{proof}
		The first estimate follows directly by interpolation error estimates \eqref{lem:interpolation-error-estimates}.
For the discrete remainder $\TT{e}_{h}=\Ih\TT{u}^{e}-\TT{u}_{h}$, we can use \eqref{eq:GalerkinOrthoDisc} and \cref{lem:discreteGalerkinOrtho} to obtain
	\begin{align*}
		\EnergyNorm{\TT{e}_{h}}{A_h}^2
		&=A_{h}(\Ih\TT{u}^{e}-\TT{u}_{h},\TT{e}_{h})\\
		&=A_{h}(\Ih\TT{u}^{e}-\TT{u}^{e},\TT{e}_{h}) + \delta(\TT{u})(\TT{e}_{h}^{\ell}) + s_h(\TT{u}^e, \TT{e}_h) \\
		&\leq \EnergyNorm{\TT{e}_{I}}{A_h}\EnergyNorm{\TT{e}_{h}}{A_h} + \Abs{\delta(\TT{u})(\TT{e}_{h}^{\ell})} + \Abs{s_h(\TT{u}^e, \TT{e}_h)} \\
		&\leqC (h^{\ku}+h^{\kg-1+\alpha}+h^{\kp-\alpha})\left(\Norm{\TT{f}}{\K} + \HNormTan[k_u+1]{\TT{u}}{\K}\right)
		\EnergyNorm{\TT{e}_{h}}{A_h}.
	\end{align*}
	Division by $\EnergyNorm{\TT{e}_{h}}{A_h}$ yields the assertion.
\end{proof}

\subsubsection{Improved estimate in tangential norms}\label{sec:improved-estimates-in-tangential-norms}
We obtain estimates for the tangential part of $\TT{e}$ by proceeding analogously to the proof of \cref{thm:energy-norm-estimate} for the continuous energy norm $\EnergyNorm{\cdot}{a}$. In order to really obtain an improvement relative to the discrete energy norm estimate, we employ the super-approximation result \cref{lem:superApproximation} to obtain an estimate on the normal part of the interpolation of the test vector $\PP\TT{e}$ that translates into a better estimate for the penalty term in particular.

\begin{theorem}\label{thm:better-h1-estimate}
	We have for $h<h_0$ small enough
	\begin{align}\label{eq:better-h1-estimate-1}
		\HNormTan{\PP(\TT{u} - \TT{u}_h^{\ell})}{\K}
		&\leqC  h^{\hat m}\left(\Norm{\TT{f}}{\K} + \HNormTan[k_u+1]{\TT{u}}{\K}\right)\,
	\end{align}
	with $\hat m =\min\{\ku,\kg,\kp+\kg-2\alpha, \kg -1 +2\alpha\}$.
\end{theorem}

\begin{proof}
	First, we note that the tangential $H^1$-norm corresponds to the continuous energy norm, i.e., $\HNormTan{\PP(\TT{u} - \TT{u}_h^{\ell})}{\K}=\EnergyNorm{\TT{e}^{\ell}}{a}$. Furthermore, interpolation of the projected error $\Ih\PP\TT{e}=\Ih^{\ku}\PP\TT{e}^{\ell}\in \Vh^{\ku}$ is well-defined and can also be written as $\Ih\PP\TT{e} = \Ih\TT{u^{e}}-\Ih\PP\TT{u}_{h}$. We will use this as a test tensor field to obtain the estimate in the continuous energy norm.
	The perturbed Galerkin orthogonality \eqref{eq:GalerkinOrthoCont} for the continuous bilinear form allows us to write
	\begin{align*}
		\EnergyNorm{\TT{e}^{\ell}}{a}^2
		&= a(\TT{e}^{\ell}, \TT{e}^{\ell} - (\Ih\PP\TT{e})^{\ell}) + \delta(\TT{u}_h^{\ell})(\Ih\PP\TT{e}) + s_h(\TT{u}_h, \Ih\PP\TT{e}) \\
		&\leq \EnergyNorm{\TT{e}^{\ell}}{a} \EnergyNorm{\TT{e}^{\ell} - (\Ih\PP\TT{e})^{\ell}}{a}	+ \Abs{\delta(\TT{u}_h^{\ell})(\Ih\PP\TT{e})} + \Abs{s_h(\TT{u}_h, \Ih\PP\TT{e})}.
	\end{align*}
	We begin by estimating $\EnergyNorm{\TT{e}^{\ell} - (\Ih\PP\TT{e})^{\ell}}{a}$.
	As in \cref{sec:improved-estimates-in-tangential-norms} we split the error into $\TT{e} = \TT{e}_{I} + \TT{e}_{h}$. Then
	\begin{align*}
		\EnergyNorm{\TT{e}^{\ell} - (\Ih\PP\TT{e})^{\ell}}{a}
		&\leq \EnergyNorm{\TT{e}_{I}^{\ell} - (\Ih\PP\TT{e}_{I})^{\ell}}{a}
		+ \EnergyNorm{\TT{e}^{\ell}_{h} - (\Ih\PP\TT{e}_{h})^{\ell}}{a}.
	\end{align*}
	By \cref{lem:interpolation-error-estimates-amb} and \cref{lem:gradient-estimate} we have
	\begin{align}\label{eq:PeI-IhPeI}
		\EnergyNorm{\TT{e}_{I}^{\ell} - (\Ih\PP\TT{e}_{I})^{\ell}}{a}
&\leqC\HNormAmb{\PP\TT{e}_{I}^{\ell} - (\Ih\PP\TT{e}_{I})^{\ell}}{\K} \leqC h \HNormAmb[2]{\PP\TT{e}_{I}^{\ell}}{\K} \leqC h^{\ku}\HNormTan[\ku+1]{\TT{u}}{\K}.
	\end{align}
	Then, using \cref{lem:h1-estimates}, $L^2$-norm equivalence and \eqref{eq:disc-h1-bounds-2} we obtain
	\begin{align*}
		\EnergyNorm{\TT{e}_{h}}{a_h} & \leqC \EnergyNorm{\TT{e}_{h}^{\ell}}{a} + h^{\kg}\GNorm{\TT{e}_{h}}\leqC \EnergyNorm{\TT{e}_{h}^{\ell}}{a} + h^{\kg-1}\Norm{\QQh\TT{e}_h}{\Kh} + h^{\kg}\EnergyNorm{\TT{e}_{h}}{a_h}.
	\end{align*}
	For $h$ small enough, we can absorb the last term into the left hand side. Using the estimate  $\EnergyNorm{\TT{e}^{\ell}_{I}}{a}\leqC h^{\ku}\HNormTan[\ku+1]{\TT{u}}{\K}$, we get
	\begin{align*}
		\EnergyNorm{\TT{e}_{h}}{a_h} &
		\leqC \EnergyNorm{\TT{e}_{h}^{\ell}}{a} + h^{\kg-1}\Norm{\QQh\TT{e}_h}{\Kh}
		\leqC \EnergyNorm{\TT{e}^{\ell}}{a} + h^{\ku+1}\HNormTan[\ku+1]{\TT{u}}{\K} + h^{\kg-1}\Norm{\QQh\TT{e}_h}{\Kh}
	\end{align*}
	Thus, we can use \eqref{eq:super-approx-2-modified}, \eqref{eq:disc-h1-bounds-1} and \cref{thm:energy-norm-estimate} to estimate
	\begin{align}
		\EnergyNorm{\TT{e}^{\ell}_{h} - (\Ih\PP\TT{e}_{h})^{\ell}}{a}
& \leqC \HNormAmb{\PP\TT{e}_{h}^{e} - (\Ih\PP\TT{e}_{h})^{e}}{\K} \notag\\
		&\leqC h \EnergyNorm{\TT{e}_{h}}{a_h} + \Norm{\QQh\TT{e}_h}{\Kh} \notag\\
		&\leqC  h \EnergyNorm{\TT{e}^{\ell}}{a} + h^{m+\alpha}\left(\Norm{\TT{f}}{\K} + \HNormTan[\ku+1]{\TT{u}}{\K}\right). \label{eq:Peh-IhPeh}
	\end{align}
So far we have shown that
	\begin{align*}
		\EnergyNorm{\TT{e}^{\ell}}{a}^2
		&\leq \EnergyNorm{\TT{e}^{\ell}}{a} \EnergyNorm{\TT{e}^{\ell} - (\Ih\PP\TT{e})^{\ell}}{a}	+ \Abs{\delta(\TT{u}_h^{\ell})(\Ih\PP\TT{e})} + \Abs{s_h(\TT{u}_h, \Ih\PP\TT{e})}\\
		&\leqC h \EnergyNorm{\TT{e}^{\ell}}{a}^2
		+ (h^{m+\alpha}+h^{\ku})\left(\Norm{\TT{f}}{\K} + \HNormTan[\ku+1]{\TT{u}}{\K}\right)\EnergyNorm{\TT{e}^{\ell}}{a}\\
		&\quad
		+ \Abs{\delta(\TT{u}_h^{\ell})(\Ih\PP\TT{e})} + \Abs{s_h(\TT{u}_h, \Ih\PP\TT{e})}.
	\end{align*}
	For $h$ small enough, we can absorb the first term into the left hand side.

	To estimate $\delta$ we interpret $\Ih\PP\TT{e}$ as an approximation of the $\TT{0}$ tensor field and use \cref{lem:contGalerkinOrtho} to estimate
	\begin{align*}
		\Abs{\delta(\TT{u}_h^{\ell})(\Ih\PP\TT{e})}
		&
		\leqC h^{\kg} \GNorm{\Ih\PP\TT{e}}\left(\Norm{\TT{f}}{\K} + \HNormTan[1]{\TT{u}}{\K} +\GNorm{\TT{u}^{e}-\TT{u}_h}\right).
	\end{align*}
For the last term in the brackets we estimate using \cref{lem:interpolation-error-estimates}, \eqref{eq:disc-h1-bounds-2}\EN{}{,} and \cref{thm:energy-norm-estimate}
	\begin{align}
		\GNorm{\TT{u}^{e}-\TT{u}_h}
		&\leq \GNorm{\TT{u}^{e}-\Ih\TT{u}^{e}}+\GNorm{\Ih\TT{u}^{e}-\TT{u}_h} \notag\\
		&\leqC h^{\ku}\HNormTan[\ku+1]{\TT{u}}{\K} +  h^{-1+\alpha}\EnergyNorm{\Ih\TT{u}^{e}-\TT{u}_h}{A_h} \notag \\
		&\leqC h^{m-1+\alpha}\left(\Norm{\TT{f}}{\K} + \HNormTan[\ku+1]{\TT{u}}{\K}\right). \label{eq:gnorm-estimate}
	\end{align}
	Thus, we obtain
	\begin{align*}
		\Abs{\delta(\TT{u}_h^{\ell})(\Ih\PP\TT{e})}
		&
		\leqC (h^{\kg}+h^{m+\kg-1+\alpha}) \GNorm{\Ih\PP\TT{e}}\left(\Norm{\TT{f}}{\K} + \HNormTan[\ku+1]{\TT{u}}{\K}\right).
	\end{align*}
	Note that by \eqref{eq:gnorm-to-h1} we get
	\begin{align*}
		\GNorm{\Ih\PP\TT{e}}
		&\leqC \HNormAmb[1]{\Ih\PP\TT{e}}{\Kh}
		 + h^{-1}\Norm{\QQh\Ih\PP\TT{e}}{\Kh}.
	\end{align*}
	For the penalty term $\Abs{s_h}$ we find
	\begin{align*}
		\Abs{s_h(\TT{u}_h, \Ih\PP\TT{e})}
		&\leqC h^{-2\alpha} \Norm{\QQhTilde\TT{u}_h}{\Kh} \Norm{\QQhTilde \Ih\PP\TT{e}}{\Kh}\\
		&\leqC h^{-\alpha}\left(\EnergyNorm{\TT{e}}{A_h} + h^{\kp-\alpha}\Norm{\TT{u}}{\K}\right) \Norm{\QQhTilde \Ih\PP\TT{e}}{\Kh}\\
		&\leqC h^{m-\alpha}\left(\Norm{\TT{f}}{\K} + \HNormTan[k_u+1]{\TT{u}}{\K}\right) \Norm{\QQhTilde \Ih\PP\TT{e}}{\Kh}
	\intertext{with}
		\Norm{\QQhTilde \Ih\PP\TT{e}}{\Kh}
		&\leqC \Norm{\QQh \Ih\PP\TT{e}}{\Kh} + h^{\kg}\Norm{\Ih\PP\TT{e}}{\Kh}.
	\end{align*}
	We now consider $\Norm{\QQh\Ih\PP\TT{e}}{\Kh}$. We use \cref{lem:superApproximation} to estimate
	\begin{align*}
		\Norm{\QQh\Ih\PP\TT{e}}{\Kh}
		& \leqC \Norm{\QQ\Ih\PP\TT{e}}{\Kh} + h^{\kg} \Norm{\Ih\PP\TT{e}}{\Kh}\\
		&\leqC (h^2+h^{\kg}) \HNormAmb[1]{\Ih\PP\TT{e}^{\ell}}{\Kh}.
	\end{align*}
	Thus we have shown
	\begin{align*}
		\Abs{\delta(\TT{u}_h^{\ell})(\Ih\PP\TT{e})}
		&\leqC (h^{\kg}+h^{m+\kg-1+\alpha})\left(\Norm{\TT{f}}{\K} + \HNormTan[\ku+1]{\TT{u}}{\K}\right)\HNormAmb[1]{\Ih\PP\TT{e}}{\Kh},\text{ and }\\
		\Abs{s_h(\TT{u}_h, \Ih\PP\TT{e})}
		&\leqC (h^{m-\alpha+2}+h^{m-\alpha+\kg})\left(\Norm{\TT{f}}{\K} + \HNormTan[k_u+1]{\TT{u}}{\K}\right)\HNormAmb[1]{\Ih\PP\TT{e}}{\Kh}.
	\end{align*}
	We further estimate using $L^2$-norm equivalence, \cref{lem:gradient-estimate}, \eqref{eq:PeI-IhPeI}\EN{}{,} and \eqref{eq:Peh-IhPeh} to obtain
	\begin{align*}
		\HNormAmb[1]{\Ih\PP\TT{e}}{\Kh}
			&\leq \HNormAmb[1]{\PP\TT{e}}{\Kh} + \HNormAmb[1]{\Ih\PP\TT{e} - \PP\TT{e}}{\Kh} \\
			&\leqC \EnergyNorm{\TT{e}^\ell}{a} + (h^{m+\alpha} + h^{\ku})\left(\Norm{\TT{f}}{\K} + \HNormTan[k_u+1]{\TT{u}}{\K}\right)\,.
	\end{align*}

	Combining all the terms and inserting the definition of $m$ results in the final form
\begin{align*}
		\EnergyNorm{\TT{e}^{\ell}}{a}
		&\leqC (h^{\ku}+h^{\kg}
		+h^{\kg-1+2\alpha}
		+h^{\kp+\kg-2\alpha}
		)\left(\Norm{\TT{f}}{\K} + \HNormTan[\ku+1]{\TT{u}}{\K}\right).
	\end{align*}
\end{proof}

\begin{corollary}\label{cor:ah-estimate}
	We have for $h<h_0$ small enough
	\begin{align}\label{eq:ah-estimate-1}
		\EnergyNorm{\TT{u}^{e} - \TT{u}_h}{a_h}
		&\leqC  h^{\hat m}\left(\Norm{\TT{f}}{\K} + \HNormTan[k_u+1]{\TT{u}}{\K}\right)\,
	\end{align}
	with $\hat m$ as in \cref{thm:better-h1-estimate}.
\end{corollary}
\begin{proof}
	This follows from \cref{thm:better-h1-estimate} using \cref{lem:h1-estimates} and \eqref{eq:gnorm-estimate} to estimate
	\begin{align*}
		\EnergyNorm{\TT{e}}{a_h}
		&\leqC \EnergyNorm{\TT{e}}{a} + h^{\kg}\GNorm{\TT{e}}
		\leqC \EnergyNorm{\TT{e}}{a} + h^{\kg+m-1+\alpha} \left(\Norm{\TT{f}}{\K} + \HNormTan[k_u+1]{\TT{u}}{\K}\right).
	\end{align*}
\end{proof}

\subsubsection{Estimate in the $L^2$-norm}\label{sec:estimate-in-the-l2-norm}
We consider the $L^2$-norm estimates separately for the tangential part $\PP\TT{e}$ and the normal part $\QQ\TT{e}$ of the error. For the tangential part we can employ an Aubin-Nitsche type argument in order to show an improvement from the estimate in the energy norm. Whether there is an improvement in the normal part as well depends on the choice of the penalty parameter $\alpha$.

In order to use the Aubin-Nitsche trick, we introduce the dual problem: Find $\TT{\psi}\in \HSpaceTan[1]{\G}$, s.t.
\begin{align}\label{eq:dual-problem}
	a(\TT{v},\TT{\psi})=\Inner{\TT{v}}{\PP\TT{e}^{\ell}}{\Gamma},\qquad \forall \TT{v}\in \HSpaceTan[1]{\G}.
\end{align}
\begin{lemma}\label{lem:aubin-nietsche-regularity}
	Assume that $\G\in C^2$. For a solution $\TT{\psi}\in\HSpaceTan[2]{\G}$ of \eqref{eq:dual-problem}, we have the $H^2$-regularity property
	\[
		\HNormTan[2]{\TT{\psi}}{\G}\leqC \Norm{\PP \TT{e}^{\ell}}{\G}\,.
	\]
\end{lemma}
\begin{proof}
	Similar to the argumentation by \citet{DE2013Finite}, the regularity result follows from the theory of elliptic systems of linear PDEs on Cartesian domains $\Omega\subset\mathbb{R}^d$ using local coordinates, see, e.g., \cite[Theorem 6.29]{Wa1983Foundations}.
\end{proof}

\begin{theorem}\label{thm:l2-norm-estimate}
	For the discrete solution $\TT{u}_{h}\in \Vh^{\ku}$ and the continuous solution $\TT{u}\in \HSpaceTan[k_u+1]{\K}$, we have for $h<h_0$ small enough and $\kp > \alpha$
	\begin{align*}
		\Norm{\QQ(\TT{u} - \TT{u}_h^{\ell})}{\K}
			&\leqC h^{\alpha}\EnergyNorm{\TT{u}^e - \TT{u}_h}{A_h}\\
			&\leqC  h^{ m + \alpha}\left(\Norm{\TT{f}}{\K} + \HNormTan[k_u+1]{\TT{u}}{\K}\right)\,,\\
		\Norm{\PP(\TT{u} - \TT{u}_h^{\ell})}{\K}
			&\leqC h \HNormTan{\PP(\TT{u} - \TT{u}_h^{\ell})}{\K} +	h^{\tilde{m}}\left(\Norm{\TT{f}}{\K} + \HNormTan[\ku+1]{\TT{u}}{\K}\right) \\
			&\leqC  h^{\min\{\hat m +1,\tilde{m}\}}\left(\Norm{\TT{f}}{\K} + \HNormTan[k_u+1]{\TT{u}}{\K}\right)\,,
	\end{align*}
	with $\tilde{m}\colonequals \min\{\kg+1, 2m,m+3-\alpha\}$, $m$ as in \cref{thm:energy-norm-estimate}\EN{}{,} and $\hat{m}$ as in \cref{thm:better-h1-estimate}.
\end{theorem}

\begin{proof}
	We have by $L^2$-norm equivalence
	\begin{align*}
		\Norm{\QQ(\TT{u} - \TT{u}_h^{\ell})}{\K}
		& \leqC \Norm{\QQhTilde(\TT{u}^{e} - \TT{u}_h)}{\Kh}
		+\Norm{(\QQ-\QQhTilde)(\TT{u}^{e} - \TT{u}_h)}{\Kh}\\
		& \leqC h^{\alpha} \EnergyNorm{\TT{u}^{e} - \TT{u}_h}{s_h}
		+h^{\kp}\Norm{\TT{u}^{e} - \TT{u}_h}{\Kh}\\
		& \leqC h^{\alpha} \EnergyNorm{\TT{u}^{e} - \TT{u}_h}{s_h}
		+h^{\kp}\Norm{\PPh(\TT{u}^{e} - \TT{u}_h)}{\Kh}
		+h^{\kp}\Norm{\QQh(\TT{u}^{e} - \TT{u}_h)}{\Kh}\\
		& \leqC h^{\alpha}(1+h^{\kp}) \EnergyNorm{\TT{u}^{e} - \TT{u}_h}{s_h}
		+(h^{\kp}+h^{2\kp})\Norm{\PPh(\TT{u}^{e} - \TT{u}_h)}{\Kh}.
	\end{align*}
	where the last estimate follows from~\eqref{eq:disc-h1-bounds-1}.

	In order to prove the estimate for the tangential part, we define $\TT{\psi}_h\colonequals \TT{I}_h^{m}\TT{\psi}^e \in\Vh^{\ku}$ with $m=\min\{\ku,2\}$ for the solution $\TT{\psi}\in \HSpaceTan[2]{\K}$ of the dual problem \eqref{eq:dual-problem}. Note that we have
	\begin{align}\label{eq:psi-psih-gnorm}
		\GNorm{\TT{\psi}^{e} - \TT{\psi}_h}\leqC h \HNormTan[2]{\TT{\psi}}{\K}
		\leqC h\Norm{\PP \TT{e}^{\ell}}{\K}
	\end{align}
	by interpolation error estimates \cref{lem:interpolation-error-estimates} and the $H^2$-regularity of the dual problem \cref{lem:aubin-nietsche-regularity}. Furthermore, we use the $L^2$-norm equivalence, super-approximation \cref{lem:superApproximation} for $\ku\geq 2$ and \cref{lem:interpolation-error-estimates} for $\ku=1$, with $\Norm{\QQ \TT{\psi}_h}{\Kh}=\Norm{\QQ (\TT{\psi}_h - \psi)}{\Kh}$, to obtain
	\begin{align*}
		\Norm{\QQhTilde  \TT{\psi}_h}{\Kh}
		&\leqC \Norm{\QQ \TT{\psi}_h}{\Kh} + h^{\kp}\Norm{\TT{\psi}_h}{\Kh}
		 \leqC h^{\min\{\ku+1,3\}} \HNormAmb[2]{\TT{\psi}_h^{\ell}}{\K} + h^{\kp}\Norm{\TT{\psi}_h^{\ell}}{\K}\\
		&\leqC h^{\min\{\ku+1,3,\kp\}} \HNormAmb[2]{\TT{\psi}_h^{\ell}}{\K}\,.
	\end{align*}
With interpolation error estimates \cref{lem:interpolation-error-estimates-amb}\EN{}{,} and \cref{lem:gradient-estimate} we have
	\begin{align*}
		\HNormAmb[2]{\TT{\psi}_h^{\ell}}{\K}
		&\leq \HNormAmb[2]{\TT{\psi}_h^{\ell} - \TT{\psi}}{\K} + \HNormAmb[2]{\TT{\psi}}{\K}
		 \leqC \HNormTan[2]{\TT{\psi}}{\K}\,.
	\end{align*}
Using $H^2$-regularity of the dual problem \cref{lem:aubin-nietsche-regularity}, we arrive at
	\begin{align}\label{eq:Qhkp-psih}
		\Norm{\QQhTilde  \TT{\psi}_h}{\Kh}
		&\leqC h^{\min\{\ku+1,3,\kp\}} \HNormTan[2]{\TT{\psi}}{\K}
		\leqC h^{\min\{\ku+1,3,\kp\}}\Norm{\PP \TT{e}^{\ell}}{\K}\,.
	\end{align}
Insertion of $\PP\TT{e}^{\ell}$ as a test function into the dual problem and using Galerkin orthogonality \eqref{eq:GalerkinOrthoCont} allows us to estimate
	\begin{align*}
	\Norm{\PP \TT{e}^{\ell}}{\K}^{2}
	&=a(\TT{e}^{\ell},\TT{\psi})\\
	&=a(\TT{e}^{\ell},\TT{\psi}-\TT{\psi}_{h}^{\ell}) + \delta(\TT{u}_{h}^{\ell})(\TT{\psi}_{h}^{\ell}) + s_h(\TT{u}_h,\TT{\psi}_h)\\
	&\leq \HNormTan{\PP \TT{e}^{\ell}}{\K}\HNormTan{\PP(\TT{\psi}-\TT{\psi}_{h}^{\ell})}{\K}
	+\Abs{\delta(\TT{u}_{h}^{\ell})(\TT{\psi}_{h}^{\ell})} + \Abs{s_h(\TT{u}_h,\TT{\psi}_h)}.
	\end{align*}
	Using \cref{lem:contGalerkinOrtho}, \eqref{eq:psi-psih-gnorm}\EN{}{,} and $H^2$-regularity of the dual problem we then get
	\if\isimastyle1
	\begin{align*}
		\Abs{\delta(\TT{u}_{h}^{\ell})(\TT{\psi}_{h}^{\ell})}
		& \leqC h^{\kg+1}\big(\HNormTan[2]{\TT{\psi}}{\K}\! + h^{-1}\GNorm{\TT{\psi}^{e}\!-\!\TT{\psi}_h}\big)\big(\Norm{\TT{f}}{\K} + \HNormTan[2]{\TT{u}}{\K}\! + h^{-1}\GNorm{\TT{u}^{e}\!-\!\TT{u}_h}\big)\\
		&\leqC h^{\kg+1}\Norm{\PP \TT{e}^{\ell}}{\K}\left(\Norm{\TT{f}}{\K} + \HNormTan[2]{\TT{u}}{\K} + h^{-1}\GNorm{\TT{u}^{e}\!-\!\TT{u}_h}\right).
		\end{align*}
	\else
	\begin{align*}
	&	\Abs{\delta(\TT{u}_{h}^{\ell})(\TT{\psi}_{h}^{\ell})} \\
	&\qquad \leqC h^{\kg+1}\big(\HNormTan[2]{\TT{\psi}}{\K}\! + h^{-1}\GNorm{\TT{\psi}^{e}\!-\!\TT{\psi}_h}\big)\big(\Norm{\TT{f}}{\K} + \HNormTan[2]{\TT{u}}{\K}\! + h^{-1}\GNorm{\TT{u}^{e}\!-\!\TT{u}_h}\big)\\
	&\qquad\leqC h^{\kg+1}\Norm{\PP \TT{e}^{\ell}}{\K}\left(\Norm{\TT{f}}{\K} + \HNormTan[2]{\TT{u}}{\K} + h^{-1}\GNorm{\TT{u}^{e}\!-\!\TT{u}_h}\right).
	\end{align*}
	\fi
	With \eqref{eq:gnorm-estimate} we arrive at
	\begin{align*}
		\Abs{\delta(\TT{u}_{h}^{\ell})(\TT{\psi}_{h}^{\ell})}
		&\leqC \left(h^{\kg+1}+ h^{2m}\right)\Norm{\PP \TT{e}^{\ell}}{\K}\left(\Norm{\TT{f}}{\K} + \HNormTan[\ku+1]{\TT{u}}{\K}\right)\,.
	\end{align*}
For the $\Abs{s_h}$ term we estimate, using \eqref{eq:Qhkp-psih} and \cref{thm:energy-norm-estimate}
\begin{align*}
	\Abs{s_h(\TT{u}_h,\TT{\psi}_h)}
	&\leqC h^{-2\alpha} \Norm{\QQhTilde\TT{u}_h}{\Kh} \Norm{\QQhTilde\TT{\psi}_h}{\Kh} \\
	&\leq h^{-\alpha} (h^{-\alpha}\Norm{\QQhTilde\TT{e}}{\Kh}+ h^{-\alpha}\Norm{\QQhTilde\TT{u}}{\Kh}) \Norm{\QQhTilde\TT{\psi}_h}{\Kh} \\
	&\leqC h^{-\alpha + \min\{\ku+1,3,\kp\}} \left(\EnergyNorm{\TT{e}}{A_h} + h^{\kp-\alpha}\Norm{\TT{u}}{\K}\right)\Norm{\PP\TT{e}^\ell}{\K} \\
	&\leqC \left(h^{m+3-\alpha}+h^{2m}\right) \left(\Norm{\TT{f}}{\K} + \HNormTan[\ku+1]{\TT{u}}{\K}\right)\Norm{\PP\TT{e}^\ell}{\K}\,.
	\end{align*}
	In summary, we finally obtain the assertion
\begin{align*}
	\Norm{\PP \TT{e}^{\ell}}{\K}
	&\leqC h\HNormTan{\PP \TT{e}^{\ell}}{\K}
	+\left(h^{\kg+1}+ h^{2m}+ h^{m+3-\alpha}\right)\left(\Norm{\TT{f}}{\K} + \HNormTan[\ku+1]{\TT{u}}{\K}\right)\,.
\end{align*}
\end{proof}

\section{Numerical Experiments}\label{sec:numerical-experiments}
We consider an ellipse curve and an ellipsoidal surface embedded in $\R^2$ and $\R^3$, respectively, given by the implicit descriptions
\begin{align*}
	\G_{1d} &= \Set{ x=(x_1,x_2)^T\in\R^2 \mid \phi_{1d}(x)\colonequals \left(\frac{x_1}{a}\right)^2 + \left(\frac{x_2}{b}\right)^2 - 1 = 0 }\,, \\
	\G_{2d} &= \Set{ x=(x_1,x_2,x_3)^T\in\R^3 \mid \phi_{2d}(x)\colonequals \left(\frac{x_1}{a}\right)^2 + \left(\frac{x_2}{b}\right)^2 + \left(\frac{x_3}{c}\right)^2 - 1 = 0 }\,,
\end{align*}
with $a=\frac{3}{4}, b=\frac{5}{4}$\EN{}{,} and $c=1$. This levelset description allows to compute the normal vector $\nn(x) = \DDer\phi(x)/\FNorm{\DDer\phi(x)}$ and the Weingarten map $\Weingarten(x)=-\P(x)\DDer\nn(x)\P(x)$ analytically. We also run tests on the unit sphere $\mathcal{S}^2$.

To construct a discrete surface approximation $\Gh$, we start with a reference grid $\hat{\G}_h$ obtained from an explicit ellipse, ellipsoid, or sphere grid, constructed using Gmsh \cite{GR2009Gmsh}. The higher-order polynomial surface approximation is then built by an iteratively computed surface projection $\pi(x)$ using the simple iterative scheme proposed by \citet{DD2007Adaptive} for the ellipse and ellipsoid and the direct coordinate projection $\pi(x)=x/\|x\|$ for the sphere. The grid representation is then implemented using \textsc{Dune-CurvedGrid}, see \cite{PS2020DuneCurvedGrid}.

In order to compute errors, we prescribe an exact solution of \eqref{eq:continuous_model_problem} and then compute the right-hand side function $\TT{f}$ analytically using Sympy \cite{MS2017SymPy}. We denote the tensor rank $n$ as superscript ${}^{(n)}$ in the following. The numerical experiments on the surfaces are performed for 1-tensor fields (vector fields) and 2-tensor fields, only.

We prescribe a solution vector field $\TT{u}_{1d}^{(1)}\in \TensorField[1]{\G_{1d}}$ and $\TT{u}_{2d}^{(1)}\in \TensorField[1]{\G_{2d}}$ as
\begin{align*}
	\TT{u}_{1d}^{(1)}(x) &= \P (x_1^3 x_2, (x_1+2)x_2^2)^T, && \text{for }x=(x_1,x_2)\in\G_{1d}\,, \\
	\TT{u}_{2d}^{(1)}(x) &= \operatorname{curl}_{\nn}(x_1 x_2 x_3), && \text{for }x=(x_1,x_2,x_3)\in\G_{2d}\,,
\end{align*}
with $\operatorname{curl}_{\nn} f = f\nn\times\nabla$ the normal-curl of a scalar field $f$. The right-hand side function $\TT{f}^{(1)}$ is then explicitly computed, as $\TT{f}^{(1)} = -\bm{\Delta}_\G \TT{u}^{(1)} + \TT{u}^{(1)}$. In \cref{fig:vector-ellipsoid-solution} the corresponding surface geometry and vector field solution is \EN{visualised}{visualized}.

\begin{figure}[ht]
  \centering
  \includegraphics[height=.2\textheight]{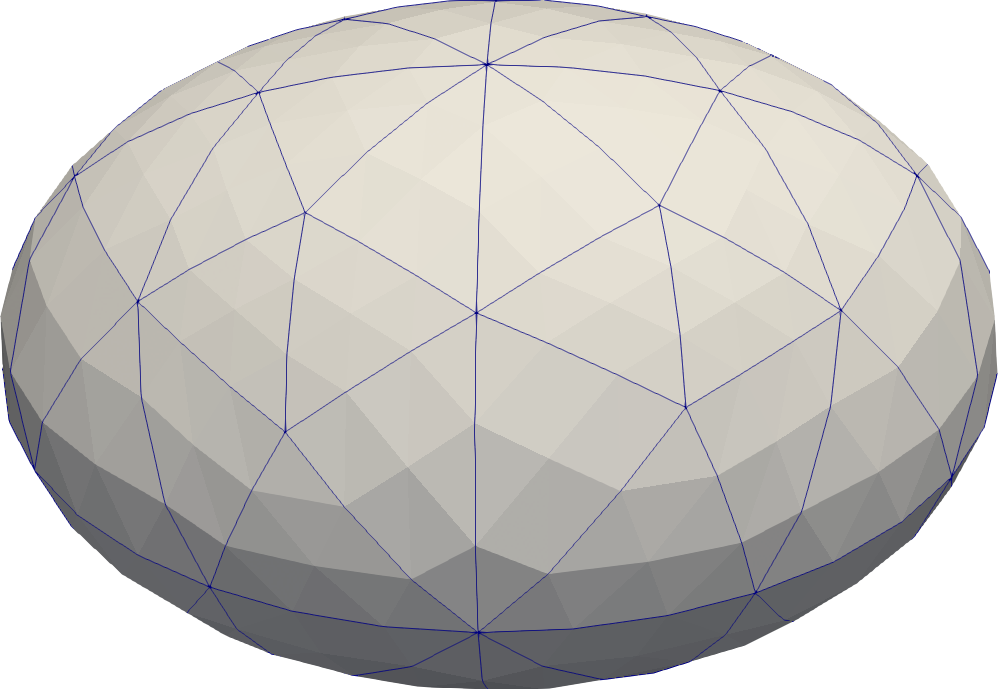}\hspace{0.5cm}
  \includegraphics[height=.2\textheight]{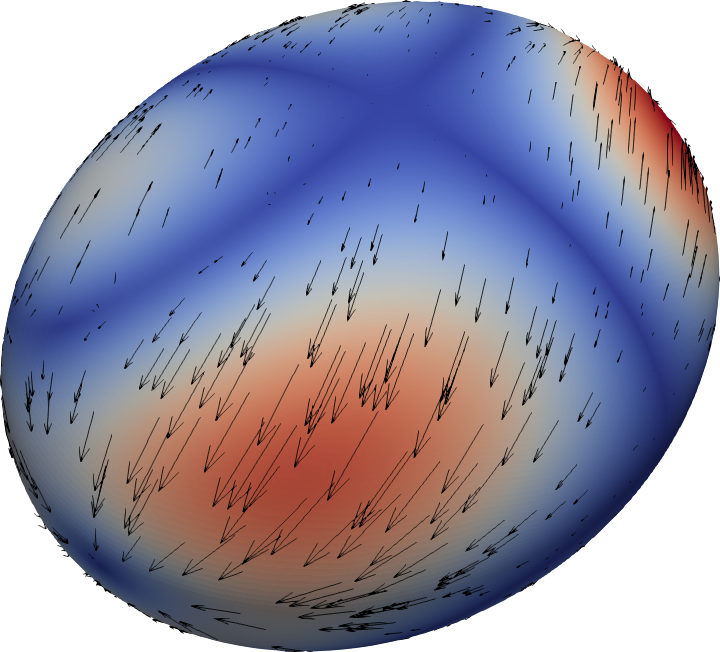}
  \caption{Left: initial grid of the ellipsoid surface $\G_{2d}$. Right: solution $\TT{u}_{2d}^{(1)}$ of the problem for $\kg=2$.}\label{fig:vector-ellipsoid-solution}
\end{figure}

A 2-tensor field solution $\TT{u}^{(2)}\in \TensorFieldTan[2]$ is prescribed as a projected constant tensor:
\begin{align*}
	\TT{u}_{1d}^{(2)}(x) &= \PP \begin{pmatrix}
		-1 & 3 \\ 1 & 2
	\end{pmatrix}, && \text{for }x=(x_1,x_2)\in\G_{1d}\,, \\
	\TT{u}_{2d}^{(2)}(x) &= \PP \begin{pmatrix}
		-1 & 3 & 0 \\ 1 & 2 & 0 \\ 0 & 0 & 1
	\end{pmatrix}, && \text{for }x=(x_1,x_2,x_3)\in\mathcal{S}^2\,.
\end{align*}
A corresponding right-hand side function $\TT{f}^{(2)}$ is computed in analogy to $\TT{f}^{(1)}$.

An experimental order of convergence is computed using a weighted least squares fitting of the estimate $\|\TT{u}-\TT{u}_h\|\simeq C h^p$ with \EN{logarithmised}{logarithmized} data, i.e.,
\[
	\operatorname{argmin}_{C,p}\left\{\sum_{l} w_l (\log(C) + p\log{h_l} - \log(\|\TT{e}_l\|))^2\right\}
\]
with $h_l$ and $\bm{e}_l$ the grid size and numerical error in refinement level $l$ and $w_l\simeq 2^l$ a weight to amplify small grid sizes. The constant $C$ is ignored in the following.

All simulations are implemented in the numerics framework \textsc{Dune}, see \cite{BB2021Dune,Sa2020Dune}. The code to produce the results and the corresponding grids and final data can be found in \cite{PH2021Code}.

\subsection{Results for vector-fields on $\G_{2d}$}
\cref{tbl:results_2d_vector} shows the experimental order of convergence for vector-field solutions of \eqref{eq:discrete_variational_problem} on $\G_{2d}$. The surface is approximated by discrete surfaces of geometry order $\kg\in\{1,2,3,4\}$ with penalty term order $\kp\in\{\kg,\kg+1\}$.

Since the error norms are composed of many terms that might be weighted by the penalty pre-factor $\beta$, we expect that with $\beta$ small or large either the one or the other term in the error estimates dominates the results for $h > h_0$, i.e., if the grid size is not yet small enough. Asymptotically, the constants in the analysis are irrelevant, but numerically we can only handle moderate grid sizes. This observation leads to the result that depending on $\beta$ we might see different convergence orders. This is partially already mentioned and observed by \citet{HL2020Analysis}. In order to compute the experimental order of convergence, we take the minimal order for computations with $\beta\in\{10,100,10^4\}$.

\begin{table}[ht]
	\begin{tabular}{l|l||ll|ll|ll|ll}
	$\ku,\kg,\kp$ & $\alpha$ & \multicolumn{2}{l|}{$\Norm{\TT{e}}{\Gh}$} & \multicolumn{2}{l|}{$\Norm{\PPh\TT{e}}{\Gh}$} & \multicolumn{2}{l|}{$\EnergyNorm{\TT{e}}{A_h}$} & \multicolumn{2}{l}{$\EnergyNorm{\TT{e}}{a_h}$} \\
	\hline
		 1, 1, 1 & 0.00 & 0.00 & (0.0) & 0.01 & (0.0) &     0.00 & (0.0)  &     0.02 & (0.0) \\
		 1, 1, 1 & 0.25 & 0.48 & (0.5) & 0.44 & (0.5) &     0.27 & (0.25) &     0.68 & (0.5) \\
		 1, 1, 1 & 0.50 & 0.99 & (1.0) & 0.93 & (1.0) &     0.53 & (0.5)  &     1.00 & (1.0) \\
		 1, 1, 1 & 0.75 & 0.52 & (0.5) & 0.49 & (0.5) &\ul{0.53}$^+$ & (0.25) &     0.51$^+$ & (0.5) \\
		 1, 1, 1 & 1.00 & 0.01 & (0.0) & 0.01 & (0.0) &     0.17 & (0.0)  &     0.15 & (0.0) \\
		 1, 1, 2 & 1.00 & 2.00 & (2.0) & 2.00 & (2.0) &     1.00 & (1.0)  &     1.00 & (1.0) \\
	\hline
		 2, 2, 2 & 0.00 &\ul{1.48} & (1.0) &\ul{2.70} & (2.0) &\ul{1.49} & (1.0)  &\ul{1.50}& (1.0) \\
		 2, 2, 2 & 0.25 &\ul{1.93} & (1.5) &\ul{2.91} & (2.5) &\ul{1.71} & (1.25) &\ul{1.82}& (1.5) \\
		 2, 2, 2 & 0.50 & 2.17$^*$ & (2.0) & 2.98 & (3.0) & 1.49$^*$ & (1.5)  &     1.96 & (2.0) \\
		 2, 2, 2 & 0.75 & 2.18 & (2.0) & 2.60$^*$ & (2.5) & 1.24$^*$ & (1.25) &     1.97 & (2.0) \\
		 2, 2, 2 & 1.00 & 2.05 & (2.0) & 2.10$^*$ & (2.0) & 0.98$^*$ & (1.0)  &     1.99 & (2.0) \\
		 2, 2, 3 & 1.00 & 3.12 & (3.0) & 2.99 & (3.0) &     2.01 & (2.0)  &     1.99 & (2.0) \\
	\hline
		 3, 3, 3 & 0.00 & 2.00 & (2.0) & 3.37 & (3.0) &     2.00 & (2.0)  &     2.11 & (2.0) \\
		 3, 3, 3 & 0.25 & 2.50 & (2.5) & 3.35 & (3.5) &     2.25 & (2.25) &     2.24 & (2.5) \\
		 3, 3, 3 & 0.50 & 2.99 & (3.0) & 3.86 & (4.0) &     2.50 & (2.5)  &     2.80 & (3.0) \\
		 3, 3, 3 & 0.75 & 3.15$^*$ & (3.0) & 3.98 & (4.0) & 2.32$^*$ & (2.25) &     2.97 & (3.0) \\
		 3, 3, 3 & 1.00 & 3.05$^+$ & (3.0) & 3.99 & (4.0) & 2.15$^*$ & (2.0)  &     2.98 & (3.0) \\
		 3, 3, 4 & 1.00 & 4.01 & (4.0) & 3.99 & (4.0) &     3.00 & (3.0)  &     2.98 & (3.0) \\
	\hline
		 4, 4, 4 & 0.00 & \ul{3.39} & (3.0) & \ul{4.46} & (4.0) & \ul{3.39} & (3.0)  & \ul{3.38} & (3.0) \\
		 4, 4, 4 & 0.25 & \ul{3.87} & (3.5) & 4.62 & (4.5) & \ul{3.62} & (3.25) &     3.56 & (3.5) \\
		 4, 4, 4 & 0.50 & 4.04$^*$ & (4.0) & 4.90 & (5.0)  & 3.48$^*$ & (3.5)  &     3.89 & (4.0) \\
		 4, 4, 4 & 0.75 & 4.05$^*$ & (4.0) & 4.92 & (5.0) & 3.23$^*$ & (3.25) &     3.92 & (4.0) \\
		 4, 4, 4 & 1.00 & 4.09$^*$ & (4.0) & 4.95 & (5.0) & 2.97$^*$ & (3.0)  &     3.94 & (4.0) \\
		 4, 4, 5 & 1.00 & 5.25 & (5.0) & 4.94 & (5.0) & 4.10 & (4.0)  &     3.93 & (4.0) \\
	\hline
\end{tabular}
	\caption{\label{tbl:results_2d_vector}Experimental order of convergence on $\G_{2d}$ for various geometry orders $\kg$ and some variation of the penalty term order $\kp\in\{\kg,\kg+1\}$ and pre-factor exponent $\alpha$. Additionally, the expected orders are shown in brackets. Numbers with a star $^*$ are computed using the pre-factor $\beta=10^4$, numbers with a sup-script plus $^+$ with $\beta=100$\EN{}{,} and the others with $\beta=10$. Underlined numbers are significantly better in our numerical experiments than the expected orders. Note, the energy-norm $\EnergyNorm{\TT{e}}{A_h}$ depends on the \EN{penalisation}{penalization} parameters $\alpha$ and $\beta$.}
\end{table}

First of all, we observe from the data that there is a difference in the energy norm $\EnergyNorm{\cdot}{A_h}$ and the discrete tangential $H^1$-norm $\EnergyNorm{\cdot}{a_h}$. Only for the case $\kp=\kg+1$ and $\alpha=1$ these convergence orders are the same.

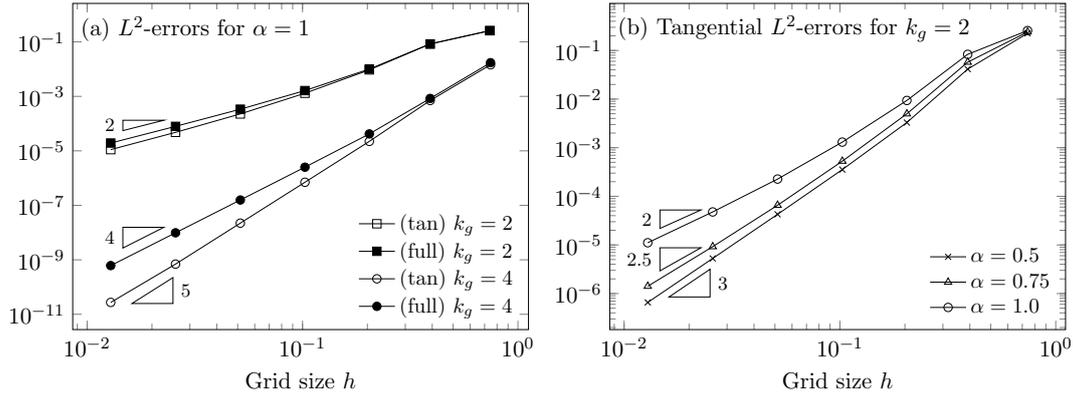
\begin{figure}[ht]
  \centering
  \begin{subfigure}[b]{0.48\textwidth}
  \scalebox{0.79}{\begin{tikzpicture}
  \def\vara{5}
  \def\varb{120}
  \begin{loglogaxis}[ xlabel={Grid size $h$}, domain={0.015625:0.29},
                      width=1.3\linewidth, height=\linewidth,
                      legend style={at={(0.99,0.015)}, anchor=south east,
                                    legend columns=1, draw=none, fill=none},
                      legend cell align=left, cycle list name=mark list ]
    \addplot[mark=square]  table[x=h, y=PL2] {vec_kg2_ku2_kp2_alpha1.dat};
    \addplot[mark=square*] table[x=h, y=L2]  {vec_kg2_ku2_kp2_alpha1.dat};
\addplot[mark=o]       table[x=h, y=PL2] {vec_kg4_ku4_kp4_alpha1.dat};
    \addplot[mark=*]       table[x=h, y=L2]  {vec_kg4_ku4_kp4_alpha1.dat};
    \legend{
      \small (tan)  $\kg=2$,
      \small (full) $\kg=2$,
\small (tan)  $\kg=4$,
      \small (full) $\kg=4$
    }
    \logLogSlopeTriangleT{0.2}{0.09}{0.61}{2}{black};
    \logLogSlopeTriangleT{0.2}{0.09}{0.25}{4}{black};
    \logLogSlopeTriangleB{0.22}{0.09}{0.08}{5}{black};
  \end{loglogaxis}
  \node[anchor=south west] at (0.0,4.7) {(a) $L^2$-errors for $\alpha=1$};
  \end{tikzpicture}}
\end{subfigure}
\begin{subfigure}[b]{0.48\textwidth}
  \scalebox{0.79}{\begin{tikzpicture}
  \def\vara{5}
  \def\varb{120}
  \begin{loglogaxis}[ xlabel={Grid size $h$}, domain={0.015625:0.29},
                      width=1.3\linewidth, height=\linewidth,
                      legend style={at={(0.99,0.015)}, anchor=south east,
                                    legend columns=1, draw=none, fill=none},
                      legend cell align=left, cycle list name=mark list ]
\addplot[mark=x]         table[x=h, y=PL2] {vec_kg2_ku2_kp2_alpha05.dat};
    \addplot[mark=triangle]         table[x=h, y=PL2] {vec_kg2_ku2_kp2_alpha075.dat};
    \addplot[mark=o]         table[x=h, y=PL2] {vec_kg2_ku2_kp2_alpha1.dat};
    \legend{
\small $\alpha=0.5$,
      \small $\alpha=0.75$,
      \small $\alpha=1.0$
    }

    \logLogSlopeTriangleT{0.2}{0.09}{0.31}{2}{black};
    \logLogSlopeTriangleT{0.2}{0.09}{0.18}{2.5}{black};
    \logLogSlopeTriangleB{0.22}{0.09}{0.1}{3}{black};
  \end{loglogaxis}
  \node[anchor=south west] at (0.0,4.7) {(b) Tangential $L^2$-errors for $\kg=2$};
  \end{tikzpicture}}
\end{subfigure}

   \caption{Errors in the norms $\Norm{\TT{u}^e - \TT{u}_h}{\Kh}$ (full) and $\Norm{\PP(\TT{u}^e - \TT{u}_h)}{\Kh}$ (tan) for vector fields on $\G_{2d}$ in the isogeometric case $\kg=\ku=\kp$. Left: fixed \EN{penalisation}{penalization} exponent $\alpha=1$ and variation of the geometry order $\kg\in\{2,4\}$. Right: variation of $\alpha\in\{0.5,0.75,1\}$ for fixed geometry order $\kg=2$. The calculations are performed with $\beta=10^4$.}\label{fig:vector-ellipsoid}
\end{figure}

Second, we observe that the tangential $L^2$-norm errors show higher order convergence than the full $L^2$-norm errors in some cases. For the isogeometric case, i.e., $\kg=\kp$, the tangential errors are up to one order better than the full errors for $\kg\geq 2$. This is \EN{visualised}{visualized} in \cref{fig:vector-ellipsoid} (a).

\begin{figure}[ht]
  \centering
  \begin{subfigure}[b]{0.48\textwidth}
  \scalebox{0.79}{\begin{tikzpicture}
  \def\vara{5}
  \def\varb{120}
  \begin{loglogaxis}[ xlabel={Grid size $h$}, domain={0.015625:0.29},
                      width=1.3\linewidth, height=\linewidth,
                      legend style={at={(0.99,0.015)}, anchor=south east,
                                    legend columns=1, draw=none, fill=none},
                      legend cell align=left, cycle list name=mark list ]
    \addplot[mark=triangle] table[x=h, y=PL2] {vec_kg1_ku1_kp1_alpha05.dat};
    \addplot[mark=square]   table[x=h, y=PL2] {vec_kg2_ku2_kp2_alpha05.dat};
    \addplot[mark=x]        table[x=h, y=PL2] {vec_kg3_ku3_kp3_alpha05.dat};
    \addplot[mark=o]        table[x=h, y=PL2] {vec_kg4_ku4_kp4_alpha05.dat};
    \legend{
      \small $\kg=1$,
      \small $\kg=2$,
      \small $\kg=3$,
      \small $\kg=4$
    }
    \logLogSlopeTriangleB{0.22}{0.09}{0.75}{1}{black};
    \logLogSlopeTriangleB{0.22}{0.09}{0.42}{3}{black};
    \logLogSlopeTriangleB{0.22}{0.09}{0.25}{4}{black};
    \logLogSlopeTriangleB{0.22}{0.09}{0.09}{5}{black};

  \end{loglogaxis}
  \node[anchor=south west] at (0.0,4.7) {(a) Tangential $L^2$-errors for $\alpha=0.5$};
  \end{tikzpicture}}
\end{subfigure}
\begin{subfigure}[b]{0.48\textwidth}
  \scalebox{0.79}{\begin{tikzpicture}
  \def\vara{5}
  \def\varb{120}
  \begin{loglogaxis}[ xlabel={Grid size $h$}, domain={0.015625:0.29},
                      width=1.3\linewidth, height=\linewidth,
                      legend style={at={(0.99,0.015)}, anchor=south east,
                                    legend columns=1, draw=none, fill=none},
                      legend cell align=left, cycle list name=mark list ]
    \addplot[mark=triangle]  table[x=h, y=PL2] {vec_kg1_ku1_kp1_alpha05.dat};
    \addplot[mark=triangle*] table[x=h, y=PL2] {vec_kg1_ku1_kp2_alpha05.dat};
    \addplot[mark=square]    table[x=h, y=PL2] {vec_kg2_ku2_kp2_alpha05.dat};
    \addplot[mark=square*]   table[x=h, y=PL2] {vec_kg2_ku2_kp3_alpha05.dat};
    \legend{
      {\small $\kg=1, \kp=1, \alpha=0.5$},
      {\small $\kg=1, \kp=2, \alpha=1.0$},
      {\small $\kg=2, \kp=2, \alpha=0.5$},
      {\small $\kg=2, \kp=3, \alpha=1.0$}
    }
    \logLogSlopeTriangleB{0.22}{0.09}{0.64}{1}{black};
    \logLogSlopeTriangleB{0.22}{0.09}{0.42}{2}{black};
    \logLogSlopeTriangleB{0.22}{0.09}{0.09}{3}{black};
  \end{loglogaxis}
  \node[anchor=south west] at (0.0,4.7) {(b) Tangential $L^2$-errors};
  \end{tikzpicture}}
\end{subfigure}
   \caption{Errors in the tangential $L^2$-norm for vector fields on $\G_{2d}$. Left: variation of $\kg=\kp$ for fixed $\alpha=0.5$. Right: comparison of $\kp\in\{\kg,\kg+1\}$ for $\alpha\in\{0.5,1.0\}$, i.e., the optimal convergence order in the isogeometric and non-isogeometric case.}\label{fig:vector-ellipsoid2}
\end{figure}
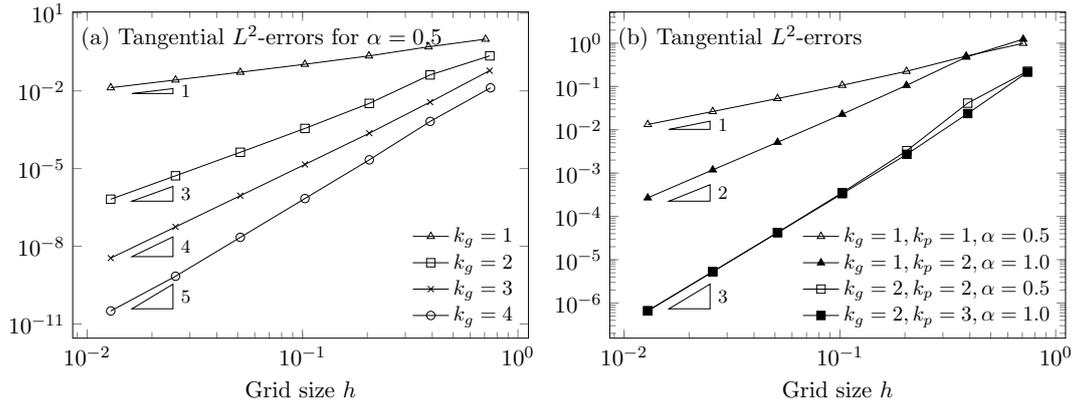

A third observation is that the case $\kg=1$ is special. In the isogeometric setup we need $0 < \alpha < 1$ in order to see convergence at all. While we get the same convergence order in the tangential $H^1$-norm for the isogeometric case $\kp=\kg=1,\alpha=0.5$ compared to the non-isogeometric case $\kp=\kg+1,\alpha=1$, this does not cross over to the $L^2$-norms. The optimal order is achieved for the non-isogeometric case only, see also \cref{fig:vector-ellipsoid2} (b). This indicates that the \EN{penalisation}{penalization} using normal vectors of the piecewise flat geometries does not represent the continuous surface well enough. An improved normal field might be desirable in this case, see \cref{rem:better-normal}.

The optimal tangential $H^1$-norm and tangential $L^2$-norm errors for the isogeometric case are obtained with $\alpha=0.5$, especially for low order geometries, $1\leq\kg\leq 2$. For higher order discrete surfaces the parameter $\alpha$ does not show a significant influence on these two error norms, while we still see a difference in the full discrete energy norm.
The influence of $\alpha$ on the convergence \EN{behaviour}{behavior} and the optimality of $\alpha=0.5$ is \EN{visualised}{visualized} in \cref{fig:vector-ellipsoid} (b) for the tangential $L^2$-norm errors.

A fourth observation in the data shown in \cref{tbl:results_2d_vector} is that the numerical convergence rates for even $\kg$ and small values of $\alpha$ are often better than the expected convergence orders. This \EN{behaviour}{behavior} is not explained in the analysis section of this paper. We have observed experimentally that for even $\kg$ the non-standard geometry estimates \eqref{eq:very-evil-trick} behave even better and seem to be of order $\kg+2$. This observation needs further investigation that goes beyond the scope of this paper.

Note, in \cref{tbl:results_2d_vector} we have computed discrete error norms only, due to numerical complexity of continuous norms. In those cases, where we have also computed continuous norms, we did not observe a difference in the convergence \EN{behaviour}{behavior}.

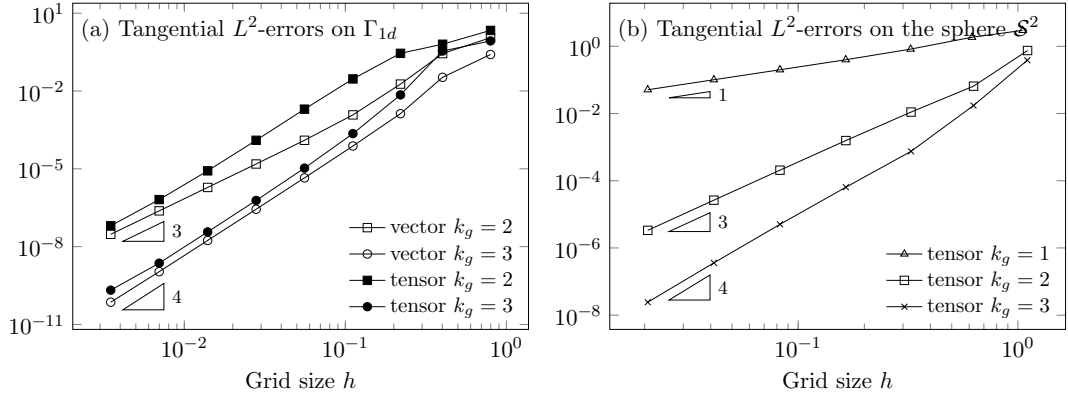
\begin{figure}[ht]
  \centering
  \begin{subfigure}[b]{0.48\textwidth}
  \scalebox{0.79}{\begin{tikzpicture}
  \def\vara{5}
  \def\varb{120}
  \begin{loglogaxis}[ xlabel={Grid size $h$}, domain={0.015625:0.29},
                      width=1.3\linewidth, height=\linewidth,
                      legend style={at={(0.99,0.015)}, anchor=south east,
                                    legend columns=1, draw=none, fill=none},
                      legend cell align=left, cycle list name=mark list ]
    \addplot[mark=square]  table[x=h, y=PL2] {vec1d_kg2_ku2_kp2_alpha05.dat};
    \addplot[mark=o]       table[x=h, y=PL2] {vec1d_kg3_ku3_kp3_alpha05.dat};
    \addplot[mark=square*] table[x=h, y=PL2] {ten1d_kg2_ku2_kp2_alpha05.dat};
    \addplot[mark=*]       table[x=h, y=PL2] {ten1d_kg3_ku3_kp3_alpha05.dat};
    \legend{
      \small vector $\kg=2$,
      \small vector $\kg=3$,
      \small tensor $\kg=2$,
      \small tensor $\kg=3$
    }
    \logLogSlopeTriangleB{0.2}{0.09}{0.27}{3}{black};
    \logLogSlopeTriangleB{0.2}{0.09}{0.06}{4}{black};
  \end{loglogaxis}
  \node[anchor=south west] at (0.0,4.7) {(a) Tangential $L^2$-errors on $\G_{1d}$};
  \end{tikzpicture}}
\end{subfigure}
\begin{subfigure}[b]{0.48\textwidth}
  \scalebox{0.79}{\begin{tikzpicture}
  \def\vara{5}
  \def\varb{120}
  \begin{loglogaxis}[ xlabel={Grid size $h$}, domain={0.015625:0.29},
                      width=1.3\linewidth, height=\linewidth,
                      legend style={at={(0.99,0.015)}, anchor=south east,
                                    legend columns=1, draw=none, fill=none},
                      legend cell align=left, cycle list name=mark list ]
\addplot[mark=triangle] table[x=h, y=PL2] {ten2d_kg1_ku1_kp1_alpha05.dat};
    \addplot[mark=square]   table[x=h, y=PL2] {ten2d_kg2_ku2_kp2_alpha05.dat};
    \addplot[mark=x]        table[x=h, y=PL2] {ten2d_kg3_ku3_kp3_alpha05.dat};
    \legend{
\small tensor $\kg=1$,
      \small tensor $\kg=2$,
      \small tensor $\kg=3$
    }
    \logLogSlopeTriangleB{0.22}{0.09}{0.71}{1}{black};
    \logLogSlopeTriangleB{0.22}{0.09}{0.30}{3}{black};
    \logLogSlopeTriangleB{0.22}{0.09}{0.09}{4}{black};
  \end{loglogaxis}
  \node[anchor=south west] at (0.0,4.7) {(b) Tangential $L^2$-errors on the sphere $\mathcal{S}^2$};
  \end{tikzpicture}}
\end{subfigure}
   \caption{Tangential $L^2$-norm errors for vector and tensor fields on the ellipse surface $\G_{1d}$ (left) and on the sphere (right), for fixed $\alpha=0.5$ and varying $\kg=\kp=\ku$.}\label{fig:ellipse}
\end{figure}

\subsection{Convergence results for tensor fields}
For tensor field solutions, we have considered the 1-dimensional ellipse surface and the 2-dimensional sphere as geometries. In \cref{fig:ellipse} (a) the errors are shown in comparison to the corresponding errors for vector fields on the ellipse and in \cref{fig:ellipse} (b) a comparison of the influence of geometry orders $\kp$ of the discrete sphere grid on the tangential errors of tensor field solutions is \EN{visualised}{visualized}.

We find that in both settings the convergence orders show the same \EN{behaviour}{behavior} as for vector fields, while the absolute error values are slightly larger in the tensor case.

\section*{\EN{Acknowledgements}{Acknowledgments}}
We thank Thomas Jankuhn for fruitful discussions.
This work was partially supported by the German Research Foundation (DFG), Research Unit Vector- and Tensor-Valued Surface PDEs (FOR 3013). We further acknowledge computing resources provided by ZIH at TU Dresden.

\bibliographystyle{plainnat}
\bibliography{references}

\end{document}